\documentclass[a4paper,12pt,headsepline,cleardoubleempty,reqno,bibtotoc,tablecaptionabove]{amsart}

\usepackage[utf8]{inputenc}
\author{Roland Donninger}
\address{Universität Wien, Fakultät für Mathematik,
  Oskar-Morgenstern-Platz 1, 1090 Vienna, Austria}
\email{roland.donninger@univie.ac.at}
\author{David Wallauch}
\address{Universität Wien, Fakultät für Mathematik,
  Oskar-Morgenstern-Platz 1, 1090 Vienna, Austria}
\email{david.wallauch@univie.ac.at}
\thanks{This work was supported by the Austrian Science Fund FWF,
  Projects P 30076: ``Self-similar blowup in dispersive wave equations''
  and P 34560: ``Stable blowup in supercritical wave equations''.}
\title{Optimal blowup stability for supercritical wave maps}
\usepackage{amsmath,amsfonts,amssymb,amsthm,empheq}
\usepackage{xcolor}
\usepackage{hyperref}
\usepackage{array}
\usepackage{fullpage}
\usepackage{graphicx}

\numberwithin{equation}{section}

\newcommand{\C}{\mathbb{C}}

\newcommand{\R}{\mathbb{R}}
\newtheorem{cor}{Corollary}[section]

\newtheorem{thm}{Theorem}[section]
\newtheorem{prop}{Proposition}[section]
\newtheorem{lem}{Lemma}[section]
\theoremstyle{remark}
\newtheorem{rem}{Remark}[section]
\theoremstyle{definition}
\newtheorem{defi}{Definition}[section]

\renewcommand{\O}{\mathcal{O}}
\newcommand{\N}{\textup{\textbf{N}}}

\newcommand{\rg}{\textup{\textbf{rg}}}
\renewcommand{\ker}{\textup{\textbf{ker}}}
\newcommand{\Span}{\textup{\textbf{span}}}

\newcommand{\Cf}{\textup{\textbf{C}}}

\newcommand{\X}{\mathcal{X}}

\newcommand{\Nf}{\textup{\textbf{N}}}
\newcommand{\Lf}{\textup{\textbf{L}}}

\newcommand{\Af}{\textup{\textbf{A}}}
\newcommand{\hfh}{\textup{\textbf{h}}}

\newcommand{\K}{\textup{\textbf{K}}}
\newcommand{\Sf}{\textup{\textbf{S}}}
\newcommand{\I}{\textup{\textbf{I}}}
\newcommand{\Uf}{\textup{\textbf{U}}}
\newcommand{\uf}{\textup{\textbf{u}}}
\newcommand{\vf}{\textup{\textbf{v}}}
\newcommand{\Rf}{\textup{\textbf{R}}}
\newcommand{\gf}{\textup{\textbf{g}}}

\newcommand{\ind}{\textup{\textbf{1}}}
\newcommand{\ff}{\textup{\textbf{f}}}
\newcommand{\Pf}{\textup{\textbf{P}}}
\newcommand{\Ef}{\textup{\textbf{E}}}

\renewcommand{\Re}{\operatorname{Re}}
\renewcommand{\Im}{\operatorname{Im}}
\newcommand{\B}{\mathbb{B}}
\newcommand{\Bz}{\mathbb{B}_{1+\delta}^4}
\newcommand{\hu}{\widehat{u}}

\begin{document}
      \begin{abstract}
        We study corotational wave maps from $(1+4)$-dimensional
        Minkowski space into the $4$-sphere. We prove the stability of an
        explicitly known self-similar wave map under
        perturbations that are small in the critical Sobolev space. 
      \end{abstract}
\maketitle
\section{Introduction}

\noindent Let $\Omega\subset \mathbb R\times\mathbb R^d$ be open
and consider a smooth map $U: \Omega\to \mathbb S^d\subset \mathbb
R^{d+1}$.
We number the slots of $U$ from $0$ to $d$ and write
$\partial^0=-\partial_0$ as well as $\partial^j=\partial_j$ for $j\in
\{1,2,\dots,d\}$.
Then $U$ 
is called a \emph{wave map} if it satisfies
\begin{equation}
  \label{wavemaps}\partial^\mu\partial_\mu U+(\partial^\mu U\cdot \partial_\mu
  U)U=0
  \end{equation}
on $\Omega$, where we employ Einstein's summation convention with Greek letters
running from $0$ to $d$ and $\cdot$ denotes the Euclidean inner
product on $\mathbb R^{d+1}$. Wave maps are fundamental objects in geometric
analysis and they generalize both the wave equation to maps with
values in manifolds 
and harmonic maps to Lorentzian geometry.

The wave maps equation is a hyperbolic partial differential equation
and in order to construct solutions, it is natural to study the \emph{Cauchy problem}, i.e., one prescribes
$U(0,.): \mathbb R^d\to \mathbb S^d\subset \mathbb R^{d+1}$ and $\partial_0 U(0,.):
\mathbb R^d\to \mathbb R^{d+1}$ with $U(0,.)\cdot \partial_0 U(0,.)=0$
as \emph{initial data}. The goal is then to
prove the existence of a unique wave map that satisfies the initial
data.
For some initial data it will be possible to construct a wave map on all of $\mathbb
R\times \mathbb R^d$
but in general,
the maximal domain of existence will be a proper subset of $\mathbb R\times
\mathbb R^d$. 
For $d\geq 3$ this is evidenced
by the explicit one-parameter family of \emph{self-similar solutions}
$U_*^T(t,x)=F_*(\frac{x}{T-t})$ for $T>0$, where $F_*: \R^d\to \mathbb
S^d\subset \R^{d+1}$ is given by
\[ F_*(\xi):=
  \frac{1}{d-2+|\xi|^2}\begin{pmatrix}
    2\sqrt{d-2}\,\xi \\
    d-2-|\xi|^2
  \end{pmatrix}
  =
  \begin{pmatrix}
    \sin(f_*(\xi))\frac{\xi}{|\xi|} \\
    \cos(f_*(\xi))
  \end{pmatrix} \]
with
  \[ f_*(\xi):=2\arctan\left
            (\frac{|\xi|}{\sqrt{d-2}}\right ).
\]
The function
$U_*^T$, found in \cite{Sha88, TurSpe90, BizBie15}, is an explicit example of a wave map that starts from smooth
initial data but becomes singular after a finite
time $T$ in the sense that the gradient blows up.
With the explicit blowup solution $U_*^T$ at hand, a natural question
arises: How typical is this type of singularity formation? Is the
solution $U_*^T$ an ``exceptional object'' or does there exist a larger
class of solutions that behave like $U_*^T$?
To approach this question, it is necessary to study
perturbations of $U_*^T$. 
In this paper, for the case $d=4$, we prove that
the blowup described by $U_*^T$ is stable under perturbations that are
small in the weakest possible Sobolev norm. In order to state our
result precisely, we define the domain
$\Omega_{T}\subset \mathbb R\times \mathbb R^4$ for
$T>0$ by
\[ \Omega_{T}:=\left ([0, \infty)\times \mathbb
    R^4\right )\setminus
  \left \{(t,x)\in [T,\infty)\times \mathbb R^4: |x|\leq t-T\right
  \}. \]
Furthermore, for $R>0$ and
$x_0\in\R^d$, we set $\B^d_R(x_0):=\{x\in \R^d: |x-x_0|<R\}$ and abbreviate $\B_R^d:=\B_R^d(0)$.
\begin{thm}
  \label{maintheorem} Let $d=4$.
  There exist constants $\delta_0,M>0$ such that the following
  holds. Let $F: \mathbb R^4\to \mathbb S^4\subset\mathbb R^5$ and $G:
  \mathbb R^4\to \mathbb R^5$ be given by
  \[ F(x)=
    \begin{pmatrix}
      \sin(|x|f(x))\frac{x}{|x|} \\
      \cos(|x|f(x))
    \end{pmatrix},\qquad
    G(x)=
    \begin{pmatrix}
      \cos(|x|f(x))g(x)x \\
      -\sin(|x|f(x))|x|g(x)
    \end{pmatrix}
  \]
  for smooth, radial functions $f: \mathbb R^4\to
  \left[-\frac{3}{2},\frac{3}{2}\right]$ and $g: \mathbb R^4\to\mathbb R$. Assume
  further that $\delta\in (0,\delta_0)$ and
  \[ \|(F,G)-(U_*^1(0,.), \partial_0
    U_*^1(0,.))\|_{H^2\times H^1(\B^4_{1+\delta})}\leq \frac{\delta}{M}. \]
  Then there exists a $T\in [1-\delta,1+\delta]$ and a unique smooth wave map
  $U: \Omega_{T}\to \mathbb S^4\subset\mathbb R^5$
 that satisfies $U(0,x)=F(x)$ and $\partial_0 U(0,x)=G(x)$ for all $x\in
 \mathbb R^4$. Furthermore, in the backward lightcone of the point $(T,0)$, we
 have the weighted Strichartz estimates
 \[ \int_0^T \left \||.|^{-\frac56}\left (U(t,.)-U_*^T(t,.)\right )\right \|_{L^{12}(\mathbb
     B^4_{T-t})}^2dt\leq M\delta^2 \]
 and
 \[ \int_0^T \left \||.|^{-\frac12}\left (\partial_j
       U(t,.)-\partial_j U_*^T(t,.)\right )\right \|_{L^4(\mathbb
     B^4_{T-t})}^2dt\leq M\delta^2 \]
 for $j\in \{1,2,3,4\}$.
\end{thm}

\subsection{Discussion} Some remarks are in order.

\subsubsection{Stability of blowup} Note that
  \[ U_*^T(t,0)=
    \begin{pmatrix}
      0 \\ 1
    \end{pmatrix}
  \]
  for all $t\in [0,T)$ and thus, by scaling,
  \[ \left \||.|^{-\frac56}\left (U_*^T(t,.)-
      U_*^T(t,0)\right )
      \right \|_{L^{12}(\B^4_{T-t})}\simeq
      (T-t)^{-\frac12}, \]
    which implies that
    \[ \int_0^T \left \||.|^{-\frac56}\left(U_*^T(t,.)-U_*^T(t,0)\right)\right \|_{L^{12}(\mathbb
        B^4_{T-t})}^2 dt\simeq \int_0^T (T-t)^{-1}dt=\infty. \]
    Analogously, we find
    \[ \int_0^T \left \||.|^{-\frac12}\partial_j
          U_*^T(t,.)\right \|_{L^{4}(\mathbb
        B^4_{T-t})}^2 dt\simeq \int_0^T (T-t)^{-1}dt=\infty. \]
 Consequently, our estimates show that the self-similar blowup
 solution $U_*^T$
  is \emph{asymptotically stable} in a Strichartz sense in the
  backward lightcone of the singularity. In other words, the solution
  we construct
  can be written as
  \[
    U(t,x)=U_*^T(t,x)+\underbrace{U(t,x)-U_*^T(t,x)}_{\mbox{small}}, \]
  i.e., $U$ blows up at time $t=T$ and equals $U_*^T$ up to an error that is small in a
  suitable Strichartz space.

\subsubsection{Optimality} The wave maps equation has the scaling symmetry $U(t,x)\mapsto
  U(\frac{t}{\lambda}, \frac{x}{\lambda})$ for $\lambda>0$ and the
  corresponding scaling-invariant Sobolev space is $\dot
  H^{\frac{d}{2}}\times \dot H^{\frac{d}{2}-1}(\R^d)$.
As the wave maps equation is ill-posed
  below scaling \cite{ShaTah94}, our smallness condition on the data is optimal in the
  sense that
  the number of derivatives cannot be lowered any further.

  \subsubsection{Relation to small-data Cauchy theory}
  Our result may be viewed as the analogue of an optimal small-data global
  existence result, \emph{but relative to a blowup solution}: If the
  data are close to the given blowup solution, the evolution exists
  for as long as it possibly can and behaves asymptotically like the
  given blowup.
  
\subsubsection{Symmetry} The initial data we prescribe have a very special form. Such
  data are called \emph{corotational} and the point is that the wave
  maps flow preserves this form. Accordingly, our Strichartz estimates
  are not translation-invariant and thus inherently
  corotational. Nonetheless, we expect a similar result
  to hold for general data but establishing this is a formidable
  endeavor that has to be postponed to the future.  
  
  \subsubsection{Maximal domain of existence}
  We show the existence of the solution in the domain $\Omega_T$,
  which includes all of $[0,\infty)\times \R^4$ except for the
  \emph{forward} lightcone of the blowup point $(T,0)$. The latter is
  precisely the spacetime region that is causally influenced by the
  singularity. It is an intriguing open question whether one can
  extend the solution even further in a meaningful way.

  \subsubsection{Other dimensions}
  There is nothing special about $d=4$ except that it is the smallest
  dimension $d\geq 3$ with a critical Sobolev space of
  integer order. This fact simplifies the
  analysis but the generalization to other dimensions $d\geq 3$ is mainly
  technical and will be treated elsewhere.
  The behavior for $d\leq 2$, on the other hand, is an entirely different subject that is
  extensively treated in the literature, see below.
  
  \subsubsection{Broader context}
Finally, we would like to emphasize that our Theorem \ref{maintheorem} is
  a large-data result for an energy-supercritical wave equation. In addition,
  it is the first result on \emph{optimal} blowup stability for a geometric wave
  equation.

\subsection{Related results}
As the prototypical example of a geometric wave equation, the wave
maps equation received a lot of attention during the last 30 years and
it is impossible to review or even mention all of the available results. We therefore restrict
ourselves to a brief overview of some of the most important recent contributions  that are closely related to the
present paper.

The local Cauchy theory at low regularity for corotational wave maps
is developed in \cite{ShaTah94} and the general case is studied in
\cite{KlaMac95, KlaSel97, Tao00, MasPla12}.  The global Cauchy problem for small
data is of course most delicate when one measures smallness
in a scaling-invariant space. This challenging problem was part of a big program in the
1990s and the beginning 2000s and was finally resolved in
\cite{Tat98, Tat01, Tao01a, Tao01b, KlaRod01, ShaStr02, Kri03,
  NahSteUhl03, Kri04, Tat05, CanHer18}.

For the global Cauchy problem with large data, the most powerful
results are available in the energy-critical case $d=2$, where energy
conservation yields invaluable global information. On the other hand,
the blowup mechanism in $d=2$ is more complicated and takes place via dynamical
rescaling of a soliton (a harmonic map). As a consequence, already the construction of solutions
that blow up in finite time is highly nontrivial and was first
achieved in \cite{KriSchTat08, RodSte10, RapRod12}, motivated by
numerical evidence \cite{BizChmTab01}, see also \cite{CanKri15}. Stability results for blowup
are established in \cite{RapRod12, KriMia20}.
The question of global existence for large data has to be addressed in
view of the fact that finite-time blowup is possible. Since the blowup
takes place via shrinking of a harmonic map, the ``first'' harmonic
map provides a natural threshold for global existence.
 This is expressed in the \emph{threshold
conjecture} 
\cite{SteTat10a, SteTat10b, KriSch12, LawOh16, ChiKriLuh18}, see also
the series of unpublished preprints \cite{Tao08} and the earlier \cite{Str03, CotKenMer08} for the corotational
setting. The most recent work on energy-critical wave maps focuses on
the precise asymptotic behavior and the \emph{soliton resolution
  conjecture}
\cite{CotKenLawSch15a, CotKenLawSch15b, Cot15, Gri17, JiaKen17,
  JenLaw18, DuyJiaKenMer18}.

The present paper is concerned with the energy-supercritical case
$d\geq 3$
where the conservation of energy is useless for
the study of the Cauchy problem and the
understanding
of large-data evolutions is still comparatively poor. 
The existence of self-similar blowup for $d\geq 3$ is established in
\cite{Sha88, TurSpe90, CazShaTah98, Biz00, BizBie15}. Motivated by
numerical evidence \cite{BizChmTab00}, the stability of self-similar
blowup under perturbations that are small in Sobolev spaces of
sufficiently high order is proved in \cite{DonSchAic12, Don11,
  CosDonXia16, CosDonGlo17, ChaDonGlo17, DonGlo19, BieDonSch21}.
Furthermore, in dimensions $d\geq 7$, there exists another blowup mechanism that is
more akin to the energy-critical case \cite{GhoIbrNgu18}, see also
\cite{DodLaw15, ChiKri17} for other large-data results.

The stability of
blowup in
the critical Sobolev space, which is the content of the present
paper,
constitutes a vast conceptional improvement over the
existing results because the critical Sobolev space is the only
\emph{distinguished} space for the study of an energy-supercritical
problem. So far, blowup stability in a critical space is only known
for the much simpler energy-critical nonlinear wave
equation \cite{Don17, DonRao20}, see also \cite{Bri20} for an
extension to randomized perturbations.

\subsection{Outline of the proof}
We give a nontechnical outline of the key steps of the proof.

\subsubsection{Preliminary reductions}
In Theorem \ref{maintheorem}
the initial data are
  \emph{corotational} and this symmetry is preserved by the wave maps
  flow. This means that the solution $U$ is of the form
  \begin{equation}
\label{eq:cor}
    U(t,x)=
    \begin{pmatrix}
      \sin(|x|u(t,x))\frac{x}{|x|} \\
      \cos(|x|u(t,x))
    \end{pmatrix}
  \end{equation}
  for a smooth function $u: [0,T)\times \R^4\to \R$ such that $u(t,.)$ is
  radial for each $t\in [0,T)$.
Indeed, if one plugs the ansatz Eq.~\eqref{eq:cor} into the wave maps
equation \eqref{wavemaps}, one ends up with the single semilinear wave
equation
\begin{align}\label{startingeq2}
\left(\partial_t^2-\partial_r^2-\frac{5}{r}\partial_r\right)\widetilde
  u(t,r) +\frac{3 \sin(2r\widetilde u(t,r))- 6 r \widetilde u (t,r)}{2 r^3}=0
\end{align}
for $r>0$, where $u(t,x)=\widetilde u(t,|x|)$. Interestingly, this is a radial wave equation in
$6$ rather than $4$ spatial dimensions. Consequently, it is natural to
regard $u$ as a smooth function on $[0,T)\times \R^6$ and then,
Eq.~\eqref{startingeq2} can be written as
\begin{equation}
  \label{eq:startingeq3}
  (\partial_t^2-\Delta_x)u(t,x)+\frac{3\sin(2|x|u(t,x))-6|x|u(t,x)}{2|x|^3}=0
\end{equation}
for $x\in \R^6\setminus \{0\}$.
Note that the apparent singularity at $x=0$ is in fact removable and
the nonlinearity is generated by a smooth function.

\subsubsection{Control of the evolution away from the origin}

 The bulk of the present work is concerned with the analysis of
 Eq.~\eqref{startingeq2}, which is essentially a cubic wave equation. 
It will be vital for us to exploit the finite speed of propagation of
Eq.~\eqref{startingeq2}, so that we can restrict ourselves to
lightcones.
To this end, we define the domain $\Gamma_{t_0}^{t_1}(r_0) $, for $0\leq t_0<t_1$ and $r_0\in \R_+:=(0,\infty)$, by 
\[
\Gamma_{t_0}^{t_1}(r_0):=\{(t,r):t\in[t_0,t_1), r\in[r_0-(t_1-t),r_0+(t_1-t)]\}.
\]
Our first step will be the proof of an existence result of
Eq. \eqref{eq:startingeq3} outside of the forward lightcone of the
origin.
First, we show that the reduced wave maps equation \eqref{startingeq2}
is well-posed on any domain $\Gamma_{t_0}^{t_1}(r_0)$ that does not
touch the origin.
\begin{thm}\label{away from the center}
  Consider $\Gamma_{t_0}^{t_1}(r_0)$ with $r_0-(t_1-t_0)>0$ and let $\widetilde
  f,\widetilde g:
  [r_0-(t_1-t_0), r_0+(t_1-t_0)]\to \R$ be smooth. Then there exists a unique
  smooth solution $\widetilde u: \Gamma_{t_0}^{t_1}(r_0)\to \R$ of
  Eq.~\eqref{startingeq2} that satisfies $\widetilde
  u(t_0, .)=\widetilde f$ and
  $\partial_0 \widetilde u(t_0, .)=\widetilde g$.
\end{thm}
As an immediate consequence, we obtain the desired existence result.
\begin{cor}
  \label{cor:awayorigin}
Let $f,g\in C^\infty(\R^6)$ be radial. Then there exists a unique smooth
solution
\[ u: \{(t,x)\in [0,\infty)\times \R^6: |x|>t\}\to\R \]
of Eq.~\eqref{eq:startingeq3} that satisfies $u(0,.)=f$ and
$\partial_0 u(0,.)=g$.
\end{cor}
The key ingredient in proving this result will be the fact that away
from the center we are effectively dealing with a
one-dimensional wave equation with a bounded nonlinearity.

\subsubsection{Control of the evolution near the origin}

The most difficult part will be the control of the evolution in the
backward lightcone of the singularity. For the following, we set
\[ u_*^T(t,x):=\frac{1}{|x|}f_*\left (\frac{x}{T-t}\right
  )=\frac{2}{|x|}\arctan\left (\frac{|x|}{\sqrt 2\,(T-t)}\right
  ). \]
Note that $u_*^T$ solves the reduced wave maps equation
\eqref{eq:startingeq3}.

\begin{thm}\label{stability}
There exist $\delta_0, M>0$ such that the following holds.
Let $f,g\in C^\infty(\overline{\B^6_{1+\delta}})$ be radial and $\delta\in
(0,\delta_0)$ such that 
\[
\|(f,g)-(u^1_*(0,.),\partial_0 u_*^1(0,.))\|_{H^2\times H^1(\B^6_{1+\delta})} \leq\frac{\delta}{M}.
\]
Then there exists a blowup time $T\in [1-\delta,1+\delta]$ and a
unique smooth solution
\[ u: \left \{(t,x)\in [0,T)\times \R^6: |x|\leq T-t\right \} \to\R \] of
Eq.~\eqref{eq:startingeq3} satisfying $u(0,.)=f$ and $\partial_0
u(0,.)=g$ on $\overline{\B^6_T}$. Furthermore, we have the Strichartz estimates
\begin{equation}
\int_0^T \left\|u(t,.)-u^T_*(t,.)\right\|^2_{L^{12}(\B^6_{T-t})}dt\leq M\delta^2
\end{equation}
and
\begin{equation}
\int_0^T \left\|u(t,.)-u^T_*(t,.)\right\|^2_{\dot{W}^{1,4}(\B^6_{T-t})}dt \leq M\delta^2.
\end{equation}
\end{thm}
Theorem \ref{stability} is in fact the main result of our paper and so we highlight some of the key steps in proving it.
\begin{itemize}

\item Our starting
point is a well-known reformulation of the problem in similarity
coordinates $\tau=-\log(T-t)$ and $\xi=\frac{x}{T-t}$.
In these new coordinates the blowup solution $u_*^T$ is
$\tau$-independent
and Eq.~\eqref{eq:startingeq3} takes the
autonomous first-order form
\[ \Phi(\tau)=\mathbf L\Phi(\tau)+\mathbf N(\Phi(\tau)) \]
for perturbations $\Phi$ of the blowup and where $\mathbf L$ is a spatial
differential operator that arises from linearising the equation at the
$\tau$-independent blowup solution. All the nonlinear terms are
summarized in $\mathbf N(\Phi(\tau))$. We show that $\mathbf L$ generates a
semigroup $\mathbf S$ on $\mathcal{H}:= H^2\times H^1(\B^6_1)$ and
moreover, by utilizing \cite{CosDonGlo17}, we prove that $\Lf$ has precisely one
unstable eigenvalue $\lambda = 1$. This
eigenvalue does not correspond to an ``actual'' instability of the
blowup solution $u^T$ but is a mere consequence of the time
translation symmetry of the wave maps equation.
We further establish that the associated spectral projection $\Pf$ is of rank $1$ and by standard semigroup theory we obtain the bound
\begin{align*}
\|\Sf(\tau) (\I-\Pf)\|_{\mathcal{H}}\leq C_\varepsilon e^{\varepsilon\tau}
\end{align*}
for any $\varepsilon>0$.
 \item The crucial observation, then, is the fact that
$\mathbf S(\tau)(\I-\Pf)$ satisfies Strichartz estimates.
To prove these estimates we first conduct an asymptotic construction of the resolvent of $\Lf$.
The spectral equation $(\lambda-\Lf)\uf=\ff$ 
with $\uf=(u_1,u_2)$ and $\ff=(f_1,f_2)$ turns out to be equivalent to the second order ODE
\begin{equation}\label{eq:outline}
  \begin{split}
    (\rho^2-1)&u_1''(\rho)+\left(2(\lambda+2)\rho-\frac{5}{\rho}\right)u_1'(\rho)+(\lambda+2)(\lambda+1)u_1
    \\
    &-\frac{48}{(\rho^2+2)^2}u_1(\rho)=F_\lambda(\rho)
    \end{split}
\end{equation}
with  $F_\lambda(\rho)=f_2(\rho)+(\lambda+2)f_1(\rho)+\rho f_1'(\rho)$ and $\rho \in (0,1)$.
To study this equation, it is advantageous to get rid of the first
order derivative by using an appropriate transformation of the
independent variable $u$. This transforms Eq.~\eqref{eq:outline} with
$F_\lambda=0$ into
\begin{equation}\label{eq:outline2}
v''(\rho)+\frac{\rho^2(10+12\lambda-4\lambda^2)-15}{4\rho^2(1-\rho^2)^2}v(\rho)=\frac{48}{(\rho^2+2)^2(1-\rho^2)}v(\rho).
\end{equation}
We then construct fundamental systems near either pole separately.
One of the main tools to do so is the diffeomorphism 
\[
\varphi(\rho):=\frac{1}{2}\log\left(\frac{1+\rho}{1-\rho}\right)
\]
which, by means of a Liouville-Green transform, turns the equation
into Bessel form.
Near $0$ we therefore obtain a fundamental system of
Eq.~\eqref{eq:outline} which essentially consists of perturbed Bessel
functions and the error is controlled by Volterra iterations. 
Unfortunately, near the endpoint $1$, we cannot use Hankel functions
straight away. This is because the standard Hankel expansion is not
good enough for our purpose and we need better control of the
perturbative parts. Hence, near 1, we first obtain a fundamental
system for
\begin{equation*}
w''(\rho)+\frac{\rho^2(10+12\lambda-4\lambda^2)-15}{4\rho^2(1-\rho^2)^2}w(\rho)=0
\end{equation*}
without changing variables and based on this, we solve
Eq.~\eqref{eq:outline2} perturbatively. Both steps are again accomplished by Volterra iterations.
Finally, to glue the solutions together we make use of the global
theory of the Bessel equation to obtain sufficient control on the connection coefficients.
As a result of this intricate ODE analysis, we obtain a solution to
Eq.~\eqref{eq:outline} and hence a suitable representation of the
resolvent $(\lambda-\Lf)^{-1}$.

\item
Next, we use the Laplace
representation
\[ [\Sf(\tau)(\I-\Pf)\ff]=\frac{1}{2\pi
    i}\lim_{N\to\infty}\int_{\epsilon-iN}^{\epsilon+iN}e^{\lambda\tau}(\lambda-\Lf)^{-1}(\I-\Pf)\ff\,d\lambda
\]
and our resolvent construction to push the contour of integration onto the imaginary axis and
prove Strichartz estimates by a delicate analysis of the occurring oscillatory integrals.
To be more precise, we prove these
estimates for the difference of the linearised evolution to the free
evolution. This is technically much simpler and equivalent because for
the free evolution the desired Strichartz estimates follow immediately
by a simple scaling argument. Readers familiar with Strichartz theory
for wave equations with potentials might wonder at this point because the comparison to the free evolution
usually only yields short-time Strichartz estimates. However, since our
problem is in fact short-time (we are working in a backward
lightcone), it is not surprising that this
strategy works. Only in the rescaled variable $\tau$ our
Strichartz estimates are global. 

\item
The full nonlinear problem is then treated
perturbatively and solved in the corresponding Strichartz space by a
fixed point argument.
The instability caused by the time translation symmetry is treated by
a Lyapunov-Perron-type argument. That is to say, we first modify the equation
to suppress the instability and in a second step we show that the
modification vanishes provided one chooses the blowup time correctly.
Finally, by more or less following the standard
procedure, we upgrade the regularity to
$C^\infty$. 
\end{itemize}

\section{Local wellposedness away from the center}
In this section we establish Theorem \ref{away from the center} and
Corollary \ref{cor:awayorigin}.
The first step is to look at the problem locally in the domains
$\Gamma_{t_0}^{t_1}(r_0)$. As usual, we start with a weaker notion of
solutions and upgrade the regularity afterwards. Thus, we use
d'Alembert's formula and consider the Duhamel form of Eq.~\eqref{startingeq2}.
\begin{defi}
  Consider $\Gamma_{t_0}^{t_1}(r_0)$ with $r_0-(t_1-t_0)>0$ and for $t\in [t_0,t_1)$,
  set $I_t:=(r_0-(t_1-t),r_0+(t_1-t))\subset \R_+$.
  Furthermore, for $\widetilde t\in\R$, we set
  \[ \Gamma_{t_0}^{t_1}(r_0, \widetilde t):=\Gamma_{t_0}^{t_1}(r_0)\cap \{(t,r)\in \R^2: t<
\widetilde t\}. \] 
We say that $u: \Gamma_{t_0}^{t_1}(r_0, \widetilde t)\to \R$ is an $H^k$ solution of
Eq.~\eqref{startingeq2} in $\Gamma_{t_0}^{t_1}(r_0, \widetilde t)$ if the following
three properties hold:
\begin{itemize}
\item $u$ is continuous,
\\
\item $u(t,.)\in H^k(I_t)$ for all $t\in [t_0,t_1)\cap[t_0,\widetilde t)$,
\\
\item $u$ satisfies the integral equation
\begin{align}\label{solutionsin1d}
u(t,r)=\frac{1}{2}\left(u(t_0,r+t-t_0)+u(t_0,r-t+t_0)\right)&+\frac{1}{2}\int_{r-t+t_0}^{r+t-t_0}\partial_0 u(t_0,y) dy\nonumber
\\
&+\frac{1}{2}\int_{t_0}^t\left(\int_{r+s-t}^{r+t-s} \widetilde{N}(u(s,.))(y)
dy\right) ds
\end{align}
for all $(t,r)\in \Gamma_{t_0}^{t_1}(r_0, \widetilde t)$, where
 \[
\widetilde{N}(u(t,.))(r):=\frac{5}{r}\partial_r u(t,r) -\frac{3 \sin(2ru(t,r))- 6 r u (t,r)}{2 r^3}.
\]
\end{itemize}

\end{defi}
Note that if $u$ is $C^2$ then $u$ is a classical solution.
In the following, we sometimes
abuse notation and do not distinguish between a radial function $f$
and its representative $\widetilde f$, given by
$f(x)=\widetilde f(|x|)$. With this identification we have the useful equivalences
\begin{align*}
\|f\|_{L^2(\B_R^d)}^2\simeq& \int_{0}^R |f(r)|^2r^{d-1} dr
\\
\|f\|_{\dot{H}^1(\B_R^d)}^2\simeq& \int_{0}^R |f'(r)|^2r^{d-1} dr
\\
\|f\|_{\dot{H}^2(\B_R^d)}^2\simeq& \int_{0}^R \left|f''(r)+\frac{d-1}{r}f'(r) \right|^2r^{d-1}dr.
\end{align*}

The decisive properties of $\widetilde N$ are given in the following
lemma.

\begin{lem}
  \label{lem:Ntw}
  Let $k\in \mathbb N_0$ and $0<a<b$. Then we have the bound
  \[ \left \|\widetilde N(f)-\widetilde N(g)\right \|_{H^k(I)}\lesssim
    \|f-g\|_{\dot
      H^{k+1}(I)}+\left (1+\|f\|_{H^k(I)}^k+\|g\|_{H^k(I)}^k\right )\|f-g\|_{H^k(I)} \]
  for all $f,g\in H^{k+1}(I)$ and all intervals $I\subset [a,b]$.
\end{lem}

\begin{proof}
  Let $I\subset [a,b]$ and $f,g\in C(I)\cap L^2(I)$. Since
  $\frac{1}{r}\simeq 1$ for all $r\in [a,b]$, we have the pointwise
  bound
  \[ |\widetilde N(f)(r)-\widetilde N(g)(r)|\lesssim
    |f'(r)-g'(r)|+|f(r)-g(r)| \]
  for all $r\in I$.
  Thus, the case $k=0$ follows by a density argument. For $k\geq 1$ we
  consider $f\in C^k(I)\cap H^k(I)$ and differentiate $\widetilde
  N(f)$ $k$ times. By the chain rule, this produces powers of
  derivatives of $f$. Thus,
  the stated estimate is a consequence of the algebra property of
  $H^1(I)$, elementary pointwise algebraic identities, and a density argument.
\end{proof}

\begin{rem}
Note that in the case $k=0$, $\widetilde N$ satisfies a \emph{global}
Lipschitz bound. This is the crucial property.
\end{rem}

Now we turn to the analysis of solutions of Eq.~\eqref{startingeq2}. The first result deals with the dependence on
the initial data and uniqueness.
\begin{lem}\label{continuous dependence}
  Consider $\Gamma_{t_0}^{t_1}(r_0)$ with
  $r_0-(t_1-t_0)>0$ and, as before, let $I_t=(r_0-(t_1-t),
  r_0+(t_1-t))$ and $\widetilde t\in (t_0, t_1)$. Then we
  have the bound
  \begin{align*}
    \|&u(t,.)-\widetilde u(t,.)\|_{H^1(I_t)} \lesssim \|(u(t_0,.),\partial_0
    u(t_0,.))-(\widetilde u(t_0,.),\partial_0\widetilde u(t_0,.))\|_{H^1\times
   L^2(I_{t_0})}
 \end{align*}
for all $t\in [t_0,\widetilde t]$ and all $H^1$ solutions $u$ and $\widetilde u$ of
Eq.~\eqref{startingeq2} in
$\Gamma_{t_0}^{t_1}(r_0, \widetilde t)$.
\end{lem}
\begin{proof}
To prove this result we first claim that the estimate
\begin{align}\label{firstestimate}
\left\|\int_{r+s-t}^{r+t-s}
f(y)
dy\right\|_{H_r^1(I_t)}
\lesssim \left\|f\right\|_{L^2(I_s)}
\end{align}
holds true for all $f\in H^1(I_{t_0})$ and $t_0\leq s\leq t\leq t_1$.
To see this, it suffices to note that
\begin{align*}
  \left |\int_{r+s-t}^{r+t-s} f(y)dy\right |^2&\leq \left |\int_{r+s-t}^{r+t-s}
  dy\right |\int_{r+s-t}^{r+t-s}|f(y)|^2dy\leq
                                                2(t-s)\int_{r_0-t_1+s}^{r_0+t_1-s}|f(y)|^2dy
  \\
  &\lesssim (t-s)\|f\|_{L^2(I_s)}^2
\end{align*}
for all $r\in I_t$, 
along with
\begin{align*}
\|f(r+t-s)\|_{L^2_r(I_t)}+\left\|f(r-t+s)\right\|_{L^2_r(I_t)}\lesssim& \left\|
f
\right\|_{L^2(r_0-(t_1-t),r_0+t_1-s)}
\\
&+\|f
\|_{L^2(r_0-(t_1-s),r_0+t_1-t)}
\\
\lesssim& \left\|
f
\right\|_{L^2(I_s)}.
\end{align*}
Let now $u$ and $\widetilde{u}$ be two $H^1$ solutions. Then we have
\begin{align*}
\left\|\int_{r+s-t}^{r+t-s} \widetilde{N}(u(s,.))(y)-\widetilde{N}(\widetilde{u}(s,.))(y)
dy\right\|_{H^1_r(I_t)}\lesssim \left\|\widetilde{N}(u(s,.))-\widetilde{N}(\widetilde{u}(s,.))
\right\|_{L^2(I_s)}.
\end{align*}
Now, because of the condition $r_0-(t_1-t_0)>0$, Lemma \ref{lem:Ntw}
applies and we obtain
\begin{align*}
\|\widetilde{N}(u(s,.))-\widetilde{N}(\widetilde{u}(s,.))
\|_{L^2(I_s)}\lesssim \|u(s,.)-\widetilde u(s,.)
\|_{H^1(I_s)}.
\end{align*}
Plugging in the definitions of the respective solutions then yields
\begin{align*}
\|u(t,.)-\widetilde{u}(t,.)\|_{H^1(I_t)}\lesssim& \|(u(t_0,.),\partial_0
    u(t_0,.))-(\widetilde u(t_0,.),\partial_0\widetilde u(t_0,.))\|_{H^1\times
    L^2(I_{t_0})}
    \\
    &+\int_{t_0}^t\|u(s,.)-\widetilde{u}(s,.)\|_{H^1(I_s)} ds
\end{align*}
for all $t\in[t_0, \widetilde t)$ and therefore, the claim follows from Gronwall's inequality.
\end{proof}

\begin{rem}
  As an immediate consequence of Lemma \ref{continuous dependence} we
  obtain the uniqueness of $H^1$ solutions in (truncated) cones
  $\Gamma_{t_0}^{t_1}(r_0, \widetilde t)$ with $r_0-(t_1-t_0)>0$.
\end{rem}

Next, we turn to
  the existence and prove a standard local existence result with a
  blowup criterion.

  \begin{lem}
    \label{lem:locexk}
    Let $k\in \mathbb N_0$, $M>0$, and $0<a<b$. Then there exists
    a $t'>0$ such that the following holds. Let
    $I_t:=(r_0-(t_1-t), r_0+(t_1-t))$ and consider
    $\Gamma_{t_0}^{t_1}(r_0)$ with $I_{t_0}\subset
    [a,b]$. Furthermore, let $(f,g)\in H^{k+1}\times
    H^k(I_{t_0})$ with $\|(f,g)\|_{H^{k+1}\times H^k(I_{t_0})}\leq M$. Then there exists an $H^k$ solution $u$ of
    Eq.~\eqref{startingeq2} in the (possibly truncated) cone
    $\Gamma_{t_0}^{t_1}(r_0, t_0+t')$
    that satisfies $(u(t_0,.), \partial_0 u(t_0,.))=(f,g)$. 
    If $u$ is inextendible beyond time $\widetilde t\in (t_0,t_1)$ as
    an $H^k$ solution of Eq.~\eqref{startingeq2}, we
    have
    \[ \sup_{t\in (t_0,\widetilde t)}\|u(t,.)\|_{H^k(I_t)}=\infty. \]
  \end{lem}

  \begin{proof}
    For $t'> 0$ and $u \in C^\infty(\overline{\Gamma_{t_0}^{t_1}(r_0, t_0+t')})$ we set
    \[ \|u\|_{X(t')}:=\sup_{t\in [t_0,t_0+t']}\|u(t,
      .)\|_{H^k(I_t)} \]
    and denote by $X(t')$ the completion of
    $C^\infty(\overline{\Gamma_{t_0}^{t_1}(r_0, t_0+t')})$ with respect to
    $\|.\|_{X(t')}$. Furthermore, for $u\in X(t')$ and $(t, r)\in
    \Gamma_{t_0}^{t_1}(r_0, t_0+t')$, we define the
    mapping
    \begin{align*}K_{f,g}(u)(t,r):=\frac{1}{2}\big(f(r+t-t_0)&+f(r-t+t_0)\big)+\frac{1}{2}\int_{r-t+t_0}^{r+t-t_0}g(y) dy\nonumber
\\
&+\frac{1}{2}\int_{t_0}^t\left(\int_{r+s-t}^{r+t-s} \widetilde{N}(u(s,.))(y)
dy\right) ds.
    \end{align*}
    From Lemma \ref{lem:Ntw} and the proof of Lemma \ref{continuous
      dependence} it follows that $K_{f,g}$ maps $X(t')$ to $X(t')$
    and we have the bound
    \[ \|K_{f,g}(u)\|_{X(t')}\leq
      CM+Ct'\left (1+\|u\|_{X(t')}^k\right )\|u\|_{X(t')} \]
    for some constant $C\geq 1$.
    We set $Y(t'):=\{u\in X(t'): \|u\|_{X(t')}\leq 2CM\}$ and upon
    choosing $t'$ small enough, we infer that $K_{f,g}$ maps $Y(t')$
    to $Y(t')$.
    Furthermore, we have the bound
    \[ \|K_{f,g}(u)-K_{f,g}(v)\|_{X(t')}\lesssim t'\|u-v\|_{X(t')} \]
    for all $u,v\in Y(t')$ and we may choose $t'$ so small that
    $K_{f,g}$ becomes a contraction on $Y(t')$. Thus, the contraction
    mapping principle yields the existence of a fixed point $u\in
    Y(t')$ of $K_{f,g}$. Clearly, by Sobolev embedding, $u:
    \Gamma_{t_0}^{t_1}(r_0, t_0+t')\to\R$ is continuous and hence an
    $H^k$ solution of Eq.~\eqref{startingeq2} in
    $\Gamma_{t_0}^{t_1}(r_0, t_0+t')$.

    If $u$ is inextendible as an $H^k$ solution of
    Eq.~\eqref{startingeq2} beyond time $\widetilde t\in
    (t_0,t_1)$, we must have
    \[ \sup_{t\in (t_0, \widetilde t)}\|u(t, .)\|_{H^k(I_t)}=\infty \]
    because otherwise, we could use the above argument to extend the solution
    beyond time $\widetilde t$.
  \end{proof}

  Thanks to the global Lipschitz property of $\widetilde N$, we can
  extend any $H^k$ solution to the full cone $\Gamma_{t_0}^{t_1}(r_0)$.

  \begin{lem} \label{localex1}
  Consider $\Gamma_{t_0}^{t_1}(r_0)$ with $r_0-(t_1-t_0)>0$ and for
  $k\in\mathbb N_0$, let
$(f,g)\in H^{k+1}\times H^k(I_{t_0})$, where $I_t=(r_0-(t_1-t), r_0+(t_1-t))$. Then there exists an $H^k$ solution
$u$ of Eq.~\eqref{startingeq2} in  $\Gamma_{t_0}^{t_1}(r_0)$ that satisfies
$(u(t_0,.),\partial_0 u(t_0,.))=(f,g)$.
\end{lem}

\begin{proof}
  We start with the case $k=0$ and argue by contradiction.
  By Lemma \ref{lem:locexk}, there exists
a $\widetilde t> t_0$ and an $H^1$ solution $u$ of
Eq.~\eqref{startingeq2} in $\Gamma_{t_0}^{t_1}(r_0, \widetilde
t)$. Suppose that $\widetilde t< t_1$ and that $u$ is inextendible
beyond $\widetilde t$. From Lemma \ref{lem:Ntw} and the proof of Lemma \ref{continuous
  dependence} we obtain the estimate
\[ \|u(t,.)\|_{H^1(I_t)}\lesssim \|(f,g)\|_{H^1\times
    L^2(I_{t_0})}+\int_{t_0}^t \|u(s,.)\|_{H^1(I_s)}ds \]
and by Gronwall's inequality, $\|u(t,.)\|_{H^1(I_t)}\lesssim 1$ for
all $t\in [t_0,\widetilde t)$. This is a contradiction to the blowup
criterion in Lemma \ref{lem:locexk} and thus, $u$ extends to all of
$\Gamma_{t_0}^{t_1}(r_0)$.

Now suppose $u$ is an $H^k$ solution of Eq.~\eqref{startingeq2} in
$\Gamma_{t_0}^{t_1}(r_0)$ and $(f,g)\in H^{k+1}\times H^k(I_{t_0})$.
By Lemma \ref{lem:locexk}, there exists a $\widetilde t>t_0$ and an
$H^{k+1}$ solution $\widetilde u$ in $\Gamma_{t_0}^{t_1}(r_0,
\widetilde t)$ with initial data $(\widetilde u(t_0,.),
\partial_0 \widetilde u(t_0,.))=(f,g)$. By Lemma \ref{continuous
  dependence}, $\widetilde u=u$ and Lemma \ref{lem:Ntw} yields the
bound
\[ \|u(t,.)\|_{H^{k+1}(I_t)}\lesssim \|(f,g)\|_{H^{k+1}\times
    H^k(I_{t_0})}+1+\sup_{t\in [t_0,\widetilde
    t]}\|u(t,.)\|_{H^k(I_t)}^{k+1}+\int_{t_0}^t
  \|u(s,.)\|_{H^{k+1}(I_s)}ds \]
for all $t\in [t_0,\widetilde t)$.
Thus, by the same argument as above we see that $u$ extends to
all of $\Gamma_{t_0}^{t_1}(r_0)$ as an $H^{k+1}$ solution and the
claim follows inductively.
\end{proof}

We can now turn to the proof of Theorem \ref{away from the center}.

\begin{proof}[Proof of Theorem \ref{away from the center}]

Let $(f,g) \in C^\infty\times C^\infty(\overline{I_{t_0}})$. By Lemma
\ref{localex1}, we obtain an $H^1$ solution $u$ of
Eq.~\eqref{startingeq2} that satisfies $u(t,.)\in
H^k(I_t)$ for any $k\geq 1$ and all $t\in [t_0,t_1)$. By Sobolev
embedding, $u(t,.)\in C^\infty(\overline{I_t})$. To deal with the time
derivatives, observe that the right hand side of
Eq.~\eqref{solutionsin1d} consists of the term involving the initial
data, which is smooth by assumption, and the Duhamel term, which gains a degree of differentiability in $t$. Since this gain is untouched by differentiating with respect to $r$, we conclude that $\partial_r^k \partial_t^l u(t,r)$ exists for any $k,l\geq1$ and is in fact continuous on $\Gamma_{t_0}^{t_1}(r_0)$. Finally,  Lemma \ref{continuous dependence} ensures that this 
solution is indeed unique.
\end{proof}
\begin{proof}[Proof of Corollary \ref{cor:awayorigin}]
  Let $(f,g) \in C^\infty\times C^\infty(\R^6)$ and note that
  $(f,g) \in H^k\times H^{k-1}(I)$ for any interval $I\subset \R$ that
  satisfies $\overline{I}\subset \R_+$ and any $k\geq 1$. Consider now
  $\Gamma_{0}^{t_1}(r_0)$ with $r_0-t_1>0$. From Theorem \ref{away from the
    center} there exists a unique smooth solution $u$ of
  Eq.~\eqref{startingeq2} on $\Gamma_{0}^{t_1}(r_0)$ with initial data
  $u(0,.)=f$ and $\partial_0u(0,.)=g$. Furthermore, since
  $\Gamma_{0}^{t_1}(r_0)$ does not touch the origin, we see that $u(t,.)$
  corresponds to a smooth radial function for any $t\in [0,t_1)$. Let
  now $T,\varepsilon >0 $ and consider the set
  $D^T_\varepsilon=\{(t,r)\in [0,T]\times \R:r\geq t+
  \varepsilon\}$. We can cover $D^T_\varepsilon$ with cones
  of the form $\Gamma_{t_0}^{t_1}(r_0)$. Now, since each of these cones
  is a positive distance away from the origin, the above consideration
  implies the existence of a unique smooth solution in each of these
  cones.
Since the intersection of two such cones is either empty or again a
cone, by Lemma \ref{continuous dependence}, we can smoothly glue together the solutions in the individual
cones to obtain a unique solution in all of $D_\epsilon^T$.
  Since $T$ and $\varepsilon$ were
  arbitrary, the claim follows.
\end{proof}

\section{Semigroup theory and free Strichartz estimates}
Before we can analyse Eq. (\ref{startingeq2}) properly in the radial lightcone $\Gamma^T$, we first need the right choice of coordinates. For what we intend to do, the similarity coordinates given by 
	\begin{equation}\label{coordinate}
	\tau=-\log(T-t)+\log(T),\,\,\; \rho=\frac{r}{T-t}
	\end{equation}
are well suited.
Thus, we set $\psi(\tau, \rho)=Te^{-\tau}u(T-Te^{-\tau},Te^{-\tau}\rho)$ and switch to the similarity coordinates which transform Eq.~\eqref{startingeq2} into 
	\begin{align}\left(2+3\partial_\tau+\partial_\tau^2 +2\rho\partial_\tau\partial_\rho+ 4 \rho\partial_\rho-\frac{5}{\rho}\partial_\rho+(\rho^2-1)\partial_\rho^2
\right)\psi +\frac{3\sin(2\rho\psi)-6\rho\psi}{2\rho^3}=0,
	\end{align}
        where we omit the arguments of $\psi$ for brevity.
	To be able to use semigroup techniques, we define  
	\begin{equation}
	\psi_1(\tau,\rho):=\psi(\tau,\rho) 
	\end{equation}
and 
\begin{equation}\label{transform}
\psi_2(\tau,\rho):=(1+\partial_\tau+\rho\partial_\rho)\psi_1(\tau,\rho),
\end{equation}
	which yields the system
	\begin{equation}
          \label{eq:syspsi}
          \begin{split}
	\partial_\tau \psi_1 &=\psi_2-\psi_1-\rho\partial_\rho\psi_1\\
\partial_\tau \psi_2
                             &=\partial_\rho^2\psi_1+\frac{5}{\rho}\partial_\rho \psi_1-\rho\partial_\rho\psi_2
                             -2\psi_2-\frac{3\sin(2\rho\psi_1)-6\rho\psi_1}{2\rho^3},
                             \end{split}
	\end{equation}
	with initial data
	\begin{equation*}\label{inidata}
	\psi_1(0,\rho)=Tf(T\rho),\,\, \psi_2(0,\rho)=T^2 g(T\rho).
	\end{equation*}

 \subsection{Semigroup theory}

Motivated by the above evolution equation, we define the formal differential operator $\widetilde{\Lf}$ as
\begin{align*}
\widetilde{\Lf}\uf(\rho):=\begin{pmatrix}
-\rho u_1'(\rho)-u_1(\rho)+u_2(\rho)\\
u_1''(\rho)+\frac{5}{\rho}u_1'(\rho)-\rho u_2'(\rho)-2u_2(\rho)
\end{pmatrix},
\end{align*}
where $\mathbf u=(u_1, u_2)$.
To endow this operator with an appropriate domain, we first need the proper Hilbert space to work in.
In view of the regularity we require, we accordingly define the space
\[
\mathcal{H}:= \{\uf \in H^2\times H^1(\B^6_1): \uf\text{ radial} \}
\]
and as a domain for $\widetilde{\Lf}$ we set
\[
D(\widetilde{\Lf})=\{\uf \in C^3\times C^2\big(\overline{\B}_1^6 \big):\uf \text{ radial}\}.
\]
In addition to the standard $H^2\times H^1$ inner product, we define another inner product on 
$
D(\widetilde{\Lf})
$
by setting
\begin{align*}
(\uf,\vf)_{\widetilde{\mathcal{H}}}:=&2\int_0^1u_1''(\rho)\overline{v_1''(\rho)}\rho^5
                                       d \rho +10\int_0^1
                                       u_1'(\rho)\overline{v_1'(\rho)}\rho^3
                                       d\rho 
+2\int_0^1 u_2'(\rho)\overline{v_2'(\rho)}\rho^5 d\rho \\
&+u_1(1)\overline{v_1(1)}+u_2(1)\overline{v_2(1)}
\end{align*}
and we denote the associated norm by $\|.\|_{\widetilde{\mathcal{H}}}.$
The reason for defining this alternative norm $\|.\|_{\widetilde{\mathcal{H}}}$ stems from the fact that this will allow us to prove a dissipative estimate for $\widetilde{\Lf}$. First we will however show that the norms $\|.\|_{\mathcal{H}}$ and $\|.\|_{\widetilde{\mathcal{H}}}$ are equivalent. To do so, we will need the following version of Hardy's inequality.

\begin{lem}\label{helplemoverrho}
  Let $R_0>0$. Then we have the bounds
  \begin{align*}
    \int_0^R |f(\rho)|^2 \rho d\rho&\lesssim \|f\|_{H^2(\B_R^6)}^2 \\
    \int_0^R |f'(\rho)|^2 \rho^3 d\rho &\lesssim \|f\|_{H^2(\B_R^6)}^2
  \end{align*}
  for all $R\geq R_0$ and all $f\in C^2(\overline{\B_R^6})$.
\end{lem}
\begin{proof}
An integration by parts and the one-dimensional Sobolev embedding
yield
\begin{align*}
  \int_0^R |f(\rho)|^2\rho d\rho &\lesssim \|f\|_{H^2(\B_R^6)}^2+\int_0^R
                        |f'(\rho)||f(\rho)|\rho^2 d\rho \\
  &\lesssim \|f\|_{H^2(\B_R^6)}^2+\frac{1}{\epsilon}\int_0^R
    |f'(\rho)|^2\rho^3 d\rho+\epsilon\int_0^R |f(\rho)|^2 \rho d\rho
\end{align*}
for all $R\geq R_0$ and any $\epsilon>0$. Thus, it suffices to prove the
second inequality. To this end, we again integrate by parts and obtain
\begin{align*}
\int_0^R |f'(\rho)|^2 \rho^{3} d\rho &= \int_0^R |\rho^{5} f'(\rho)|^2 \rho^{-7} d\rho \lesssim\int_0^R |f'(\rho)|\left|f''(\rho)+\frac{5}{\rho}f'(\rho)\right|\rho^{4} d\rho
\\
  &\lesssim \left(\int_0^R|f'(\rho)|^2 \rho^{3} d\rho\right)^{\frac{1}{2}} \left(\int_0^R\left|f''(\rho)+\frac{5}{\rho}f'(\rho)\right|^2\rho^5d\rho \right)^{\frac{1}{2}}
\end{align*}
for all $R\geq R_0$ and $f\in C^2(\overline{\B_R^6})$.
\end{proof}
\begin{lem}\label{lem:norms}
The norms $\|.\|_{\widetilde{\mathcal{H}}}$ and $\|.\|_{\mathcal{H}}$ are equivalent on $D(\widetilde{\Lf})$.
Consequently, they are also equivalent on $\mathcal{H}$.
\end{lem}
\begin{proof}
First, we show that for any $\uf \in D(\widetilde{\Lf})$ we have that 
\[\|\uf\|_{\widetilde{\mathcal{H}}}\lesssim\|\uf\|_{\mathcal{H}}.
\]
From the one-dimensional Sobolev embedding $L^\infty([0,1])\hookrightarrow H^1((0,1))$ we infer that
\[
|u_1(1)|^2+|u_2(1)|^2 \lesssim\|\uf\|_{\mathcal{H}}^2.
\]
An application of Lemma \ref{helplemoverrho} yields
\begin{align*}
\int_0^1 |u_1'(\rho)|^2 \rho^3 d\rho \lesssim \|u_1\|_{H^2(\B^6_1)}^2\leq \|\uf\|_{\mathcal{H}}^2.
\end{align*}
Therefore, we are left with establishing the estimate
\[\int_0^1 |u_1''(\rho)|^2 \rho^5 d\rho \lesssim \|\uf\|_{\mathcal{H}}^2.
\]
To see this, we again employ Lemma \ref{helplemoverrho} and obtain
\begin{align*}
\int_0^1 |u_1''(\rho)|^2 \rho^5 d\rho \lesssim \|u_1\|_{\dot{H}^2(\B^6_1)}^2+\int_0^1 |u_1'(\rho)|^2\rho^3 d\rho \lesssim \|\uf\|_{\mathcal{H}}^2.
\end{align*}
To prove the reverse inequality, we note that 
\begin{align*}
\int_0^1 |u_j(\rho)|^2  \rho^5 d\rho \leq& |u_j(1)|^2+\int_0^1 |u_j'(\rho)||u_j(\rho)| \rho^6 d\rho
\\
\leq& |u_j(1)|^2+\frac12\int_0^1 \left (|u_j'(\rho)|^2\rho^5+|u_j(\rho)|^2
      \rho^5\right ) d\rho
\end{align*}
which implies
\[
\|u_j\|_{L^2(\B^6_1)}^2\lesssim \|\uf\|_{\widetilde{\mathcal{H}}}^2,
\]
for $j=1,2$.
Finally, we estimate
\begin{align*}
\|u_1\|_{\dot H^2(\B^6_1)}^2\lesssim \int_0^1 \left (|u_1''(\rho)|^2
  \rho^5+|u_1'(\rho)|^2 \rho^3\right ) d\rho \leq \|\uf\|_{\widetilde{\mathcal{H}}}^2.
\end{align*} 
\end{proof}
\begin{lem}
The operator $\widetilde{\Lf}$ satisfies
\[
\Re\left(\widetilde{\Lf}\uf,\uf\right)_{\widetilde{\mathcal{H}}}\leq 0
\]
for all $\uf \in D(\widetilde{\Lf})$.
\end{lem}
\begin{proof}
Let $\uf \in D(\widetilde{\Lf})$. 
Integrating by parts shows
\begin{align*}
-\int_0^1u_1^{(3)}(\rho)\overline{u_1''(\rho)} \rho^6 d\rho= -|u_1''(1)|^2 +6\int_0^1|u_1''(\rho)|^2 \rho^5 d\rho+\int_0^1u_1''(\rho)\overline{u_1^{(3)}(\rho)} \rho^6 d\rho.
\end{align*}
Thus, 
\begin{align*}
-\Re\int_0^1u_1^{(3)}(\rho)\overline{u_1''(\rho)} \rho^6 d\rho= -\frac{|u_1''(1)|^2}{2} +3\int_0^1|u_1''(\rho)|^2 \rho^5 d\rho.
\end{align*}
Using this, we obtain
\begin{align*}
\Re\int_0^1[\widetilde{\Lf}\uf]_1''(\rho)\overline{u_1''(\rho)}\rho^5 d\rho
&=\Re\left(\int_0^1 u_2''(\rho)\overline{u_1''(\rho)}\rho^5 d \rho-\int_0^1u_1^{(3)}(\rho)\overline{u_1''(\rho)} \rho^6 d\rho\right)
\\
&\quad-3 \int_0^1 |u_1''(\rho)|^2 \rho^5 d\rho
\\
&=\Re\int_0^1 u_2''(\rho)\overline{u_1''(\rho)}\rho^5 d \rho-\frac{|u_1''(1)|^2}{2}.
\end{align*}
Similarly, we see that 
\begin{align*}
\Re \int_0^1  [\widetilde{\Lf}\uf]_2'(\rho)\overline{u_2'(\rho)} \rho^5 d \rho
&=\Re\left(\int_0^1 u_1^{(3)}(\rho)\overline{u_2'(\rho)}\rho^5 d
   \rho+5\int_0^1\left (u_1''(\rho)\overline{u_2'(\rho)} \rho^4 -u_1'(\rho)\overline{u_2'(\rho)} \rho^3\right) d \rho\right)
\\
&\quad-\frac{|u_2'(1)|^2}{2}
\\
&=\Re\left(u_1''(1)\overline{u_2'(1)}-\int_0^1 u_1''(\rho)\overline{u_2''(\rho)}\rho^5 d \rho-5\int_0^1u_1'(\rho)\overline{u_2'(\rho)} \rho^3 d \rho\right)
\\
&\quad-\frac{|u_2'(1)|^2}{2}
\end{align*}
It follows that
\begin{align*}
\Re\left(\int_0^1[\widetilde{\Lf}\uf]_1''(\rho)\overline{u_1''(\rho)}\rho^5 d\rho+\int_0^1  [\widetilde{\Lf}\uf]'_2(\rho)\overline{u_2'(\rho)} \rho^5 d \rho\right)=& -\frac{1}{2}(|u_1''(1)|^2+|u_2'(1)|^2)
\\
&+\Re(u_1''(1)\overline{u_2'(1)})
\\
&-5\Re \int_0^1 u_1'(\rho)\overline{u_2'(\rho)}\rho^{3} d \rho=:I_1.
\end{align*}
Therefore, a short calculation shows
\begin{align*}
2I_1+10\Re  \int_0^1 [\widetilde{\Lf}\uf]_1'(\rho)\overline{ u_1'(\rho)}\rho^3d\rho \leq&-\frac{1}{2}\left(|u_1''(1)|^2+|u_2'(1)|^2\right) -5|u_1'(1)|^2 \\
&+\Re(u_1''(1)\overline{u_2'(1)}).
\end{align*}
Putting everything together, we obtain
\begin{align*}
\Re\left(\widetilde{\Lf}\uf,\uf\right)_{\mathcal{\widetilde{H}}}&\leq
-\frac{1}{2}\left(|u_1''(1)|^2+|u_2'(1)|^2\right)- 5|u_1'(1)|^2
+\Re(u_1''(1)\overline{u_2'(1)})
\\
&\quad+\Re\left(u_1(1)\overline{u_2(1)}-u_1'(1)\overline{u_1(1)}\right)-|u_1(1)|^2\\
&\quad+\Re\left( u_1''(1)\overline{u_2(1)}+5u_1'(1)\overline{u_2(1)}-u_2'(1)\overline{u_2(1)}\right)-2|u_2(1)|^2.
\end{align*}
By employing the elementary inequality
\begin{equation}\label{elementaryineq}
\Re(a\overline{b}+a\overline{c}-b\overline{c})\leq \frac{1}{2}(|a|^2+|b|^2+|c|^2),
\end{equation}
once with $a=u_2(1)$, $b=u_1'(1)$, $c=u_1(1)$
and once with $a=u_1''(1)$, $b=u_2'(1)$, $c=u_2(1)$,
we obtain
\begin{align*}
\Re\left(\widetilde{\Lf}\uf,\uf\right)_{\mathcal{\widetilde{H}}}\leq
\Re\left(4 u_1'(1)\overline{u_2(1)}\right)-4|u_1'(1)|^2-|u_2(1)|^2\leq 0.
\end{align*}
\end{proof}
The next Lemma will be the final ingredient in showing that $\widetilde{\Lf}$ satisfies the assumption of the Lumer-Phillips theorem.
\begin{lem}\label{density}
The range of the operator $1-\widetilde{\Lf}$ is dense in $\mathcal{H}.$ 
\end{lem}
\begin{proof}
Let $\ff=(f_1,f_2) \in C^\infty\times C^{\infty}(\overline{\B^6_1})$. Then the equation $(\lambda-\widetilde{\Lf})\uf=\ff$ written out explicitly has the form
\begin{align*}
(\lambda +1)u_1(\rho)-u_2(\rho) + \rho u_1'(\rho)&=f_1(\rho)\\
(\lambda+2)u_2(\rho)-u_1''(\rho)-\frac{5}{\rho}u_1'(\rho)+ \rho u_2'(\rho)&=f_2(\rho).
\end{align*}
Hence for $\lambda=1$ the first equation implies 
\begin{equation}\label{form of u_2}
u_2(\rho)=2u_1(\rho) + \rho u_1'(\rho)-f_1(\rho).
\end{equation}
Plugging this into the second one yields
\begin{equation}\label{lambda=1}
(\rho^2-1)u_1''(\rho)+
\left(6\rho-\frac{5}{\rho}\right)u_1'(\rho)+ 6u_1(\rho)=F_1(\rho)
\end{equation}
with $F_1(\rho)=f_2(\rho)+3 f_1(\rho)+\rho f_1'(\rho)$.
A fundamental system for the homogeneous equation 
\begin{equation*}
(\rho^2-1)u_1''(\rho)+
\left(6\rho-\frac{5}{\rho}\right)u_1'(\rho)+ 6u_1(\rho)=0
\end{equation*}is given by
$$\psi_1(\rho):=\frac{2-\rho^2-2\sqrt{1-\rho^2}}{\rho^4},\;\;\;\psi_2(\rho):=\frac{2-\rho^2}{\rho^4}.
$$
Furthermore, the Wronskian of $\psi_1$ and $\psi_2$ is given by
$$W(\psi_1,\psi_2)(\rho)=-\frac{2}{\rho^5\sqrt{1-\rho^2}}.
$$
Therefore, a solution of Eq.~\eqref{lambda=1} is given by
\begin{align*}
u_1(\rho)=&\psi_1(\rho)\int_\rho^1\frac{\psi_2(s)F_1(s)}{W(\psi_1,\psi_2)(s)\left(s^2-1\right)} ds+\psi_2(\rho)\int_0^\rho\frac{\psi_1(s)F_1(s)}{W(\psi_1,\psi_2)(s)\left(s^2-1\right)} ds\\
=& \psi_1(\rho)\int_\rho^1\frac{s(2-s^2)F_1(s)}{2\sqrt{1-s^2}}ds+ \psi_2(\rho)\int_0^\rho\frac{s(2-s^2-2\sqrt{1-s^2})F_1(s)}{2\sqrt{1-s^2}} ds.
\end{align*}
By standard ODE theory, it follows that $u_1 \in C^\infty(\B^6_1)$ 
and so we only need to check the endpoints.
A Taylor expansion shows that
$\psi_1$ is a smooth even function on $[0,1)$ and so, we see that
\begin{align}\label{eq:1 smooth}
\psi_1(\rho)\int_\rho^1\frac{s(2-s^2)F_1(s)}{2\sqrt{1-s^2}}ds\in C^\infty([0,1)).
\end{align}
Further, by rescaling according to $t\rho= s $ we obtain that
\begin{align*}
\psi_2(\rho)\int_0^\rho\frac{s(2-s^2-2\sqrt{1-s^2})F_1(s)}{2\sqrt{1-s^2}} ds=\frac{2-\rho^2}{\rho^2}\int_0^1\frac{t(2-\rho^2t^2-2\sqrt{1-\rho^2t^2})F_1(\rho t)}{2\sqrt{1-\rho^2t^2}} dt.
\end{align*}
Note that for $\rho$ sufficiently close to $0$, the expansion
\begin{align*}
\sqrt{1-\rho^2}=1-\frac{\rho^2}{2}-\frac{\rho^4}{8}+O(\rho^6)
\end{align*}
 holds, where the $O$ term is a smooth function.
 Using this, we see that
\begin{align}\label{eq:2 near 0}
\psi_2(\rho)\int_0^\rho\frac{s(2-s^2-2\sqrt{1-s^2})F_1(s)}{2\sqrt{1-s^2}} ds=(2-\rho)\int_0^1\frac{\left(\frac
{\rho^2t^5}{8}+O(\rho^4t^7)\right)F_1(\rho t)}{2\sqrt{1-\rho^2t^2}} dt.
\end{align}
As a consequence of \eqref{eq:1 smooth}, \eqref{eq:2 near 0} and the fact that $\psi_1$ is a smooth even function, one sees that
\begin{align*}
u_1\in C^\infty([0,1)) 
\end{align*}
and 
\begin{align*}
u_1'(0)=u_1^{(3)}(0)=0.
\end{align*}
As we are left with investigating $u_1$ at the endpoint $\rho=1$, we rewrite $u_1$ as
\begin{align*}
u_1(\rho)&=-\frac{2\sqrt{1-\rho^2}}{\rho^4}\int_\rho^1\frac{s(2-s^2)F_1(s)}{2\sqrt{1-s^2}}ds+ \psi_2(\rho)\int_0^1 \frac{s(2-s^2)F_1(s)}{2\sqrt{1-s^2}}ds
\\
&\quad -\psi_2(\rho)\int_0^\rho sF_1(s) ds.
\end{align*}
Evidently, all of the above terms but the first one are smooth at $\rho=1$ and so we only need to show that
\begin{align*}
v_1(\rho):=\sqrt{1-\rho^2}\int_\rho^1\frac{s(2-s^2)F_1(s)}{\sqrt{1-s^2}}ds
\end{align*}
is smooth at $\rho=1$.
By transforming according to $s=(1-\rho)t+1$
we obtain that
\begin{align*}
\int_\rho^1\frac{s(2-s^2)F_1(s)}{\sqrt{1-s^2}}ds&=(1-\rho)\int_{-1}^0\frac{\widetilde{F_1}((1-\rho)t+1)}{\sqrt{1-[(1-\rho)t+1]^2}}dt
\\
&=(1-\rho)^{\frac{1}{2}}\int_{-1}^0\frac{\widetilde{F_1}((1-\rho)t+1)}{\sqrt{-t}\sqrt{2+(1-\rho)t}}dt
\end{align*}
with $\widetilde{F_1}(s)=s(2-s^2)F_1(s)$.
Consequently,
\begin{align*}
v_1(\rho)=(1-\rho)\sqrt{1+\rho}\int_{-1}^0\frac{\widetilde{F_1}((1-\rho)t+1)}{\sqrt{-t}\sqrt{2+(1-\rho)t}}dt
\end{align*}
and thus, $v_1$ is a smooth function at $\rho=1$. In summary, we conclude that $u_1 \in C^3(\overline{\B^6_1})$ and from \eqref{form of u_2} it follows that $u_2\in C^2(\overline{\B^6_1})$.
\end{proof}

The preceding two lemmas show that the requirements of the Lumer-Phillips theorem are satisfied
and we obtain our next result.
\begin{lem}\label{semigroupgen}
The operator $\widetilde{\Lf}$ is closable and its closure, denoted by $\Lf_0$, generates a semigroup $\Sf_0$ on $\mathcal{H}$ that satisfies
\begin{equation*}
\|\Sf_0(\tau)\ff\|_{\mathcal{H}} \lesssim \|\ff\|_{\mathcal{H}}
\end{equation*}
for all $\ff \in \mathcal{H}$ and all $\tau \geq 0$.
\end{lem}
\subsection{The free Strichartz estimates}
Having established the existence of the ``free'' semigroup $\Sf_0$, we will continue by proving Strichartz estimates for this semigroup. To do so, we will from now on always assume that $T$ is confined to $\left[\frac{1}{2},\frac{3}{2}\right]$. This restriction of $T$ leads to no loss of generality as we only care for $T$ close to $1$ anyway.
\begin{lem}\label{lem:extension}
There exists a family of extension operators $\Ef_T:H^2\times H^1(\B^6_T) \to H^2\times H^1(\R^6)$ such that
\begin{align*}
\|\Ef_T\ff\|_{H^2\times H^1(\R^6)}\lesssim \|\ff\|_{H^2\times H^1(\B^6_T)}
\end{align*}
for all $T\in \left[\frac{1}{2},\frac{3}{2}\right]$ and all $\ff \in H^2\times H^1(\B^6_T)$.

\end{lem}
\begin{proof}
Let $\Ef_1$ be a bounded extension operator  from $H^2\times H^1(\B^6_1)$ to $ H^2\times H^1(\R^6)$.  For $\ff \in H^2\times H^1(\B^6_T)$ we define an extension by first mapping $\ff$ to $H^2\times H^1(\B^6_1)$ via the scaling $\ff\mapsto \ff(T.)$, then extending via $\Ef_1$, and finally undoing the scaling. Since $T\in \left[\frac{1}{2},\frac{3}{2}\right]$, the resulting family of extension operators satisfies the desired estimate.
\end{proof}
This extension Lemma allows us to prove the aforementioned Strichartz
estimates for $\Sf_0$.

\begin{lem}\label{Strichart}
Let $p\in [2,\infty]$ and $q\in [6,12]$ be such that $\frac{1}{p}+\frac{6}{q}=1$. Then we have the estimate
\begin{align*}
\|\left[\Sf_0(\tau)\ff\right]_1\|_{L^p_\tau(\R_+)L^q(\B^6_1)}\lesssim \|\ff\|_{\mathcal{H}}
\end{align*}
for all $\ff \in \mathcal{H}$.
Furthermore, also the inhomogeneous estimate
\begin{align*}
\left\|\int_0^\tau\left[\Sf_0(\tau-\sigma)\hfh(\sigma)\right]_1 d\sigma\right\|_{L^p_\tau(I)L^q(\B^6_1)}\lesssim \|\hfh\|_{L^1(I)\mathcal{H}}
\end{align*}
holds for all $\hfh \in L^1(\R_+,\mathcal{H})$ and all intervals $I
\subset [0,\infty)$ containing $0$.
\end{lem}
\begin{proof}
Let $T\in \left[\frac{1}{2},\frac{3}{2}\right]$, $\ff \in C^2 \times C^1(\overline{\B^6_1})$, and define the scaling operator\\ $\Af_T:H^2\times H^1(\B^6_T)\to \mathcal{H}$ by
 \begin{align*}
\Af_T \ff =( Tf_1(T.),T^2 f_2(T.)).
\end{align*}
In view of the coordinate transformation \eqref{coordinate}, the evolution $\Sf_0(.)\ff$
is given by the solution $u\in C^2(\R_+\times \R)$ restricted to the
lightcone $\Gamma^T$ of the equation
\begin{align*}
\begin{cases}
\left(\partial_t^2-\partial_r^2-\frac{5}{r}\partial_r
\right)u(t,r)=0\\
(u(0,.),\partial_0 u(0,.))=\Ef_T\Af_T^{-1}\ff,
\end{cases}
\end{align*}
where $\Ef_T$ is the Sobolev
extension from Lemma \ref{lem:extension}.
Therefore, $$
\left[\Sf_0(\tau)\ff\right]_1(\rho)=Te^{-\tau}u(T-Te^{-\tau},Te^{-\tau}\rho).
$$
This implies
\begin{align*}
\|\left[\Sf_0(\tau)\ff\right]_1\|_{L^{12}(\B^6_1)}\lesssim&
\|Te^{-\tau}u(T-Te^{-\tau},e^{-\tau}.)\|_{L^{12}(\B^6_1)}\\
\lesssim& Te^{-\frac{1}{2}\tau}\|u(T-Te^{-\tau},.)\|_{L^{12}(\R^6)}
\end{align*}
and so 
\begin{align*}
\|\left[\Sf_0(.)\ff\right]_1\|_{L^2(\R_+)L^{12}(\B^6_1)}
\lesssim& \left (\int_0^\infty \left(Te^{-\frac{1}{2}\tau}\|u(T-Te^{-\tau},.)\|_{L^{12}(\R^6)}\right)^2 d \tau\right)^{\frac{1}{2}}
\\
\lesssim&\|u\|_{L^2(\R_+)L^{12}(\R^6)}
\lesssim\left\|\Ef_T\Af_T^{-1}\ff\right\|_{H^2\times H^1(\R^6)}
\\
\lesssim&\left\|\Af_T^{-1}\ff\right\|_{H^2\times H^1(\B^6_T)} \lesssim \|\ff\|_{\mathcal{H}}
\end{align*}
by the standard Strichartz estimates, which can for instance be found in \cite{Sog95}.
Thanks to the Sobolev embedding $H^2(\B^6_1)\hookrightarrow L^6(\B^6_1)$,
we have that
\begin{align*}
\|\left[\Sf_0(\tau)\ff\right]_1\|_{L^6(\B^6_1)}\lesssim \|\Sf_0(\tau)\ff\|_{\mathcal{H}}\lesssim \|\ff\|_{\mathcal{H}}
\end{align*} and so, the other endpoint estimate, given by
\[
\|\left[\Sf_0(.)\ff\right]_1\|_{L^\infty(\R_+)L^6(\B^6_1)}\lesssim \|\ff\|_{\mathcal{H}},
\]
follows immediately.

For the general estimate we use interpolation to obtain
\begin{align*}
\|[\Sf_0(\tau)\ff]_1\|_{L^q(\B^6_1)}\lesssim \|[\Sf_0(\tau)\ff]_1\|_{L^6(\B^6_1)}^{\frac{12}{q}-1}\|[\Sf_0(\tau)\ff]_1\|_{L^{12}(\B^6_1)}^{2-\frac{12}{q}},
\end{align*}
hence
\begin{align*}
\|[\Sf_0(.)\ff]_1\|_{L^p(\R_+)L^q(\B^6_1)}\lesssim \|[\Sf_0(.)\ff]_1\|_{L^\infty(\R_+)L^6(\B^6_1)}^{\frac{12}{q}-1}\|[\Sf_0(.)\ff]_1\|_{L^2(\R_+)L^{12}(\B^6_1)}^{2-\frac{12}{q}}\lesssim \|\ff\|_{\mathcal{H}},
\end{align*}
provided that $\frac{1}{p}+\frac{6}{q}=1$.

To obtain the inhomogeneous estimate, we let $I=[0,\tau_0)\subset [0,\infty)$ and use Minkowski's inequality to estimate
\begin{align*}
&\left\|\int_0^\tau\left[\Sf_0(\tau-\sigma)\hfh(\sigma)\right]_1 d\sigma\right\|_{L^p_\tau(I) L^q(\B^6_1)} 
\\
=&\left\|\int_0^{\tau_0}\ind_{[0,\tau_0]}(\tau-\sigma)\left[\Sf_0(\tau-\sigma)\hfh(\sigma)\right]_1 d\sigma\right\|_{L^p_\tau(I)L^q(\B^6_1)}
\\
\leq &\int_0^{\tau_0}\left\|\ind_{[0,\tau_0]}(\tau-\sigma)\left[\Sf_0(\tau-\sigma)\hfh(\sigma)\right]_1 \right\|_{L^p_\tau(\R_+)L^q(\B^6_1)}d\sigma
\\
\leq &\int_0^{\tau_0}\left\|\left[\Sf_0(\tau)\hfh(\sigma)\right]_1 \right\|_{L^p_\tau(\R_+)L^q(\B^6_1)}d\sigma
\\
\lesssim &\int_0^{\tau_0}\left\|\hfh(\sigma)\right\|_{\mathcal{H}}d\sigma.
\end{align*}
\end{proof}

We will also require Strichartz estimates for the derivatives. 
\begin{lem}\label{freestr2}
We have
	\begin{align*}
	\|\left[\Sf_0(\tau)\ff\right]_1\|_{L^2_\tau(\R_+)\dot{W}^{1,4}(\B^6_1)}\lesssim \|\ff\|_{\mathcal{H}}
	\end{align*}
	for all $\ff\in \mathcal{H}$.
	Moreover, 
	\begin{align*}
	\left\|\int_0^\tau\left[\Sf_0(\tau-\sigma)\hfh(\sigma)\right]_1 d\sigma\right\|_{L^2_\tau(I)\dot{W}^{1,4}(\B^6_1)}\lesssim \|\hfh\|_{L^1(I)\mathcal{H}}
	\end{align*}
	for all $\hfh \in L^1(\R_+,\mathcal{H})$ and all intervals
        $I\subset [0,\infty)$ containing $0$.
	Finally, for the second component of the semigroup $\Sf_0$ we have the estimate
\begin{equation*}
	\left\| [\Sf_0(\tau)\ff]_2\right\|_{L^2_\tau(\R_+)L^4(\B_1^6)}\lesssim\|\ff\|_{\mathcal{H}}
	\end{equation*}
for all $\ff\in \mathcal{H}$ as well as the inhomogeneous estimate
	\begin{equation*}
	\left\| \int_0^\tau [\Sf_0(\tau-\sigma)\hfh(\sigma)]_2 d \sigma\right\|_{L^2_\tau(I)L^4(\B_1^6)} \lesssim\|\hfh\|_{L^1(I)\mathcal{H}}
	\end{equation*}
	 for all $\hfh \in L^1(\R_+,\mathcal{H})$ and all intervals
         $I\subset [0,\infty)$ containing $0$.
\end{lem}
\begin{proof}
We argue in a similar fashion as in the proof of Lemma \ref{Strichart}
and so we again assume that $\ff\in C^2\times C^1(\overline{\B}^6_1)$
and let $u\in C^2(\R_+\times \R)$ be the solution of the Cauchy problem  
\begin{align*}
\begin{cases}
\left(\partial_t^2-\partial_r^2-\frac{5}{r}\partial_r
\right)u(t,r)=0\\
(u(0,.), \partial_0 u(0,.))=\Ef_T\Af_T^{-1}\ff,
\end{cases}
\end{align*}
where, as before, $\Af_T:H^2\times H^1(\B^6_T)\to \mathcal{H}$ with
 \begin{align*}
\Af_T \ff =( Tf_1(T.),T^2 f_2(T.))
\end{align*}
and $\Ef_T$ is the Sobolev extension operator from Lemma
\ref{lem:extension}. 
It follows that
\begin{align*}
\|[\Sf_0(\tau)\ff]_1 \|_{\dot{W}^{1,4}(\B^6_1)}\lesssim& \|e^{-2\tau}\partial_1u(T-Te^{-\tau},Te^{-\tau}.)\|_{L^4(\B^6_1)}
\\
\lesssim& \|e^{-\frac{1}{2}\tau}\partial_1u(T-Te^{-\tau},.)\|_{L^4(\R^6)}.
\end{align*}  
This implies
\begin{align*}
\|[\Sf_0(.)\ff]_1 \|_{L^2(\R_+)\dot{W}^{1,4}(\B^6_1)}
\lesssim& \|u\|_{L^2(\R_+) \dot{W}^{1,4}(\R^6)}.
\end{align*}
and so the first estimate follows from the standard $L^2\dot{W}^{1,4}$
Strichartz estimate. To establish the homogeneous estimate on the
second component, it suffices to remark that
\begin{align*}
[\Sf_0(\tau)\ff(\rho)]_2=[(\partial_\tau + \rho\partial_\rho+1)\Sf_0(\tau)\ff(\rho)]_1
\end{align*}
and therefore, the claim follows analogously to the above. 

The inhomogeneous estimates are a consequence of Minkowski's inequality, see the proof of Lemma \ref{Strichart}.
\end{proof}

This concludes our discussion of the free semigroup for the time being and we turn to perturbations of the blowup solution $u^T_*$. 
In view of Eq.~\eqref{eq:syspsi} we set 
\begin{align*}
N(u)(\rho):=-\frac{3\sin(2\rho u(\rho))-6\rho u(\rho)}{2\rho^3}
\end{align*}
 and further make the ansatz $\Psi=\Psi_*+\Phi,$ where
\begin{equation}
  \Psi_*(\rho):=
  \begin{pmatrix}
    \psi_{*_1}(\rho) \\ \psi_{*_2}(\rho)
  \end{pmatrix}
  :=
  \begin{pmatrix}
\frac{2}{\rho} \arctan\left(\frac{\rho}{\sqrt{2}}\right) \\
\frac{2\sqrt{2}}{\rho^2+2}
\end{pmatrix}
\end{equation}
is the blowup solution $u^T_*$ in similarity coordinates. Our next step is to linearise the nonlinearity at this solution and so we define the operator $\Lf'$, mapping from $\mathcal{H}$ to $\mathcal{H}$, by
\begin{align*}
\Lf' \uf (\rho):=\begin{pmatrix}
0\\
\frac{48}{(\rho^2+2)^2}u_1(\rho)
\end{pmatrix}
\end{align*}
and set
\begin{align*}
\Nf (\uf)(\rho) :=\begin{pmatrix}
0\\
N(\psi_{*_1}+u_1)(\rho)-N(\psi_{*_1})(\rho)-\frac{48}{(\rho^2+2)^2}u_1(\rho)
\end{pmatrix}.
\end{align*}
Finally, we set 
$\Lf:=\Lf_0+\Lf'$
and remark that since $\Lf'$ is a compact linear operator, the bounded perturbation theorem implies that $\Lf$ will also generate a semigroup on $\mathcal{H}$, which we denote by $\Sf$. 
This allows us to formally rewrite our equation in Duhamel form as
\begin{align}\label{integraleq}
\Phi(\tau)=\Sf(\tau)\uf+\int_0^\tau \Sf(\tau-\sigma)\Nf(\Phi(\sigma)) d \sigma.
\end{align}
To make sense of this equation, we will show in the following that
$\Sf$ satisfies Strichartz estimates as in Lemmas \ref{Strichart} and
\ref{freestr2}, provided we project away the unstable direction. This will naturally give meaning to Eq.~\eqref{integraleq} in an appropriate Strichartz space.

\subsection{Spectral analysis of $\Lf$}

In what follows, we will analyse the spectrum of the operator
$\Lf$. To this end, we remark that for any $\lambda\in \C$ with
$\Re\lambda>0$, we have that $\lambda \in \rho(\Lf_0)$ since
$\Lf_0$ generates a contraction semigroup. Thus, the identity
$\lambda-\Lf=(1-\Lf'\Rf_{\Lf_0}(\lambda))(\lambda-\Lf_0)$, with
$\Rf_{\Lf_0}(\lambda):=(\lambda-\Lf_0)^{-1}$, implies that any
spectral point $\lambda$  with $\Re\lambda> 0$ has to be an eigenvalue
by the spectral theorem for compact operators.

\begin{lem}
   \label{lem:spec}
  The point spectrum $\sigma_p(\Lf)$ of $\Lf$ is contained in $\{z\in \C:
  \Re z< 0\}\cup \{1\}$.
\end{lem}

\begin{proof}
Suppose that $\lambda$ is an eigenvalue of $\Lf$ with $\Re\lambda\geq 0$,
i.e., there exists a nontrivial $\uf\in D(\Lf)$ such that
$(\lambda-\Lf)\uf=\mathbf 0$.
A similar computation
as in the proof of Lemma \ref{density} shows that this implies the linear second order ODE
\begin{equation} \label{spectraleq}
(\rho^2-1)u_1''(\rho)+\left(2(\lambda+2)\rho-\frac{5}{\rho}\right)u_1'(\rho)+(\lambda+2)(\lambda+1)u_1(\rho)-\frac{48}{(\rho^2+2)^2}u_1(\rho)=0
\end{equation}
and $u_1$ is nontrivial.
By assumption, $\uf \in \mathcal H$ and thus,
$u_1\in H^2(\B^6_1)$. ODE regularity yields $u_1\in
C^\infty(0,1)$. In order to determine the behavior at the singular
endpoints, we employ Frobenius' method.
The indices are $\{0, -4\}$ at $\rho=0$ and $\{0, \frac32-\lambda\}$
at $\rho=1$.
From the endpoint behavior and $u_1\in H^2(\B_1^6)$ it follows that $u_1\in C^\infty([0,1])$.
In terms of the variable $v_1(\rho):=\rho u_1(\rho)$, Eq.~\eqref{spectraleq}
reads
\[ (1-\rho^2)v_1''(\rho)+\left
    [\frac{3}{\rho}-2(\lambda+1)\rho\right]v_1'(\rho)
  -\lambda(\lambda+1)v_1(\rho)-\frac{3(\rho^4-12\rho^2+4)}{\rho^2(\rho^2+2)^2}v_1(\rho)=0. \]
This spectral problem has been studied thoroughly in
\cite{CosDonGlo17} and from there we obtain that necessarily $\lambda=1$. 
\end{proof}

\begin{lem}
The spectrum of $\Lf=\Lf_0+\Lf'$ satisfies $\sigma(\Lf)\subset\{z\in
\C:\Re(z)\leq 0\} \cup\{1\}$, with $1\in \sigma_p(\Lf)$. Furthermore, the eigenvalue $1$ has geometric and algebraic multiplicity one and the associated eigenfunction is given by
$$
\gf(\rho)=\begin{pmatrix}
\frac{1}{2+\rho^2}\\ \frac{4}{(2+\rho^2)^2}
\end{pmatrix}.
$$
\end{lem}
\begin{proof}
Obviously, $\gf\in D(\Lf)$ and a straightforward computation shows
that $(1-\Lf)\gf=\mathbf 0$. Hence, $1\in \sigma_p(\Lf)$.
By reduction of order we see that a second solution to the generalized eigenvalue equation (\ref{spectraleq}) with $\lambda=1$ is given by
$$
\widetilde{g}_1(\rho)=g_1(\rho)\int_\rho^1 \frac{(s^2+2)^2}{s^5\sqrt{1-s^2}}\, d s.
$$
Now any solution of the equation has to be a linear combination of $g_1$ and $\widetilde{g}_1$ and as $\widetilde{g}_1$ is not in $H^2(\B^6_1)$, we conclude that an eigenfunction has to be a multiple of $\gf$ and therefore the geometric multiplicity of the eigenvalue $1$ is one. Next, let $\Pf$ be the spectral projection associated to this eigenvalue, i.e.,
$$
\Pf \uf:=\int_\gamma \Rf_{\Lf}(\lambda) \uf\, d \lambda,
$$
where $\gamma:[0,1]\to \C$, $\gamma(t)=1+\frac{e^{2\pi i t}}{2}$. 
Note that $\dim \Pf < \infty$, since otherwise $1$ would belong to the
essential spectrum of $\Lf$. This, however, is impossible because the
essential spectrum is invariant under compact perturbations and
$1\notin\sigma(\Lf_0)$. As $\Pf$ is a projection, we have a
decomposition of $ \mathcal{H}$ into the closed subspaces $\rg \Pf$
and $\ker \Pf$. This also yields a decomposition of $\Lf$ into the
operators $\Lf_{\rg \Pf}$ and $\Lf_{\ker \Pf}$ which act as operators
on $\rg \Pf$ and $\ker \Pf$ respectively.
The inclusion $\langle \gf \rangle \subset \rg \Pf$ is immediate and
we claim that in fact $\rg\Pf=\langle\gf\rangle$. To see this, note
that the operator $(\I_{\rg \Pf}-\Lf_{\rg \Pf}): \rg \Pf \to \rg \Pf$
is nilpotent as the only eigenvalue of $ \Lf_{\rg \Pf}$ is
$1$. Therefore, there exists a minimal $n\geq 1$ such that
$(\I_{\rg\Pf}-\Lf_{\rg \Pf})^n \uf=0$ for all $\uf \in \rg \Pf$. If
$n=1$, we are done. If not, then there exists a $\vf \in \rg \Pf$ such
that $(\I_{\rg \Pf}-\Lf_{\rg \Pf})\vf =\gf$. This implies that $v_1 $ satisfies the inhomogeneous ODE
\begin{equation*}
(\rho^2-1)v_1''(\rho)+\left(6\rho-\frac{5}{\rho}\right)v_1'(\rho)+\left(6-\frac{48}{(\rho^2+2)^2}\right)v_1(\rho)=G(\rho)
\end{equation*}
with $G(\rho)=g_2(\rho)+3 g_1(\rho)+\rho g_1'(\rho)=\frac{\rho^2+10}{(2+\rho^2)^2}.$
By the variation of constants formula, $v_1$ has to be of the form
\begin{align*}
v_1(\rho)=&c_1 g_1(\rho)+c_2\widetilde{g}_1(\rho)-g_1(\rho)\int_{\rho}^{1} \frac{\widetilde{g_1}(s)G(s)}{\left(1-s^2\right)W(g_1,\widetilde{g_1})(s)} d s
\\
&-
\widetilde{g_1}(\rho)\int_{0}^\rho \frac{g_1(s)G(s)}{\left(1-s^2\right)W(g_1,\widetilde{g_1})(s)} ds
\end{align*}
with $c_1,c_2\in \C$.
Note that
$$
W(g_1,\widetilde{g}_1)(\rho)=-g_1(\rho)^2\frac{(\rho^2+2)^2}{\rho^5\sqrt{1-\rho^2}}
$$
is strictly negative on $(0,1)$ and therefore nonvanishing on that interval.
Evidently, the expressions $
\frac{g_1(\rho)G(\rho)}{(1-\rho^2)W(g_1,\widetilde{g_1})(\rho)}$ and $
\frac{\widetilde{g_1}(\rho)G(\rho)}{(1-\rho^2)W(g_1,\widetilde{g_1})(\rho)}$
are continuous
and integrable on $(0,1)$. 
Consequently, since $v_1 \in H^2(\B^6_1)$, we must have $c_2=0$.
Furthermore, $|\widetilde g_1'(\rho)|\simeq (1-\rho)^{-\frac12}$ near
$\rho=1$ and thus, we must have
\[ \int_0^1 \frac{g_1(s)G(s)}{(1-s^2)W(g_1, \widetilde g_1)(s)}ds=0. \]
 This is however impossible due to the strict negativity of the integrand on $(0,1)$.
\end{proof}

A calculation which is very similar to the one done in the proof of Lemma 2.6 in \cite{DonRao20} yields our next result.
\begin{lem}\label{resbound}
For every $\varepsilon > 0$, there exist constants $C_\varepsilon,
K_\varepsilon >0$ such that
\[
	\|\Rf_{\Lf}(\lambda)\|_{\mathcal{H}}\leq C_\varepsilon
\]
for all $\lambda \in \C$ satisfying $|\lambda|\geq K_\varepsilon$ and $\Re \lambda\geq\varepsilon$. 
\end{lem}
This in turn implies
\begin{lem}\label{proectionsemigroup}
For every $\varepsilon >0$, there exists a constant $C_\varepsilon >0$ such that
$$
\| \Sf(\tau)(\I-\Pf)\ff\|_{\mathcal{H}} \leq C_\varepsilon e^{\epsilon\tau}\|\ff\|_{\mathcal{H}}
$$
for all $\ff \in \mathcal{H}$.
\end{lem}
\begin{proof}
This Lemma follows immediately from Lemma \ref{resbound} and the
Gearhart-Prüss-Greiner Theorem, see e.g.~\cite{EngNag99}, p.~302,
Theorem 1.11, since $\sigma(\Lf_{\ker \Pf})\subset \{\lambda\in \C: \Re(\lambda)\leq 0\}$.
\end{proof}
As the crude growth estimate from Lemma \ref{proectionsemigroup} is not good enough for our purposes, a more detailed analysis of the semigroup is needed. Therefore, let $\ff \in D(\Lf)$ and set $\widetilde{\ff}:= (\I-\Pf)\ff \in D(\Lf)$. Then, for any $\varepsilon>0$, Laplace inversion yields
\begin{equation}
\Sf(\tau)\widetilde{\ff}=\lim_{N \to \infty}\frac{1}{2\pi i}\int_{\varepsilon-i N}^{\varepsilon+i N}e^{\lambda\tau}\Rf_{\Lf}(\lambda) \widetilde{\ff}d \lambda,
\end{equation}
see \cite{EngNag99}, p.~234, Corollary 5.15.
Thus, in order to obtain quantitative information on the semigroup $\Sf(\I-\Pf)$, we need to investigate $\Rf_{\Lf}(\lambda)$ for $\lambda$ close to the imaginary axis. Note that 
$\uf=\Rf_{\Lf}(\lambda)\widetilde{\ff}$ implies $(\lambda-\Lf)\uf=\widetilde{\ff},$ which in turn implies
\begin{equation}\label{resolventeq}
(\rho^2-1)u_1''(\rho)+\left(2(\lambda+2)\rho-\frac{5}{\rho}\right)u_1'(\rho)+(\lambda+2)(\lambda+1)u_1-\frac{48}{(\rho^2+2)^2}u_1(\rho)=F_\lambda(\rho)
\end{equation}
where $F_\lambda(\rho)=f_2(\rho)+(\lambda+2)f_1(\rho)+\rho f_1'(\rho)$.
Therefore, our next step will be a detailed analysis of
Eq.~\eqref{resolventeq}.
\section{ODE analysis}
In this section we will consider the slightly more general ODE
\begin{equation}\label{greenode}
-(1-\rho^2)u''(\rho)+\left(2(\lambda+2)\rho-\frac{5}{\rho}\right)u'(\rho)+\left((\lambda+2)(\lambda+1)+V(\rho)\right)u(\rho)=F_\lambda(\rho),
\end{equation}
where, unless we explicitly state otherwise, $V$ does not denote our specific potential but rather any potential that satisfies $V\in C^\infty([0,1])$.
From now on, we will often make use of functions of symbol type. Let $I\subset \R$, $\rho_0\in\overline{I}$, and $\alpha \in \R$. We say that a smooth function $f:I \to \R$ is of symbol type and write $f(\rho)=\O((\rho_0-\rho)^{\alpha})$ if
\begin{align*}
|\partial_\rho^n f(\rho)|\lesssim_n |\rho_0-\rho|^{\alpha-n},
\end{align*}
 for all $\rho \in I$ and all $n\in \mathbb{N}_0$.
A discussion of symbol calculus can for instance be found in \cite{Don17} and \cite{DonRao20}. We will also make use of the ``Japanese bracket'' $\langle x\rangle:=\sqrt{1+|x|^2}$ and lastly, whenever we are given a function of the form $f(\rho,\lambda)$, then $f'(\rho,\lambda)$ denotes $\partial_\rho f(\rho,\lambda)$.
\subsection{Construction of a fundamental system}
To get rid of the first order term in Eq.~\eqref{greenode}, we
define a new unknown $v$ by 
$$v(\rho):=\frac{\rho^{\frac{5}{2}}}{(1-\rho^2)^{\frac{1}{4}-\frac{\lambda}{2}}}u(\rho),$$
which, for $F_\lambda=0$, turns Eq.~\eqref{greenode} into
\begin{equation}\label{no first order}
v''(\rho)+\frac{\rho^2(10+12\lambda-4\lambda^2)-15}{4\rho^2(1-\rho^2)^2}v(\rho)=\frac{V(\rho)}{1-\rho^2}v(\rho).
\end{equation}
As we are only interested in values of $\lambda$ that are close to the imaginary axis,
we assume that $\lambda=\varepsilon+ i\omega$ with $ \varepsilon \in [0,\frac{1}{4}]$ and $\omega \in \R$.
In order to better understand Eq.~\eqref{no first order}, the diffeomorphism $\varphi:(0,1)\to (0,\infty)$, given by
$$
\varphi(\rho):=\frac{1}{2}\log\left(\frac{1+\rho}{1-\rho}\right),
$$
will be crucial. Note that 
$$
\varphi'(\rho)=\frac{1}{1-\rho^2}
$$
and that the Liouville-Green Potential $Q_{\varphi}$, which is defined as
$$
Q_\varphi(\rho)=-\frac{3}{4}\frac{\varphi''(\rho)^2}{\varphi'(\rho)^2}+\frac{1}{2}\frac{\varphi'''(\rho)}{\varphi'(\rho)},
$$ is given by
$$
Q_{\varphi}(\rho )=\frac{1}{(1-\rho^2)^2}.
$$
Bearing this in mind, we rewrite Eq.~\eqref{no first order} as
\begin{align}
  \label{eq:no1strewrite}
&v''(\rho)+\frac{-9+12\lambda-4\lambda^2}{4(1-\rho^2)^2}v(\rho)-\frac{15}{4\varphi^2(\rho)(1-\rho^2)^2}v(\rho) +Q_\varphi(\rho)v(\rho)\nonumber
\\
=&\left(\frac{V(\rho)}{1-\rho^2}-\frac{15}{4(1-\rho^2)^2}+\frac{15}{4\rho^2(1-\rho^2)^2}-\frac{15}{4\varphi^2(\rho)(1-\rho^2)^2}\right)v(\rho),
\end{align}
and perform a Liouville-Green transform.
That is to say, we set $w(\varphi(\rho)):= \varphi'(\rho)^{\frac{1}{2}}v(\rho)$, which transforms
$$
v''(\rho)+\frac{-9+12\lambda-4\lambda^2}{4(1-\rho^2)^2}v(\rho)- \frac{15}{4\varphi^2(\rho)(1-\rho^2)^2}v(\rho)+Q_\varphi(\rho)v(\rho)=0
$$
into the equation
\begin{equation}\label{bessel}
w''(\varphi(\rho))+\frac{-9+12\lambda-4\lambda^2}{4}w(\varphi(\rho))-\frac{15}{4\varphi^2(\rho)} w(\varphi(\rho))=0.
\end{equation}
This is a Bessel equation and we have the fundamental system
\begin{align*}
\sqrt{\varphi(\rho)}J_2\left(a(\lambda)\varphi(\rho)\right)\\
\sqrt{\varphi(\rho)}Y_2(a(\lambda)\varphi(\rho))
\end{align*}
with $a(\lambda)= i\frac{3-2\lambda}{2}$ and the standard Bessel
functions $J_2, Y_2$, see \cite{OlvLonBoiClar10}.
It follows that the equation
\begin{align}\label{ODE3}
v''(\rho)&+\frac{-9+12\lambda-4\lambda^2}{4(1-\rho^2)^2}v(\rho)- \frac{15}{4\varphi^2(\rho)(1-\rho^2)^2}v(\rho)+Q_\varphi(\rho)v(\rho)=0
\end{align}
has the fundamental system
\begin{align*}
b_1(\rho,\lambda)&:= \sqrt{(1-\rho^2)\varphi(\rho)}J_2(a(\lambda)\varphi(\rho))\\
b_2(\rho,\lambda)&:=\sqrt{(1-\rho^2)\varphi(\rho)}Y_2(a(\lambda)\varphi(\rho)).
\end{align*}
Consequently, the asymptotics of $J_2$ and $Y_2$ will be important and
for us it will be sufficient to study them on the set
\begin{align*}
\{z\in \C: 0<|z|< r\}\cap\{ z \in \C: 0 \leq \arg z <\pi\} 
\end{align*} 
for an arbitrary (but fixed) $r>0$, where $\arg$ denotes the (principal branch of the) argument function. A Taylor expansion shows that on this set, the representations
\begin{align}\label{BesselTaylor}
J_2(z)=z^2 \left(\frac{1}{8}+\O(z^2)\right)
\qquad
Y_2(z)=z^{-2}\left(-\frac{4}{\pi}+\O(z)\right)
\end{align}
hold.

Another fundamental system is given by the Hankel
functions. However, we do not resort to the theory of Hankel functions
but rather prefer to construct a different fundamental system directly
in order to obtain good control on the error terms in the asymptotic
expansion. The point is that the decomposition of Eq.~\eqref{no first
  order} introduced in Eq.~\eqref{eq:no1strewrite} has an artificial
singularity at $\rho=1$ which would lead to weaker control of the error.

\begin{lem}\label{free ODE near 1}
There exist $r>0$ and $\rho_0\in [0,1)$ such that for $\rho \in [\rho_\lambda,1),$ where \\$\rho_\lambda:= \min\{\frac{r}{|a(\lambda)|},\rho_0\}$,
the equation
\begin{equation}\label{no V}
v''(\rho)+\frac{\rho^2(10+12\lambda-4\lambda^2)-15}{4\rho^2(1-\rho^2)^2}v(\rho)=0
\end{equation} has a fundamental system of the form
\begin{align*}
h_1(\rho,\lambda)=&\frac{\sqrt{1-\rho^2}}{\sqrt{a(\lambda)}}\left(\frac{1-\rho}{1+\rho}\right)^{\frac{3}{4}-\frac{\lambda}{2}}
\left[1+(1-\rho)\O(\langle\omega\rangle^{-1})+\O(\rho^{-1}(1-\rho)^2\langle\omega\rangle^{-1})\right]
\\
h_2(\rho,\lambda)=&\frac{\sqrt{1-\rho^2}}{\sqrt{a(\lambda)}}\left(\frac{1-\rho}{1+\rho}\right)^{-\frac{3}{4}+\frac{\lambda}{2}}
\left[1+(1-\rho)\O(\langle\omega\rangle^{-1})+\O(\rho^{-1}(1-\rho)^2\langle\omega\rangle^{-1})\right].
\end{align*}
\end{lem}
\begin{proof}
We start by rewriting Eq.~\eqref{no V} as
\begin{align*}
v''(\rho)+\frac{-5+12\lambda-4\lambda^2}{4 (1-\rho^2)^2}v(\rho)=\frac{15}{4\rho^2 (1-\rho^2)^2}v(\rho)-\frac{15}{4(1-\rho^2)^2}v(\rho)=\frac{15}{4\rho^2(1-\rho^2)}v(\rho).
\end{align*}
Now, two linearly independent solutions of the equation
\begin{align*}
w''(\rho)+\frac{-5+12\lambda-4\lambda^2}{4 (1-\rho^2)^2}w(\rho)=0
\end{align*}
are given by 
\begin{align*}
w_1(\rho,\lambda)=&\frac{\sqrt{1-\rho^2}}{\sqrt{a(\lambda)}}
                    \left(\frac{1-\rho}{1+\rho}\right)^{\frac{3}{4}-\frac{\lambda}{2}}
                    =a(\lambda)^{-\frac12}(1-\rho)^{\frac54-\frac{\lambda}{2}}(1+\rho)^{-\frac14+\frac{\lambda}{2}}
\\
w_2(\rho,\lambda)=&\frac{\sqrt{1-\rho^2}}{\sqrt{a(\lambda)}}
                    \left(\frac{1+\rho}{1-\rho}\right)^{\frac{3}{4}-\frac{\lambda}{2}}
                    =a(\lambda)^{-\frac12}(1-\rho)^{-\frac14+\frac{\lambda}{2}}(1+\rho)^{\frac54-\frac{\lambda}{2}}
\end{align*}
 with constant Wronskian
\begin{align*}
W(w_1,w_2)=\frac{2}{i}.
\end{align*}
As we intend to construct the solutions $h_1$ and $h_2$ by Volterra iterations, 
we make the ansatz
\begin{align*}
w(\rho,\lambda)=& w_1(\rho,\lambda)+\int_\rho^{\rho_1}\frac{15i w_1(\rho,\lambda)w_2(s,\lambda)}{8s^2(1-s^2)} w(s,\lambda) ds 
\\
&-\int_\rho^{\rho_1}\frac{15i w_2(\rho,\lambda)w_1(s,\lambda)}{8s^2(1-s^2)} w(s,\lambda) ds
\end{align*}
with $\rho_1$ to be determined.
Since $w_1(.,\lambda)$ has no zeros on $[0,1)$, we can rewrite this
integral equation in terms of the auxiliary variable $\widetilde w:=\frac{w}{w_1}$, which yields
\begin{align}\label{volterra near 1}
\widetilde w(\rho,\lambda)=& 1+\int_\rho^{\rho_1}\frac{15i
                             w_1(s,\lambda)w_2(s,\lambda)}{8s^2(1-s^2)}
                             \widetilde w(s,\lambda) ds 
\\
&-\int_\rho^{\rho_1}\frac{15i
     w_2(\rho,\lambda)w_1(s,\lambda)^2}{8w_1(\rho,\lambda)s^2(1-s^2)}
     \widetilde w(s,\lambda) ds \nonumber
\\
=&1+15i\int_\rho^{\rho_1}\frac{1-\left(\frac{1+\rho}{1-\rho}\frac{1-s}{1+s}\right)^{\frac{3}{2}-\lambda}}{4s^2(3-2\lambda)}
   \widetilde w(s,\lambda) ds. \nonumber
\end{align}
It follows that
\begin{align*}
\int_{\frac{1}{|a(\lambda)|}}^{\rho_1} \sup_{\rho\in [\frac{1}{|a(\lambda)|},s]}\left|\frac{1-\left(\frac{1+\rho}{1-\rho}\frac{1-s}{1+s}\right)^{\frac{3}{2}-\lambda}}{s^2(3-2\lambda)} \right| ds\lesssim \int_{\frac{1}{|a(\lambda)|}}^{\rho_1}\frac{1}{s^2|3-2\lambda|} ds \lesssim 1
\end{align*}
for any $\rho_1 \in [\frac{1}{|a(\lambda)|},1]$. This allows us to set $\rho_1=1$ and
invoke Lemma B.1 in \cite{DonSchSof11} to conclude the existence of a
unique solution $\widetilde w$ to Eq.~\eqref{volterra near 1} of the form 
$
\widetilde w(\rho,\lambda)=1+ O(\rho^{-1}\langle\omega\rangle^{-1}).
$
By re-inserting this expression into Eq.~\eqref{volterra near 1}, we
obtain the refined representation $\widetilde w(\rho,\lambda)=1+O(\rho^{-1}(1-\rho)\langle\omega\rangle^{-1})$.
Strictly speaking, the $O$-term also depends on
$\varepsilon=\Re(\lambda)$ but this dependence is of no relevance to
us and hence suppressed in our notation.

To establish that this solution is of symbol type, we again use our diffeomorphism $\varphi$ to obtain 
\begin{align*}
\int_\rho^1\frac{1-\left(\frac{1+\rho}{1-\rho}\frac{1-s}{1+s}\right)^{\frac{3}{2}-\lambda}}{s^2(3-2\lambda)} ds=&\int_\rho^1\frac{1-e^{(3-2\lambda)(\varphi(\rho)-\varphi(s))}}{s^2(3-2\lambda)}ds
\\
=&\int_{\varphi(\rho)}^\infty\frac{1-e^{(3-2\lambda)(\varphi(\rho)-y)}}{\varphi^{-1}(y)^2(3-2\lambda)}(\varphi^{-1})'(y)dy
\\
=&\int_0^\infty\frac{1-e^{-(3-2\lambda)y}}{\varphi^{-1}(y+\varphi(\rho))^2(3-2\lambda)}(\varphi^{-1})'(y+\varphi(\rho))dy.
\end{align*}
Since \[|\partial_y^k \varphi^{-1}(y)|\lesssim_k e^{-2y},\] we see that
\begin{align*}
\left| \partial_\omega^l \partial_\rho^k \int_0^\infty\frac{1-e^{-(3-2\lambda)y}}{\varphi^{-1}(y+\varphi(\rho))^2(3-2\lambda)}(\varphi^{-1})'(y+\varphi(\rho))dy\right| \lesssim_{k,l} \rho^{-k}(1-\rho)^{1-k}.
\end{align*}
To also establish symbol behavior in $\omega$, we note that the
estimate above allows us to safely assume $\omega\geq 1$ and so, rescaling leads to 
\begin{align*}
&\int_0^\infty\frac{1-e^{-(\frac{3}{2}-\lambda)y}}{\varphi^{-1}(y+\varphi(\rho))^2(3-2\lambda)}(\varphi^{-1})'(y+\varphi(\rho))dy
\\
=&\frac{1}{\omega}\int_0^\infty\frac{1-e^{-(\frac{3}{2}-\varepsilon)\frac{y}{\omega}}e^{iy}}{\varphi^{-1}\left(\frac{y}{\omega}+\varphi(\rho)\right)^2(3-2\lambda)}(\varphi^{-1})'\left(\frac{y}{\omega}+\varphi(\rho)\right)dy.
\end{align*}
Consequently, $\widetilde w$ is of the form
\begin{align*}
\widetilde w(\rho,\lambda)=1+ \O(\rho^{-1}(1-\rho)\langle\omega\rangle^{-1}).
\end{align*}
In summary, we have found a solution to Eq. \eqref{no V} of the form
\begin{align*}
h_1(\rho,\lambda)=&\frac{\sqrt{1-\rho^2}}{\sqrt{a(\lambda)}}\left(\frac{1-\rho}{1+\rho}\right)^{\frac{3}{4}-\frac{\lambda}{2}}
\left[1+\O(\rho^{-1}(1-\rho)\langle\omega\rangle^{-1})\right].
\end{align*}

The form of the error term in the expression for $h_1$ implies that we can find a $\rho_0\in [0,1)$ such that $h_1(.,\lambda)$ has no zeros on
$[\rho_0,1)$ for $\Re(\lambda)\in [0,\frac14]$.
In addition, there exists an $r>0$ such that
$h_1(.,\lambda)$ has no zeros on $[\rho_\lambda, 1)$, where  
$\rho_\lambda:=\min\{\rho_0, \frac{r}{|a(\lambda)|}\}$.
By reduction of order, the second solution is then given by 
\begin{align*}
h_2(\rho,\lambda)=& \frac{2h_1(\rho,\lambda)}{i}\int_{\rho_\lambda}^\rho h_1(s,\lambda)^{-2} ds+c(\lambda) h_1(\rho,\lambda)
\end{align*}
with $c(\lambda)$ to be determined. 
Our next step is to develop a better understanding of $h_1$ and so we set
      \begin{align*}
  K(\rho,s,\lambda):=\frac{1-\left(\frac{1+\rho}{1-\rho}\frac{1-s}{1+s}\right)^{\frac{3}{2}-\lambda}}{s^2(3-2\lambda)} 
      \end{align*}
      and observe that \begin{align*}
      \widetilde w(\rho,\lambda)=&1+\frac{15i}{4}\int_\rho^1
                                   K(\rho,s,\lambda)\widetilde
                                   w(s,\lambda) ds \\
      =&  1+\frac{15i}{4}\int_\rho^1 K(\rho,s,\lambda) ds+ \O(\rho^{-1}(1-\rho)^2 \langle\omega\rangle^{-1})
      \end{align*}
for $\rho\in [\rho_\lambda,1)$. Now, 
      \begin{align*}
 \partial_\rho \int_\rho^1 K(\rho,s,\lambda) ds=-(1-\rho)^{-\frac{5}{2}+\lambda}(1+\rho)^{\frac{1}{2}-\lambda}\int_\rho^1\frac{\left(\frac{1-s}{1+s}\right)^{\frac{3}{2}-\lambda}}{s^2} ds
      \end{align*}
and from de l'Hospital's rule it follows that
\begin{align*}
\int_\rho^1\frac{\left(\frac{1-s}{1+s}\right)^{\frac{3}{2}-\lambda}}{s^2} ds \sim 2^{-\frac{3}{2}+\lambda} \frac{(1-\rho)^{\frac{5}{2}-\lambda}}{\frac{5}{2}-\lambda}
\end{align*}
as $\rho \to 1$. This implies $$\lim_{\rho \to 1}\partial_\rho \int_\rho^1 K(\rho,s,\lambda)ds=-\frac{1}{5-2\lambda}.$$
Next, we compute that 
\begin{align*}
\partial_\rho^2 \int_\rho^1 K(\rho,s,\lambda)ds=&\rho^{-2}(1-\rho)^{-1}(1+\rho)^{-1}+\bigg[-\left(\frac{5}{2}-\lambda\right)(1-\rho)^{-\frac{7}{2}+\lambda}(1+\rho)^{\frac{1}{2}-\lambda}
\\
&-\left(\frac{1}{2}-\lambda\right)(1-\rho)^{-\frac{5}{2}+\lambda}(1+\rho)^{-\frac{1}{2}-\lambda}\bigg]\int_\rho^1\frac{\left(\frac{1-s}{1+s}\right)^{\frac{3}{2}-\lambda}}{s^2} ds.
\end{align*}
An integration by parts now yields \begin{align*}
\int_\rho^1\frac{\left(\frac{1-s}{1+s}\right)^{\frac{3}{2}-\lambda}}{s^2} ds =&\frac{2(1-\rho)^{\frac{5}{2}-\lambda}}{(1+\rho)^{\frac{3}{2}-\lambda}\rho^2 (5-2\lambda)}+4\int_\rho^1\frac{(1-s)^{\frac{5}{2}-\lambda}}{(5-2\lambda)s^3(1+s)^{\frac{3}{2}-\lambda}} ds
\\
&+\int_\rho^1\frac{(3-2\lambda)(1-s)^{\frac{5}{2}-\lambda}}{(5-2\lambda)(1+s)^{\frac{5}{2}-\lambda} s^2} ds
\end{align*}
and with this, one can easily check that $$\partial_\rho^2 \left(\int_\rho^1K(\rho,s,\lambda)ds\right)$$ is continuous on $[\rho_\lambda,1]$.
With that in mind, we rewrite $\widetilde w$ on $[\rho_\lambda,1) $ as 
\[ \widetilde w(\rho,\lambda)=1+(1-\rho)\O(\langle\omega\rangle^{-1})+\O(\rho^{-1}(1-\rho)^2\langle\omega\rangle^{-1})
\]
and
$h_2$ as
\begin{align*}
h_2(\rho,\lambda)=&\frac{2h_1(\rho,\lambda)}{i}\int_{\rho_\lambda}^\rho w_1(s,\lambda)^{-2}[1+(1-s)\O(\langle\omega\rangle^{-1})+\O(s^{-1}(1-s)^2\langle\omega\rangle^{-1})]^{-2}
ds
\\
&+c(\lambda)h_1(\rho,\lambda).
\end{align*}
Further,
as \begin{align}\label{Eq:derivative w2}
w_1(\rho,\lambda)^{-2}=\frac{i}{2}\partial_\rho \left( \frac{1+\rho}{1-\rho} \right)^{\frac{3}{2}-\lambda},
\end{align}
we compute
\begin{align*}
h_2(\rho,\lambda)=&\frac{2h_1(\rho,\lambda)}{i}\int_{\rho_\lambda}^\rho w_1(s,\lambda)^{-2}[1+(1-s)\O(\langle\omega\rangle^{-1})+\O(s^{-1}(1-s)^2\langle\omega\rangle^{-1})]
ds
\\
&+c(\lambda)h_1(\rho,\lambda)
\\
=& w_2(\rho,\lambda)[1+(1-\rho)\O(\langle\omega\rangle^{-1})+\O(\rho^{-1}(1-\rho)^2\langle\omega\rangle^{-1})]+[c(\lambda)+O(\langle\omega\rangle^{0})]h_1(\rho,\lambda)
\\
&+\frac{2h_1(\rho,\lambda)}{i}\int_{\rho_\lambda}^\rho w_1(s,\lambda)^{-2}[(1-s)\O(\langle\omega\rangle^{-1})+\O(s^{-1}(1-s)^2\langle\omega\rangle^{-1})]ds.
\end{align*}
Next, an integration by parts yields
\begin{align*}
\int_{\rho_\lambda}^\rho w_1(s,\lambda)^{-2}(1-s)\O(\langle\omega\rangle^{-1}) ds=&(1-\rho)^{-\frac{1}{2}+\lambda}(1+\rho)^{\frac{1}{2}-\lambda}\O(\langle\omega\rangle^{-1})\\
&+\int_{\rho_\lambda}^\rho(1-s)^{-\frac{1}{2}+\lambda}(1+s)^{-\frac{1}{2}-\lambda}\O(\langle\omega\rangle^0)ds+O(\langle\omega\rangle^{-1}),
\end{align*}
and since 
\begin{align*}
h_1(\rho,\lambda)(1-\rho)^{-\frac{1}{2}+\lambda}(1+\rho)^{\frac{1}{2}-\lambda}\O(\langle\omega\rangle^{-1})=w_2(\rho,\lambda)[(1-\rho)\O(\langle\omega\rangle^{-1})+\O(\rho^{-1}(1-\rho)^2\langle\omega\rangle^{-1})],
\end{align*} 
we see that
\begin{align*}
h_2(\rho,\lambda)=& w_2(\rho,\lambda)[1+(1-\rho)\O(\langle\omega\rangle^{-1})+\O(\rho^{-1}(1-\rho)^2\langle\omega\rangle^{-1})]
\\
&+h_1(\rho,\lambda)\int_{\rho_\lambda}^\rho (1-s)^{-\frac{5}{2}+\lambda}(1+s)^{\frac{1}{2}-\lambda}\O(s^{-1}(1-s)^2\langle\omega\rangle^{0}) ds
\\
&+h_1(\rho,\lambda)\int_{\rho_\lambda}^\rho(1-s)^{-\frac{1}{2}+\lambda}(1+s)^{-\frac{1}{2}-\lambda}\O(\langle\omega\rangle^0)
ds+[c(\lambda)+O(\langle\omega\rangle^{0})]  h_1(\rho,\lambda).
\end{align*}
Note, that both integrands are integrable on $[\rho_\lambda,1]$ and so we rewrite $h_2$ as
\begin{align}\label{calc:h_2}
h_2(\rho,\lambda)=& w_2(\rho,\lambda)[1+(1-\rho)\O(\langle\omega\rangle^{-1})+\O(\rho^{-1}(1-\rho)^2\langle\omega\rangle^{-1})]\nonumber
\\
&+h_1(\rho,\lambda)\int_{\rho}^1 (1-s)^{-\frac{5}{2}+\lambda}(1+s)^{\frac{1}{2}-\lambda}\O(s^{-1}(1-s)^2\langle\omega\rangle^{0}) ds 
\\
&+h_1(\rho,\lambda)\int_{\rho}^1(1-s)^{-\frac{1}{2}+\lambda}(1+s)^{-\frac{1}{2}-\lambda}\O(\langle\omega\rangle^{0})
ds+[c(\lambda)-\widehat{c}(\lambda)] h_1(\rho,\lambda)\nonumber
\end{align}
for a suitable $\widehat{c}(\lambda)$.
Finally, using the identity \eqref{Eq:derivative w2}, we observe that
\begin{align*} 
&h_1(\rho,\lambda)\int_{\rho}^1 (1-s)^{-\frac{5}{2}+\lambda}(1+s)^{\frac{1}{2}-\lambda}\O(s^{-1}(1-s)^2\langle\omega\rangle^{0}) ds 
\\
=&h_1(\rho,\lambda) \left(\frac{1+\rho}{1-\rho}\right)^{\frac{3}{2}-\lambda}\O(\rho^{-1}(1-\rho)^2\langle\omega\rangle^{-1})
\\
&+ h_1(\rho,\lambda)\int_{\rho}^1 \left(\frac{1+s}{1-s}\right)^{\frac{3}{2}-\lambda}\O(s^{-2}(1-s)\langle\omega\rangle^{-1}) ds 
\\
=&w_2(\rho,\lambda)[1+\O(\rho^{-1}(1-\rho)\langle\omega\rangle^{-1})]\O(\rho^{-1}(1-\rho)^2\langle\omega\rangle^{-1})
\\
&+w_2(\rho,\lambda)[1+\O(\rho^{-1}(1-\rho)\langle\omega\rangle^{-1})]\int_\rho^1 \left(\frac{1-\rho}{1+\rho}\frac{1+s}{1-s}\right)^{\frac{3}{2}-\lambda}\O(s^{-2}(1-s)\langle\omega\rangle^{-1})ds
\\
=&w_2(\rho,\lambda)\O(\rho^{-1}(1-\rho)^2\langle\omega\rangle^{-1})
\end{align*}
on $[\rho_\lambda,1)$,
where the last step follows from once more employing the
diffeomorphism $\varphi$ and a similar calculation as in the
construction of $h_1$ above. Thus, choosing $c=\widehat c$ yields the existence of a second solution on $[\rho_\lambda,1)$ which is of the claimed form 
\begin{align*}
h_2(\rho,\lambda)=w_2(\rho,\lambda)[1+(1-\rho)\O(\langle\omega\rangle^{-1})+\O(\rho^{-1}(1-\rho)^2\langle\omega\rangle^{-1})].
\end{align*}
\end{proof}
Note that by enlarging $r$ and $\rho_0$ if necessary, we can enforce that $h_2$ does not vanish on $[\rho_\lambda,1)$ either.
We now set
$\widehat{\rho}_\lambda:=\min\{\frac{1}{2}(\rho_0+1),\frac{2r}{|a(\lambda)|}\}\in
(\rho_\lambda,1)$
and with this, we turn to the full equation \eqref{no first order}.

\begin{lem}\label{Besselsol}
Eq.~\eqref{no first
  order} has a fundamental system of the form
\begin{align*}
\psi_1(\rho,\lambda)=&b_1(\rho,\lambda)[1+\O(\rho^2\langle\omega\rangle^0)]
\\
=&\sqrt{(1-\rho^2)\varphi(\rho)}J_2( a(\lambda)\varphi(\rho))[1+\O(\rho^2\langle\omega\rangle^0)]
\\
\psi_2(\rho,\lambda)=& b_2(\rho,\lambda)[1+\O(\rho^2\langle\omega\rangle^0)]+\O(\rho^\frac12\langle\omega\rangle^{-2})
\\
=&\sqrt{(1-\rho^2)\varphi(\rho)}Y_2( a(\lambda)\varphi(\rho))[1+\O(\rho^2\langle\omega\rangle^0)]+\O(\rho^\frac12\langle\omega\rangle^{-2})
\end{align*}
for all $\rho\in (0,\widehat \rho_\lambda]$.
\end{lem}
\begin{proof}
We commence by noting that $|\varphi(\rho)|\lesssim \rho$ for all
$\rho\in [0,\widehat\rho_\lambda]$ and thus, the condition $\rho\leq
\widehat{\rho}_\lambda$ implies that
$|a(\lambda)\varphi(\rho)|\lesssim 1$. Consequently, the Bessel asymptotics \eqref{BesselTaylor} apply.
In order to find solutions for Eq.~(\ref{no first order}), we set 
\begin{align*}
\widetilde{V}(\rho)&:=V(\rho)+\frac{15}{4(1-\rho^2)}\left(\frac{1}{\rho^2}-\frac{1}{\varphi(\rho)^2}-1\right)
\end{align*}
and by a Taylor expansion, it follows that $\widetilde{V}\in C^\infty([0,\widehat\rho_\lambda])$.
Thus, we write Eq.~\eqref{no first order} as
\begin{align}
v''(\rho)&+\frac{-9+12\lambda-4\lambda^2}{4(1-\rho^2)^2}v(\rho)- \frac{15}{4\varphi^2(\rho)(1-\rho^2)^2}v(\rho)+Q_\varphi(\rho)v(\rho)\nonumber=\frac{\widetilde{V}(\rho)}{1-\rho^2}v(\rho),
\end{align}
and note that $W(b_1(.,\lambda),b_2(.,\lambda))=\frac{2}{\pi}$.
Motivated by this, we make the ansatz
\begin{align*}
b(\rho,\lambda)= b_1(\rho,\lambda)&-\frac{\pi}{2}b_1(\rho,\lambda)\int_{0}^{\rho} b_2(s,\lambda)\frac{\widetilde{V}(s)}{1-s^2}b(s,\lambda)d s\\
&+\frac{\pi}{2}b_2(\rho,\lambda)\int_{0}^{\rho} b_1(s,\lambda)\frac{\widetilde{V}(s)}{1-s^2}b(s,\lambda) d s.
\end{align*}
Next, we remark that $J_2$ has only real zeros (see \cite{Olv97}, p. 244 Theorem 6.2) and as $a(\lambda)$ always has a nonzero imaginary part, we can divide the whole integral equation by $b_1$ and set 
$h:=\frac{b}{b_1}$.
This yields  
	\begin{equation}\label{int0}
h(\rho,\lambda)=1+\int_{0}^{\rho}K(\rho,s,\lambda)h(s,\lambda) ds,
	\end{equation}
with $$
K(\rho,s,\lambda)=\frac{\pi\widetilde{V}(s)}{2(1-s^2)}\left(\frac{b_2(\rho,\lambda)}{b_1(\rho,\lambda)}b_1(s,\lambda)^2-b_2(s,\lambda)b_1(s,\lambda)\right).
$$
From 
\begin{align*}
b_2(\rho,\lambda)b_1(\rho,\lambda)=(1-\rho^2)\varphi(\rho)Y_2(a(\lambda)\varphi(\rho))J_2(a(\lambda)\varphi(\rho))
\end{align*}
we see that the product $b_2(\rho,\lambda) b_1(\rho,\lambda)$ satisfies
$$
|b_2(\rho,\lambda)b_1(\rho,\lambda)|\lesssim \varphi(\rho)\lesssim \rho.
$$
Furthermore,
$$
\left|\frac{b_2(\rho,\lambda)}{b_1(\rho,\lambda)}b_1(s,\lambda)^2\right|\lesssim
\varphi(s)\lesssim s
$$
for all $0\leq s\leq\rho\leq \widehat{\rho}_\lambda$. Consequently, we obtain
$$
\int_0^{\widehat{\rho}_\lambda} \sup_{\rho\in [s,\widehat\rho_\lambda]}|K(\rho,s,\lambda)|ds \lesssim \langle\omega\rangle^{-2}
$$
and so, a Volterra iteration yields
the existence of a unique solution $h(\rho,\lambda)$ to Eq. (\ref{int0}) that satisfies
$$
h(\rho,\lambda)-1= O( \rho^2).
$$
Furthermore, since all the involved functions behave like symbols,
Appendix B of \cite{DonSchSof11} shows that\footnote{Strictly
  speaking, $\widehat{\rho}_\lambda$ not differentiable at $\frac{1}{2}(\rho_0+1)$. However, this is inessential and can easily be remedied by using a smoothed out version of $\widehat{\rho}_\lambda$.}
$h(\rho,\lambda)=1+\O(\rho^2\langle\omega\rangle^0)$ and thus, we
obtain the existence of a solution to
Eq.~\eqref{no first order} of the form
$$\psi_1(\rho,\lambda)=b_1(\rho,\lambda)[1+\O(\rho^2\langle\omega\rangle^0)].$$

To construct the second solution stated in the lemma, we pick a $\rho_1 \in (0,1]$ such that $\psi_1$ does not vanish for $\rho\leq\min\{\rho_1,\widehat{\rho}_\lambda\}=:\widetilde{\rho}_\lambda$ for any $0\leq \Re\lambda\leq \frac{1}{4}.$ Moreover, as
$\widetilde{b}_1(\rho,\lambda):=b_1(\rho,\lambda)\int_{\rho}^{\widetilde{\rho}_\lambda} b_1(s,\lambda)^{-2} ds$ is also a solution of Eq.~\eqref{ODE3}, there exist constants $c_1(\lambda),c_2(\lambda)$ such that
\begin{align*}
b_2(\rho,\lambda)=c_1(\lambda) b_1(\rho,\lambda)+ c_2(\lambda)\widetilde{b}_1(\rho,\lambda).
\end{align*}
Moreover, these constants are given by
\begin{align}
c_1(\lambda)&=\frac{W(b_2(.,\lambda),\widetilde{b}_1(.,\lambda))}{W(b_1(.,\lambda),\widetilde{b}_1(.,\lambda))}\\
c_2(\lambda)&=-\frac{W(b_2(.,\lambda),b_1(.,\lambda))}{W(b_1(.,\lambda),\widetilde{b}_1(.,\lambda))}.
\end{align}
Using that
$W(b_2(.,\lambda),b_1(.,\lambda))=-\frac{2}{\pi}$ and $ W(b_1(.,\lambda),\widetilde{b}_1(.,\lambda))=-1$, we infer that
$c_2=-\frac{2}{\pi}$ and $c_1(\lambda)=-W(b_2(.,\lambda),\widetilde{b_1}(.,\lambda))$.
Next, evaluating $W(b_2(.,\lambda),\widetilde{b}_1(.,\lambda))$ at $\widetilde{\rho}_\lambda$ yields
\[
W(b_2(.,\lambda),\widetilde{b_1}(.,\lambda))=-b_2(\widetilde{\rho}_\lambda,\lambda)b_1(\widetilde{\rho}_\lambda,\lambda)^{-1}=\O(\langle\omega\rangle^{0}).
\]
Keeping these facts in mind, we now turn our attention to $\psi_2$ and remark that a second solution of Eq.~\eqref{no first order} is given by
$\widetilde{\psi}_1(\rho,\lambda)=\psi_1(\rho,\lambda)\int_{\rho}^{\widetilde{\rho}_\lambda} \psi_1(s,\lambda)^{-2} ds$.
Considering this, we calculate
\begin{align*}
\psi_2(\rho,\lambda):&=c_1(\lambda)\psi_1(\rho,\lambda)+ c_2\psi_1(\rho,\lambda)\int_{\rho}^{\widetilde{\rho}_\lambda} \psi_1(s,\lambda)^{-2}ds 
\\
&=c_1(\lambda)\psi_1(\rho,\lambda)+
   c_2\psi_1(\rho,\lambda)\int_{\rho}^{\widetilde{\rho}_\lambda}
   b_1(s,\lambda)^{-2}ds \\
&\quad +c_2\psi_1(\rho,\lambda)\int_{\rho}^{\widetilde{\rho}_\lambda}\left [
   \psi_1(s,\lambda)^{-2}-b_1(s,\lambda)^{-2}\right ] ds
\\
&=b_2(\rho,\lambda)[1+\O(\rho^2\langle\omega\rangle^0)]
+c_2\psi_1(\rho,\lambda)\int_{\rho}^{\widetilde{\rho}_\lambda} \frac{\O(s^2\langle\omega\rangle^0)}{b_1(s,\lambda)^2[1+\O(s^2\langle\omega \rangle^0)]^2} ds.
\end{align*}
Since $b_1(\rho,\lambda)^{-2}=\O(\rho^{-5}\langle\omega\rangle^{-4})$,
we obtain
 \[
\int_{\rho}^{\widetilde{\rho}_\lambda} \frac{\O(s^2\langle\omega\rangle^0)}{b_1(s,\lambda)^2[1+\O(s^2\langle\omega \rangle^0)]^2} ds=\O(\rho^0\langle\omega\rangle^{-2})+\O(\rho^{-2}\langle\omega\rangle^{-4})=\O(\rho^{-2}\langle\omega\rangle^{-4}).
\]
Finally, for $|\lambda|$ large enough we see that
$\widetilde{\rho}_\lambda=\widehat{\rho}_\lambda$ and so we can safely
assume that $\widetilde{\rho}_\lambda=\widehat{\rho}_\lambda$.
\end{proof}
Having constructed a fundamental system near $0$, we turn to the endpoint $1$.
\begin{lem}\label{Hankelsol}
There exists a fundamental system
for Eq. (\ref{no first order}) of the form
\begin{align*}
\psi_3(\rho,\lambda)&= h_1(\rho,\lambda)[1+(1-\rho) \O(\langle\omega\rangle^{-1})+\O(\rho^0 (1-\rho)^2 \langle \omega \rangle^{-1})]
\\
\psi_4(\rho,\lambda)&=h_2(\rho,\lambda)[1+(1-\rho)\O(\langle\omega\rangle^{-1})+\O(\rho^0 (1-\rho)^2 \langle \omega \rangle^{-1})]
\end{align*}
for all $\rho\geq \rho_\lambda$.
\end{lem}
\begin{proof}
As $\rho_\lambda$ was chosen such that neither $h_1$ nor $h_2$ vanish on $[\rho_\lambda,1),$ a Volterra iteration akin to the one in Lemma \ref{free ODE near 1} proves the existence of the solution $\psi_3$. We can also w.l.o.g assume that $\psi_3(\rho,\lambda)$ does not vanish on $[\rho_\lambda,1)$ and obtain the second solution $\psi_4$ as
\begin{align*}
\psi_4(\rho,\lambda)=\frac{2\psi_3(\rho,\lambda)}{i}\int_{\rho_\lambda}^\rho \psi_3(s,\lambda)^{-2} ds +c(\lambda)\psi_3(\rho,\lambda)
\end{align*}
for $c(\lambda)$ to be determined.
By emulating our calculations from the proof of Lemma \ref{free ODE near 1}, we infer that $\psi_3$ takes the form
\begin{align*}
\psi_3(\rho,\lambda)=h_1(\rho,\lambda)[1+ (1-\rho)\O(\langle\omega\rangle^{-1})+ \O(\rho^0(1-\rho)^2\langle\omega\rangle^{-1})]
\end{align*}
and so
\begin{align*}
\psi_4(\rho,\lambda)&=\frac{2\psi_3(\rho,\lambda)}{i}\int_{\rho_\lambda}^\rho h_1(s,\lambda)^{-2}[1+(1-s)\O(\langle\omega\rangle^{-1})+\O(s^{0}(1-s)^2\langle\omega\rangle^{-1})]ds 
\\
&\quad+c(\lambda)\psi_3(\rho,\lambda).
\end{align*}
Manipulating the integral term as in the derivation of Lemma \ref{free ODE near 1} then shows that $\psi_4$ is of the claimed form for appropriately chosen $c(\lambda)$.
\end{proof}
Having constructed all these solutions, we now come to the task of patching them together.
\begin{lem}\label{connectioncoef}
On $[\rho_\lambda,\widehat{\rho}_\lambda]$ the solutions $\psi_3$ and $\psi_4$ have the representations

\begin{align*}
\psi_3(\rho,\lambda) &= c_{1,3}(\lambda)\psi_1(\rho,\lambda)+ c_{2,3}(\lambda)\psi_2(\rho,\lambda)\\
\psi_4(\rho,\lambda) &= c_{1,4}(\lambda)\psi_1(\rho,\lambda)+ c_{2,4}(\lambda)\psi_2(\rho,\lambda),
\end{align*}
with
\begin{align*}
	c_{1,3}(\lambda)=&\frac{e^{i\frac{5}{4}\pi}\sqrt{\pi}}{\sqrt{2}}+\O(\langle\omega\rangle^{-1})\\
	c_{2,3}(\lambda)=&i \frac{e^{i\frac{5}{4}\pi}\sqrt{\pi}}{\sqrt{2}}+\O(\langle\omega\rangle^{-1})
\end{align*}
and
\begin{align*}
c_{1,4}(\lambda)=&\frac{e^{i\frac{5}{4}\pi}\sqrt{\pi}}{\sqrt{2}}\frac{h_2(\rho_\lambda,\lambda)}{h_1(\rho_\lambda,\lambda)}+i\pi\frac{b_2(\rho_\lambda,\lambda)}{h_1(\rho_\lambda,\lambda)}+\O(\langle\omega\rangle^{-1})=\O(\langle\omega\rangle^{0})
\\
c_{2,4}(\lambda)=&i\frac{e^{i\frac{5}{4}\pi}\sqrt{\pi}}{\sqrt{2}}\frac{h_2(\rho_\lambda,\lambda)}{h_1(\rho_\lambda,\lambda)}-i\pi\frac{b_1(\rho_\lambda,\lambda)}{h_1(\rho_\lambda,\lambda)}+\O(\langle\omega\rangle^{-1})=\O(\langle\omega\rangle^{0}).
\end{align*}
\end{lem}
\begin{proof}
We know the explicit representations
\begin{align*}
c_{1,3}(\lambda)&=\frac{W(\psi_3(.,\lambda),\psi_2(.,\lambda))}{W(\psi_1(.,\lambda),\psi_2(.,\lambda))}\\
c_{2,3}(\lambda)&=-\frac{W(\psi_3(.,\lambda),\psi_1(.,\lambda))}{W(\psi_{1}(.,\lambda),\psi_{2}(.,\lambda))}
\end{align*}
and so our task boils down to calculating these Wronskians.
Evaluating the Wronskian $W(\psi_1(.,\lambda),\psi_2(.,\lambda))$ at
$\rho=0$ yields
$$
W(\psi_{1}(.,\lambda),\psi_{2}(.,\lambda))=\frac{2}{\pi}.
$$
However, computing the Wronskian of $\psi_3$ and $\psi_2$ is more challenging and for this we will need the Hankel function $H^1_2$ (see \cite{OlvLonBoiClar10}). Since 
$b_1$ and $b_2$ solve Eq.~ \eqref{ODE3} so does
\begin{align*}
\widehat{h}_1(\rho,\lambda)=\sqrt{(1-\rho^2)\varphi(\rho)}H^1_2(a(\lambda)\varphi(\rho)).
\end{align*}
Moreover, as
\begin{align}\label{Hankel 1 symbol}
H^1_2(z)=\frac{\sqrt{2}e^{-i\frac{5}{4}\pi+i z}}{\sqrt{\pi z}}[1+\O(|z|^{-1})]
\end{align}
 holds on $\{z: \Im(z)\geq 0,|z|\geq c \}$ (see \cite{Olv97}, p. 238) for any $c>0$ fixed, we see that the two representations
\begin{align*}
\widehat{h}_1(\rho,\lambda)=&\frac{\sqrt{2}e^{-i\frac{5}{4}\pi}\sqrt{1-\rho^2}}{\sqrt{\pi a(\lambda)}}\left(\frac{1-\rho}{1+\rho}\right)^{\frac{3}{4}-\frac{\lambda}{2}}
\left[1+O(\rho^{-1}\langle\log(1-\rho)\rangle^{-1}\langle\omega\rangle^{-1})\right]
\end{align*}
and 
\begin{align*}
\widehat{h}_1(\rho,\lambda)=&\frac{\sqrt{2}e^{-i\frac{5}{4}\pi}\sqrt{1-\rho^2}}{\sqrt{\pi a(\lambda)}}\left(\frac{1-\rho}{1+\rho}\right)^{\frac{3}{4}-\frac{\lambda}{2}}
\left[1+\O(\rho^{-1}(1-\rho)^0\langle\omega\rangle^{-1})\right]
\end{align*}
hold for all $\rho \in [\rho_\lambda,1]$ and all $0\leq \Re\lambda\leq \frac{1}{4}.$
Furthermore, using \eqref{Hankel 1 symbol} we can also w.l.o.g.~assume
that $\rho_\lambda$ is chosen such that $\widehat{h}_1$ has no zeros
on $[\rho_\lambda,1)$. Now, given that the potential
$\widetilde{V}$ as defined in the proof of Lemma \ref{Besselsol} is
integrable on $[0,1]$, we are able to construct a solution
$\widehat{\psi}_3 $ to Eq.~\eqref{no first order} of the form
\begin{align*}
\widehat{\psi}_3(\rho,\lambda)= \widehat{h}_1(\rho,\lambda)[1+O(\rho^{0}\langle\log(1-\rho)\rangle^{-1}\langle\omega\rangle^{-1})]
\end{align*} by employing a Volterra iteration based on $\widehat{h}_1$ and $\widehat{h}_1(\rho,\lambda)\int_{\rho_\lambda}^\rho \widehat{h}_1(s,\lambda)^{-2} ds$. 
The behavior of $\widehat{\psi}_3$ as $\rho \to 1$ now necessitates that
\begin{align*}
\widehat{\psi}_3=\frac{\sqrt{2}e^{-i\frac{5}{4}\pi}}{\sqrt{\pi}}\psi_3.
\end{align*}
Moreover, the latter form of $\widehat{h}_1$ implies that $\widehat{\psi}_3$ can also be recast as
\begin{align*}
\widehat{\psi}_3(\rho,\lambda)= \widehat{h}_1(\rho,\lambda)[1+\O(\rho^{0}(1-\rho)^0\langle\omega\rangle^{-1})]
\end{align*}
Using this and the fact that $W(H_2^1,Y_2)(z)= \frac{2}{\pi z}$, we
infer that $W(\widehat h_1(.,\lambda), b_2(.,\lambda))=\frac{2}{\pi}$
and an evaluation at $\rho_\lambda$ yields
\begin{align*}
  W(\psi_3(.,\lambda),\psi_2(.,\lambda))
  &=\frac{\sqrt\pi e^{i\frac{5}{4}\pi}}{\sqrt 2}W(\widehat
    h_1(.,\lambda), b_2(.,\lambda))[1+\O(\langle
    \omega\rangle^{-1})] \\
  &\quad +\widehat{h}_1(\rho_\lambda,\lambda)b_2(\rho_\lambda,\lambda)\O(\langle\omega\rangle^{0})+\O(\langle\omega\rangle^{-2}) \\
  &=\frac{\sqrt{2}e^{i\frac{5}{4}\pi}}{\sqrt{\pi}}[1+\O(\langle\omega\rangle^{-1})]
    +\O(\langle\omega\rangle^{-1})
\end{align*}
since $\widehat{ h}_1(\rho_\lambda,\lambda)b_j(\rho_\lambda,\lambda)= \O(\langle\omega\rangle^{-1})$
for $j=1,2$. Consequently, we obtain
$$
c_{1,3}(\lambda)=\frac{e^{i\frac{5}{4}\pi}\sqrt{\pi}}{\sqrt{2}}+\O(\langle\omega\rangle^{-1})
$$
and, analogously, one computes
$$
c_{2,3}(\lambda)=i \frac{e^{i\frac{5}{4}\pi}\sqrt{\pi}}{\sqrt{2}}+\O(\langle\omega\rangle^{-1}).
$$

Now, to patch together $\psi_4$ with the other solutions, we need some more considerations.
First, for $\widehat{\psi}_4(\rho,\lambda):=\widehat{\psi}_3(\rho, \lambda)\int_{\rho_\lambda}^\rho \widehat{\psi}_3(s,\lambda)^{-2} ds$ 
we compute that
\begin{align*}
W(\psi_4(.,\lambda),\widehat{\psi}_4(.,\lambda))
&=\frac{\psi_4(\rho_\lambda,\lambda)}{\widehat
                                                   \psi_3(\rho_\lambda,\lambda)}
  =\frac{\sqrt{\pi}e^{i\frac{5}{4}\pi}h_2(\rho_\lambda,\lambda)[1+\O(\langle\omega\rangle^{-1})]}{\sqrt{2}h_1(\rho_\lambda,\lambda)[1+\O(\langle\omega\rangle^{-1})]}
\\
&=\frac{\sqrt{\pi}e^{i\frac{5}{4}\pi}h_2(\rho_\lambda,\lambda)}{\sqrt{2}h_1(\rho_\lambda,\lambda)}+\O(\langle\omega\rangle^{-1})
\end{align*}
and so, as
$W(\psi_3(.,\lambda),\widehat{\psi}_4(.,\lambda))=\frac{\sqrt{\pi}e^{i\frac{5}{4}\pi}}{\sqrt{2}}$
and $W(\psi_4(.,\lambda),\psi_3(.,\lambda))=2i$, we see that
\begin{align*}
\psi_4(\rho,\lambda)=\left(\frac{h_2(\rho_\lambda,\lambda)}{h_1(\rho_\lambda,\lambda)}+\O(\langle\omega\rangle^{-1})\right)\psi_3(\rho,\lambda)-\frac{2\sqrt{2}ie^{-i\frac{5}{4}\pi}}{\sqrt{\pi}}\widehat{\psi}_4(\rho,\lambda)
\end{align*}
Lastly, we have that
\begin{align*}
W(\psi_j(.,\lambda),\widehat{\psi}_4(.,\lambda))=\frac{\sqrt{\pi}e^{i\frac{5}{4}\pi}b_j(\rho_\lambda,\lambda)}{\sqrt{2}h_1(\rho_\lambda,\lambda)}+\O(\langle\omega\rangle^{-1})
\end{align*}
for $j=1,2$ and the
 claim follows by putting everything together.
\end{proof}
Analogously we can patch together the solutions of the free equation. To this end, let $\psi_{\mathrm{f}_1}$ and $\psi_{\mathrm{f}_2}$ by the solutions obtained from Lemma \ref{Besselsol} in the case $V=0$ and, for notational convenience, let $h_1=\psi_{\mathrm{f}_3}$ and $h_2=\psi_{\mathrm{f}_4}$.
\begin{lem}\label{connectioncoeffree}
On $[\rho_\lambda,\widehat{\rho}_\lambda]$, the solutions $\psi_{\mathrm{f}_3}$ and $\psi_{\mathrm{f}_4}$ have the representations
\begin{align*}
\psi_{\mathrm{f}_3}(\rho,\lambda) &= c_{\mathrm{f}_{1,3}}(\lambda)\psi_{\mathrm{f}_1}(\rho,\lambda)+ c_{\mathrm{f}_{2,3}}(\lambda)\psi_{\mathrm{f}_2}(\rho,\lambda)\\
\psi_{\mathrm{f}_4}(\rho,\lambda) &= c_{\mathrm{f}_{1,4}}(\lambda)\psi_{\mathrm{f}_1}(\rho,\lambda)+ c_{\mathrm{f}_{2,4}}(\lambda)\psi_{\mathrm{f}_2}(\rho,\lambda)
\end{align*}
with
\begin{align*}
	c_{\mathrm{f}_{1,3}}(\lambda)=&\frac{e^{i\frac{5}{4}\pi}\sqrt{\pi}}{\sqrt{2}}+\O(\langle\omega\rangle^{-1})\\
	c_{\mathrm{f}_{2,3}}(\lambda)=&i \frac{e^{i\frac{5}{4}\pi}\sqrt{\pi}}{\sqrt{2}}+\O(\langle\omega\rangle^{-1})
\end{align*}
and
\begin{align*}
c_{\mathrm{f}_{1,4}}(\lambda)=&\frac{e^{i\frac{5}{4}\pi}\sqrt{\pi}}{\sqrt{2}}\frac{h_2(\rho_\lambda,\lambda)}{h_1(\rho_\lambda,\lambda)}+i\pi\frac{b_2(\rho_\lambda,\lambda)}{h_1(\rho_\lambda,\lambda)}+\O(\langle\omega\rangle^{-1})=\O(\langle\omega\rangle^{0})
\\
c_{\mathrm{f}_{2,4}}(\lambda)=&i\frac{e^{i\frac{5}{4}\pi}\sqrt{\pi}}{\sqrt{2}}\frac{h_2(\rho_\lambda,\lambda)}{h_1(\rho_\lambda,\lambda)}-i\pi\frac{b_1(\rho_\lambda,\lambda)}{h_1(\rho_\lambda,\lambda)}+\O(\langle\omega\rangle^{-1})=\O(\langle\omega\rangle^{0}).
\end{align*}
\end{lem}                                                       
Next, let $\chi: [0,1]\times \{z\in \C:0\leq \Re z\leq \frac{1}{4}\} \to
[0,1]$, $\chi_\lambda(\rho):=\chi(\rho,\lambda)$, be a smooth cutoff function that satisfies
$\chi_\lambda(\rho)=1$ for $\rho \in [0,\rho_\lambda]$, $
\chi_\lambda(\rho)=0$ for $\rho \in [\widehat{\rho}_\lambda,1] $, and
$|\partial_\rho^k\partial_\omega^\ell \chi_\lambda(\rho)|\leq
C_{k,\ell}\langle\omega\rangle^{k-\ell}$ for $k,\ell\in\mathbb N_0$. 
We then define two solutions of Eq.~\eqref{no first order} as
\begin{align*}
v_1(\rho,\lambda):=&\chi_\lambda(\rho)[c_{1,4}(\lambda)\psi_1(\rho,\lambda)+c_{2,4}(\lambda)\psi_2(\rho,\lambda)]
+\left(1-\chi_\lambda(\rho)\right)\psi_4(\rho,\lambda)\\
v_2(\rho,\lambda):=& \chi_\lambda(\rho)[c_{1,3}(\lambda)\psi_1(\rho,\lambda)+c_{2,3}(\lambda)\psi_2(\rho,\lambda)]
+\left(1-\chi_\lambda(\rho)\right)\psi_3(\rho,\lambda)
\end{align*}
and note that an evaluation at $\rho=1$ yields 
$$
W(v_1,v_2)=W(\psi_4,\psi_3)=2i.
$$

\subsection{The original equation}
We now turn back to our original ODE \eqref{resolventeq}. To obtain solutions of Eq.~\eqref{greenode}, with $F_\lambda=0$, we set
\begin{equation} \label{transformation}
u_j(\rho,\lambda)=\rho^{-\frac{5}{2}}(1-\rho^2)^{\frac{1}{4}-\frac{\lambda}{2}}v_j(\rho,\lambda)
\end{equation}
for $j=1,2$.
\begin{lem}\label{asymptotic}
The solutions $u_1$ and $u_2$ are of the form
\begin{align}
u_1(\rho,\lambda)&=
\rho^{-\frac{5}{2}}(1-\rho^2)^{\frac{1}{4}-\frac{\lambda}{2}}h_2(\rho,\lambda)\left[1+(1-\rho)\O(\langle\omega\rangle^{-1})+\O(\rho^{0}(1-\rho)^2\langle\omega\rangle^{-1})\right]
\nonumber\\
&=
\rho^{-\frac{5}{2}}\frac{(1+\rho)^{\frac{3}{2}-\lambda}}{\sqrt{ a(\lambda)}}\left[1+(1-\rho)\O(\langle\omega\rangle^{-1})+\O(\rho^{-1}(1-\rho)^2\langle\omega\rangle^{-1})\right]\nonumber
\\
u_2(\rho,\lambda)&=\rho^{-\frac{5}{2}}(1-\rho^2)^{\frac{1}{4}-\frac{\lambda}{2}}h_1(\rho,\lambda)\left[1+(1-\rho)\O(\langle\omega\rangle^{-1})+\O(\rho^{0}(1-\rho)^2\langle\omega\rangle^{-1})\right]\nonumber
\\
&= 
\rho^{-\frac{5}{2}}\frac{(1-\rho)^{\frac{3}{2}-\lambda}}{\sqrt{ a(\lambda)}}\left[1+(1-\rho)\O(\langle\omega\rangle^{-1})+\O(\rho^{-1}(1-\rho)^2\langle\omega\rangle^{-1})\right]\nonumber
\end{align}
for all $\rho \geq
\widehat{\rho}_\lambda=\min\{\frac{1}{2}(\rho_0+1),\frac{2r}{|a(\lambda)|}\}$
and all $\lambda\in \C$ with $\Re\lambda\in [0,\frac14]$.
\end{lem}
\begin{proof}
The claimed forms of $u_1$ and $u_2$  follow immediately from Lemma \ref{Hankelsol} and the transformation \eqref{transformation}.
\end{proof}
As a consequence of this lemma, we obtain that for $\lambda \in \C$
with $\Re\lambda\in [0,\frac{1}{4}]$, $c_{2,4}(\lambda)$ does not
vanish.

\begin{lem}
  \label{lem:c24}
  We have $c_{2,4}(\lambda)\not=0$ for all $\lambda\in \C$ with
  $\Re\lambda\in [0,\frac14]$.
\end{lem}

\begin{proof}
We argue by contradiction. Suppose there exists a $\lambda$ such that
$c_{2,4}(\lambda)=0$. Then, by definition,
\[
  u_1(\rho,\lambda)=c_{1,4}(\lambda)\rho^{-\frac52}(1-\rho^2)^{\frac14-\frac{\lambda}{2}}\psi_1(\rho,\lambda)=\O(\rho^0) \]
for $\rho\in [0,\rho_\lambda]$ and hence, $u_1(.,\lambda)\in
H^2(\B^6_1)$. This implies that $\lambda\in \sigma_p(\Lf)$, a
contradiction to Lemma \ref{lem:spec}.
\end{proof}
 Based on Lemma \ref{lem:c24}, we define a third solution to
 Eq.~\eqref{greenode} by
 \[ u_0(\rho,\lambda):=u_2(\rho,\lambda)-\frac{c_{2,3}(\lambda)}{c_{2,4}(\lambda)}u_1(\rho,\lambda) \] and remark that
$$
W(u_1(.,\lambda),u_0(.,\lambda))(\rho)=2i\rho^{-5}(1-\rho^2)^{\frac{1}{2}-\lambda}.
$$

Note that neither $u_1$ nor $u_0$ is in $H^2(\B_1^6)$. The function $u_1$ fails to be in the Sobolev space due to the divergent behavior of $Y_2$ at 0, while $u_0''$ is not square integrable at 1.
Hence, to find a solution of Eq.~\eqref{resolventeq}, with $0\leq \Re \lambda\leq \frac{1}{4}$ and $F_\lambda \in C^2(\overline{\B^6_1})$, which is in $H^2(\B^6_1)$, we make the ansatz
\begin{align*}
u(\rho,\lambda)&= \left[c-U_1(\rho_1,\lambda)F_\lambda(\rho_1)\right] u_0(\rho,\lambda) +u_0(\rho,\lambda) \int_\rho^{\rho_1}\frac{u_1(s,\lambda)}{W(u_1(.,\lambda),u_0(.,\lambda))(s)}\frac{F_\lambda(s)}{1-s^2} d s
\\
&\quad +u_1(\rho,\lambda) \int_0^\rho\frac{u_0(s,\lambda)}{W(u_1(.,\lambda),u_0(.,\lambda))(s)}\frac{F_\lambda(s)}{1-s^2} d s,
\end{align*}
with $c \in \C$, $\rho_1\in [0,1)$,
and $$U_1(\rho,\lambda):=\int_0^\rho\frac{u_1(s,\lambda)}{W(u_1(.,\lambda),u_0(.,\lambda))(s)(1-s^2)} ds.$$
Since $F_\lambda$ is in $C^2(\overline{\B^6_1})$ by assumption, we can do an integration by parts to obtain
\begin{align*}
 \int_\rho^{\rho_1}\frac{u_1(s,\lambda)}{W(u_1(.,\lambda),u_0(.,\lambda))(s)}\frac{F_\lambda(s)}{1-s^2} d s&= U_1(\rho_1,\lambda)F_\lambda(\rho_1)-U_1(\rho,\lambda)F_\lambda(\rho)
 \\
 &\quad-\int_\rho^{\rho_1}U_1(s,\lambda)F'_\lambda(s)ds
\end{align*}
and so
\begin{align*}
u(\rho,\lambda)&=cu_0(\rho,\lambda) -u_0(\rho,\lambda)\left(U_1(\rho,\lambda)F_\lambda(\rho)+\int_\rho^{\rho_1}U_1(s,\lambda)F'_\lambda(s) ds\right)\\
&\quad +u_1(\rho,\lambda) \int_0^\rho\frac{u_0(s,\lambda)}{W(u_1(.,\lambda),u_0(.,\lambda))(s)}\frac{F_\lambda(s)}{1-s^2} d s.
\end{align*}

We continue by setting  
\[ U_0(\rho,\lambda):=\int_0^\rho\frac{u_0(s,\lambda)}{W(u_1(.,\lambda),u_0(.,\lambda))(s)(1-s^2)} ds\] and also perform an integration by parts in the second integral, which yields
\begin{align} \label{u intermediate}
u(\rho,\lambda)&= u_0(\rho,\lambda)\left(c-U_1(\rho,\lambda)F_\lambda(\rho)-\int_\rho^{\rho_1}U_1(s,\lambda)F'_\lambda(s)ds \right) \nonumber
\\
&\quad +u_1(\rho,\lambda)\left(U_0(\rho,\lambda)F_\lambda(\rho)- \int_0^\rho U_0(s,\lambda)F'_\lambda(s) d s\right).
\end{align}
By standard ODE theory we obtain that $u \in C^2((0,1))$ and so we only have to check the behavior at the end points.
To proceed with that task, we will need the asymptotic behavior of $U_j(\rho,\lambda)$ as $\rho$ tends to $1$.
We have that 
\begin{align*}
U_1(\rho,\lambda)=\frac{1}{2i}\int_0^\rho u_1(s,\lambda)s^5(1-s^2)^{-\frac{3}{2}+\lambda} ds
\end{align*}
and hence by using l'Hospitals rule and Lemma \ref{asymptotic} we obtain the asymptotic behavior
\begin{equation}\label{U_1 asy}
U_1(\rho,\lambda)\sim -\frac{(1-\rho^2)^{-\frac{1}{2}+\lambda}}{4 i(-\frac{1}{2}+\lambda)}u_1(1,\lambda)\sim -\frac{(1-\rho)^{-\frac{1}{2}+\lambda}}{2i(-\frac{1}{2}+\lambda)\sqrt{a(\lambda)}}
\end{equation}
as $\rho $ tends to $1$.
For $U_0$, we similarly obtain
\begin{align}\label{U_0 asy}
U_0(\rho,\lambda)\sim(1-\rho)^{-\frac{1}{2}+\lambda}\O(\langle\omega\rangle^{-\frac32})
\end{align}
as $\rho \to 1$, since $u_0=u_2-\frac{c_{2,3}}{c_{2,4}}u_1$.

Using the symbol forms of $\psi_1$ and $\psi_2$ given in Lemma
\ref{Besselsol} it is straightforward to check that $u(.,\lambda)\in L^2(\B^6_{\frac{1}{2}})$. 
To study $u$ near the endpoint $1$ we infer that
\begin{align*}
-u_0(\rho,\lambda)U_1(\rho,\lambda)+u_1(\rho,\lambda)U_0(\rho,\lambda)=-u_2(\rho,\lambda)U_1(\rho,\lambda)+u_1(\rho,\lambda)U_2(\rho,\lambda)
\end{align*}
where
\[ U_2(\rho,\lambda):=\int_0^\rho
  \frac{u_2(s,\lambda)}{W(u_1(.,\lambda),
    u_0(.,\lambda))(s)(1-s^2)}ds. \]
Consequently,
\begin{align*}
u(\rho,\lambda)=&u_0(\rho,\lambda)\left(c-\int_\rho^{\rho_1}U_1(s,\lambda)F'_\lambda(s)ds \right)-u_2(\rho,\lambda)U_1(\rho,\lambda)F_\lambda(\rho)
\\
&+u_1(\rho,\lambda)\left(U_2(\rho,\lambda)F_\lambda(\rho)- \int_0^\rho U_0(s,\lambda)F'_\lambda(s) d s\right)
\end{align*}
and from this representation one can easily verify that
$u(.,\lambda)\in L^2(\B^6_1)$, by simply plugging in the forms of $u_1$ and $u_2$ stated in Lemma \ref{asymptotic}.
 Next, from Eq.~\eqref{u intermediate} we compute that
\begin{align*}
u'(\rho,\lambda)&= u_0'(\rho,\lambda)\left(c-U_1(\rho,\lambda)F_\lambda(\rho)-\int_\rho^{\rho_1}U_1(s,\lambda)F'_\lambda(s)ds \right) 
\\
&\quad+u_1'(\rho,\lambda)\left(U_0(\rho,\lambda)F_\lambda(\rho)- \int_0^\rho U_0(s,\lambda)F'_\lambda(s) d s\right)
\end{align*}
and as above one establishes that $u'(.,\lambda)\in L^2(\B^6_1)$ .
Lastly, the second derivative of $u$ is given by
\begin{align*}
u''(\rho,\lambda)&= u_0''(\rho,\lambda)\left(c-U_1(\rho,\lambda)F_\lambda(\rho)-\int_\rho^{\rho_1}U_1(s,\lambda)F'_\lambda(s)ds \right) \nonumber
\\
&\quad+u_1''(\rho,\lambda)\left(U_0(\rho,\lambda)F_\lambda(\rho)- \int_0^\rho U_0(s,\lambda)F'_\lambda(s) d s\right)
\\
&\quad-u_0'(\rho,\lambda)U_1'(\rho,\lambda)F_\lambda(\rho)+u_1'(\rho,\lambda)U_0'(\rho,\lambda)F_\lambda(\rho).
\end{align*}
One can again make use of the symbol forms of $\psi_1$ and $\psi_2$ to promptly verify that $u''(.,\lambda)\in L^2(\B^6_{\frac{1}{2}})$. To check the behavior at $\rho =1$ we use that
\[-u_0'(\rho,\lambda)U_1'(\rho,\lambda)F_\lambda(\rho)+u_1'(\rho,\lambda)U_0'(\rho,\lambda)F_\lambda(\rho)=-\frac{F_\lambda(\rho)}{1-\rho^2}
\]
and so we can rewrite $u''$ in similar fashion as $u$ before as
\begin{align*}
u''(\rho,\lambda)&= u_0''(\rho,\lambda)\left(c-\int_\rho^{\rho_1}U_1(s,\lambda)F'_\lambda(s)ds \right)-u_2''(\rho,\lambda)U_1(\rho,\lambda)F_\lambda(\rho)
\\
&\quad+u_1''(\rho,\lambda)\left(U_2(\rho,\lambda)F_\lambda(\rho)- \int_0^\rho U_0(s,\lambda)F'_\lambda(s) d s\right)
-\frac{F_\lambda(\rho)}{1-\rho^2}.
\end{align*} 
Taking the limit $\rho_1\to 1$ now yields 
\begin{align*}
\lim_{\rho_1\to 1} -u_0''(\rho,\lambda)\int_\rho^{\rho_1}U_1(s,\lambda)F'_\lambda(s)ds =-u_0''(\rho,\lambda)\int_\rho^{1}U_1(s,\lambda)F'_\lambda(s)ds 
\end{align*}
and one readily checks that this is a square integrable expression near 1 by using de l'Hospitals rule and Lemma \ref{asymptotic}.
Further, using Lemma \ref{asymptotic} and the asymptotics \eqref{U_0 asy} and \eqref{U_1 asy} it follows that $cu_1''(\rho,\lambda)$ and $u_1''(\rho,\lambda)[U_2(\rho,\lambda)-\int_0^\rho U_0(s,\lambda)F'_\lambda(s) d s]$ are square integrable near $1$. Consequently, we are left with investigating the terms
\begin{align*}
u_2''(\rho,\lambda)[c-U_1(\rho,\lambda)F_\lambda(\rho)]-\frac{F_\lambda(\rho)}{1-\rho^2}.
\end{align*}
To see that the two first order poles cancel out, we remark that
\[
\partial_\rho(1-\rho^2)^{-\frac{1}{2}+\lambda}=-\frac{(-1+2\lambda)\rho}{(1-\rho^2)^{\frac{3}{2}-\lambda}}
\]
which we use to perform yet one more integration by parts, which yields
\begin{align*}
U_1(\rho,\lambda)=&-\frac{\rho^{4}u_1(\rho,\lambda)}{2i(1-2\lambda)(1-\rho^2)^{\frac{1}{2}-\lambda}}+l_1(\lambda)+\int_0^\rho\frac{\partial_s
                    (s^{4}u_1(s,\lambda))}{2i(1-2\lambda)(1-s^2)^{\frac{1}{2}-\lambda}} ds
\end{align*}
with $$l_1(\lambda)=\lim_{\rho \to 0}\frac{\rho^{4}u_1(\rho)}{2i(1-2\lambda)(1-\rho^2)^{\frac{1}{2}-\lambda}}.$$
Note that Lemma \ref{asymptotic}
implies
\begin{align*}
u_2''(\rho,\lambda)\sim&\frac{(\frac{3}{2}-\lambda)(\frac{1}{2}-\lambda)(1-\rho)^{-\frac{1}{2}-\lambda}}{\sqrt{a(\lambda)}}\
\end{align*}
as $\rho \rightarrow 1$. 
Using this, we obtain that 
\begin{equation}
-\frac{\rho^{4}u_1(\rho,\lambda) u_2''(\rho,\lambda)}{2i(1-2\lambda)(1-\rho^2)^{\frac{1}{2}-\lambda}}\sim\frac{1}{2(1-\rho)}\sim\frac{1}{1-\rho^2}
\end{equation}
as $\rho\to 1$ and therefore the poles cancel out exactly.
Observe now that $\rho^2Y_2(\rho)\in C^\infty([0,1])$ and so the expression
$$
\frac{\partial_s(s^{4}u_1(s,\lambda))}{2i(1-2\lambda)(1-s^2)^{\frac{1}{2}-\lambda}}
$$
is integrable on $[0,1]$.
Thus, if we make the choice 
$$
c=c(F_\lambda,\lambda)=-F_\lambda(1)\left(
  l_1(\lambda)+\int_0^1\frac{\partial_s (s^{4}u_1(s,\lambda))} {2i(1-2\lambda)(1-s^2)^{\frac{1}{2}-\lambda}} ds \right),
$$
once more employing Lemma \ref{asymptotic} yields
\begin{align*}
\left|\lim_{\rho \to 1} u_2''(\rho,\lambda)[c(F_\lambda,\lambda)-U_1(\rho,\lambda)F_\lambda(\rho)]-\frac{F_\lambda(\rho)}{1-\rho^2}\right|\lesssim 1.
\end{align*} Summarising, the unique function $u(.,\lambda)\in H^2(\B^6_1)$, which solves Eq.~\eqref{resolventeq} for \\ $0 \leq \Re\lambda \leq \frac14$ and $F_\lambda\in C^2(\overline{\B^6_1})$, is given by
\begin{align*}
u(\rho,\lambda)&= u_0(\rho,\lambda)\left(c(F_\lambda,\lambda)-U_1(\rho,\lambda)F_\lambda(\rho)-\int_\rho^{1}U_1(s,\lambda)F'_\lambda(s)ds \right)
\\
&\quad +u_1(\rho,\lambda)\left(U_0(\rho,\lambda)F_\lambda(\rho)- \int_0^\rho U_0(s,\lambda)F'_\lambda(s) d s\right)
\end{align*}
and we have shown the following Lemma.
\begin{lem}
Let $f\in C^2(\overline{\B^6_1})$ and $0\leq\Re(\lambda)\leq \frac{1}{4}$. Then the function $\mathcal{R}(f)$ defined as
\begin{align*}
\mathcal{R}(f)(\rho,\lambda)&:= u_0(\rho,\lambda)\left(c(f,\lambda)-U_1(\rho,\lambda)f(\rho)-\int_\rho^{1}U_1(s,\lambda)f'(s)ds \right)\nonumber
\\
&\quad +u_1(\rho,\lambda)\left(U_0(\rho,\lambda)f(\rho)- \int_0^\rho U_0(s,\lambda)f'(s) d s\right)
\end{align*}
with 
\[ U_j(\rho,\lambda):=\int_0^\rho\frac{u_j(s,\lambda)}{W(u_1(.,\lambda),u_0(.,\lambda))(s)(1-s^2)} ds\]
 for $j=0,1$ and
\[
c(f,\lambda)=-f(1)\left( \lim_{\rho \to 0}\frac{\rho^{4}u_1(\rho,\lambda)}{2i(1-2\lambda)(1-\rho^2)^{\frac{1}{2}-\lambda}}+\int_0^1\frac{\partial_s(s^{4}u_1(s,\lambda))} {2i(1-2\lambda)(1-s^2)^{\frac{1}{2}-\lambda}} ds \right),
\]
is the unique function $u \in H^2(\B^6_1)\cap C^\infty(\B^6_1)$ that
solves the equation 
\begin{equation}\label{greenode2}
-(1-\rho^2)u''(\rho)+\left(2(\lambda+2)\rho-\frac{5}{\rho}\right)u'(\rho)+\left((\lambda+2)(\lambda+1)-\frac{48}{(\rho^2+2)^2}\right)u(\rho)=f(\rho).
\end{equation}
\end{lem}
To continue, let $u_{\mathrm{f}_1}$, $u_{\mathrm{f}_2}$ be the solutions of Eq. \eqref{greenode}, that arise by repeating the previous constructions in the special case $V=0$. We would like to copy our construction of $\mathcal{R}(F_\lambda)$ with these solutions. However, to obtain the solution $u_{\mathrm{f}_0}$, we first have to make sure that $c_{\mathrm{f}_{2,4}}$ does not vanish for $0 \leq \Re\lambda\leq \frac14$. This is ensured by the next Lemma.
\begin{lem}\label{lem:roots of free c24}
We have $c_{\mathrm{f}_{2,4}}(\lambda)\not=0$ for all $\lambda\in \C$ with
  $\Re\lambda\in [0,\frac14]$.
\end{lem}
\begin{proof}
To begin, we remark that whenever $c_{\mathrm{f}_{2,4}}$ vanishes for $0\leq \Re \lambda\leq \frac14$, then $u_{\mathrm{f}_1}(.,\lambda)\in H^2(\B^6_1)$. To see this, one can argue in the same way as in the proof of Lemma \ref{asymptotic}. 
As a consequence, $u_{\mathrm{f}_1}(.,\lambda)\in H^2(\B^6_1)$ implies
that $\lambda$ is an eigenvalue of the operator $\Lf_0$. But, as
$\Lf_0$ generates a contraction semigroup, this cannot happen for $\Re\lambda>0$ and so we only have to check the claim on the imaginary axis.
By construction, $u_{\mathrm{f}_1}(.,i\omega)$ is a solution of the equation
\begin{equation}\label{Eq:spectralfree}
-(1-\rho^2)u''(\rho)+\left(2(i\omega+2)\rho-\frac{5}{\rho}\right)u'(\rho)+(i\omega+2)(i\omega+1)u(\rho)=0.
\end{equation}
Assume now that $c_{\mathrm{f}_{2,4}}(i\omega)=0$ for some $\omega\in \R$.
Then, as the Frobenius indices of Eq.~\eqref{Eq:spectralfree} are given by $\{0,-4\}$ at $0$ and $\{0,\frac{3}{2}-i \omega\}$ at 1, $u_{\mathrm{f}_1}(.,i\omega)$ is in fact a smooth function on $[0,1]$.
Setting $z=\rho^2$ and $v(z)=u(\sqrt{z})$ transforms equation \eqref{Eq:spectralfree} into
\begin{align*}
z(1-z)v''(z)+\left(3-\left(i\omega +\frac{5}{2}\right)z\right)v'(z)-\frac{(i\omega+2)(i\omega+1)}{4}v(z)=0.
\end{align*} 
This is now a hypergeometric differential equation, i.e., it is of the form
\begin{align*}
z(1-z)v''(z)+(c-(a+b+1)z)v'(z)-abv(z)=0
\end{align*}
with $a=\frac{1+i\omega}{2}, b=1+\frac{i\omega}{2},$ and $c=3$.
Near $0$, this equation therefore has a fundamental system given by 
\begin{align*}
f_1(z):=\, _2F_1(a,b;c;z)
\end{align*}
and a second solution which diverges at $\rho=0$,
while near 1, a fundamental system  is given by 
\begin{align*}
f_2(z):&=\,_2F_1(a,b;a+b+c-1;1-z)\\
f_3(z):&= (1-z)^{\frac{3}{2}-i\omega} \,_2F_1(c-a,c-b;c-a-b+1;1-z),
\end{align*}
see \cite{Olv97}.
For a solution of Eq. $\eqref{Eq:spectralfree}$ to be smooth, it has to be a multiple of $f_1$.
Let $c_2$ and $c_3$ be such that
\begin{align*}
f_1(z)=c_2 f_2(z)+c_3 f_3(z)
\end{align*}
and note that $c_3$ has to vanish in order for $f_1$ to be smooth. Thanks to the explicit connection formula, $c_3$ is given by
\begin{align*}
c_3=\frac{\Gamma(c)\Gamma(a+b-c+1)}{\Gamma(a) \Gamma(b)}
\end{align*}
where $\Gamma$ denotes the Gamma function. So, for $c_3$ to vanish, $a$ or $b$ would have to be a pole of $\Gamma$, which is not the case.
\end{proof} 
Thanks to Lemma \ref{lem:roots of free c24}, we can define the solution $u_{\mathrm{f}_0}$ analogously to $u_0$ as
\[
u_{\mathrm{f}_0}:=u_{\mathrm{f}_2}-\frac{c_{\mathrm{f}_{2,3}}}{c_{\mathrm{f}_{2,4}}}
u_{\mathrm{f}_1}. \]
Furthermore,
$$
W(u_{\mathrm{f}_1}(.,\lambda),u_{\mathrm{f}_0}(.,\lambda))=W(u_1(.,\lambda),u_0(.,\lambda))
$$
and so, we can copy our construction of $\mathcal{R}(F_\lambda)(\rho,\lambda)$ with the solutions $u_{\mathrm{f}_j}$. This leads to a solution $\mathcal{R}_{\mathrm{f}}(F_\lambda)(\rho,\lambda)$ of Eq.~\eqref{greenode} with $V=0$ which is of the form
\begin{align*}
\mathcal{R}_{\mathrm{f}}(F_\lambda)(\rho,\lambda):&= u_{\mathrm{f}_0}(\rho,\lambda)\left(c_{\mathrm{f}}(F_\lambda,\lambda)-U_{\mathrm{f}_1}(\rho,\lambda)F_\lambda(\rho)-\int_\rho^{1}U_{\mathrm{f}_1}(s,\lambda)F'_\lambda(s)ds \right)
\\
&\quad+u_{\mathrm{f}_1}(\rho,\lambda)\left(U_{\mathrm{f}_0}(\rho,\lambda)F_\lambda(\rho)- \int_0^\rho U_{\mathrm{f}_0}(s,\lambda)F'_\lambda(s) d s\right),
\end{align*}
where $ c_{\mathrm{f}},$ $U_{\mathrm{f}_0}$, and $U_{\mathrm{f}_1}$ are defined just as their ``non-free'' counterparts, by replacing the corresponding $u_j$ by $u_{\mathrm{f}_j}.$
\section{Strichartz estimates}
Using Laplace inversion, we are now able to explicitly write down $[\Sf(\tau)(\I-\Pf) \ff]_1=[\Sf(\tau)\widetilde{\ff}]_1$ for any $\ff \in C^3\times C^2(\overline{\B^6_1})$ as 
\begin{align}\label{semirep}
[\Sf(\tau) \widetilde{\ff}]_1(\rho)=[\Sf_0(\tau) \widetilde{\ff}]_1(\rho)
+\frac{1}{2\pi i}\lim_{N \to \infty} \int_{\varepsilon-i N}^{\varepsilon+ i N}e^{\lambda\tau}[\mathcal{R}(F_\lambda)(\rho,\lambda)-\mathcal{R}_{\mathrm{f}}(F_\lambda)(\rho,\lambda)] d\lambda,
\end{align}
with $\varepsilon >0$ and $F_\lambda(\rho)=\rho\widetilde
f_1'(\rho)+(\lambda+2)\widetilde f_1(\rho)+\widetilde f_2(\rho)$, since $[\mathbf R_{\mathbf L}(\lambda)\widetilde
{\mathbf f}]_1=\mathcal R(F_\lambda)(.,\lambda)$ and $[\mathbf R_{\mathbf L_0}(\lambda)\widetilde
{\mathbf f}]_1=\mathcal R_{\mathrm f}(F_\lambda)(.,\lambda)$.
Therefore, our next job will be to estimate the above integral term in the limit case $\varepsilon\to 0$. To this end, we will now state a couple of preliminary lemmas which will be crucial in accomplishing that task.
\subsection{Preliminary and technical Lemmas}
The first set of lemmas will be concerned with oscillatory integrals.
\begin{lem}\label{osci1}
	Let $\alpha >0$. Then 
	$$
	\left|\int_\R e^{i a\omega}\O(\langle\omega\rangle^{-(1+\alpha)}) d \omega \right|\lesssim \langle a\rangle^{-2},
	$$
	 for any $a\in \R$. 
\end{lem}

\begin{proof}
	Since the integral is absolutely convergent the claim follows by doing two integrations by parts.
\end{proof}
\begin{lem}\label{osci2}
	Let $\alpha \in (0,1)$. Then
	\begin{align*}
	\left|\int_\R e^{i a\omega}\O(\langle\omega\rangle^{-\alpha}) d \omega \right|\lesssim |a|^{\alpha-1}\langle a\rangle^{-2}
	\end{align*}
	holds for $a\in \R\setminus\{0\}$.
\end{lem}
\begin{proof}
	See Lemma 4.2 in \cite{DonRao20}.
\end{proof}
\begin{lem}\label{osci3}
We have
	\begin{align*}
	\left|\int_\R e^{i a\omega}(1-\chi_\lambda(\rho))\O\left(\rho^{-n}\langle\omega\rangle^{-(n+1)}\right) d \omega \right|\lesssim \langle a\rangle^{-2}
	\end{align*}
	for all $n\geq 1$, $\rho \in (0,1)$, and $a\in \R$.
\end{lem}
\begin{proof}
	This can be proven in the same manner as Lemma 4.3 in \cite{DonRao20}
\end{proof}
Finally, we need one more lemma on oscillatory integrals.
\begin{lem}\label{osci4}
We have
	\begin{align*}
	\left|\int_\R e^{i a\omega}(1-\chi_\lambda(\rho))\O\left(\rho^{-n}\langle\omega\rangle^{-n}\right)d\omega\right|\lesssim|a|^{-1}\langle a\rangle^{-2}
	\end{align*}
	for any $n \geq 2$, $\rho \in (0,1)$, and $a\in \R \setminus\{0\}$.
\end{lem}

\begin{proof}
This can be proven as Lemma 4.4 in \cite{DonRao20}.
\end{proof}
We will also make use of the following technical Lemmas.
\begin{lem}\label{teclem1}
The estimates
\begin{align*}
\left\||.|^{2-\frac{1}{12}}f\right\|_{L^{12}(\B^6_1)}\lesssim \|f\|_{H^1(\B^6_1)}
\end{align*}
and 
\begin{align*}
\left\|\int_\rho^1 s^{2-\frac{1}{12}}f'(s)ds\right\|_{L^{12}_\rho(\B^6_1)}\lesssim \|f\|_{H^1(\B^6_1)}
\end{align*}
hold true for all $f\in C^1(\overline{\B^6_1})$.
Further, also the estimate
\begin{align*}
\left\||.|^{1-\frac{1}{12}}f\right\|_{L^{12}(\B^6_1)}\lesssim \|f\|_{H^2(\B^6_1)}
\end{align*}
is true for all $f\in C^2(\overline{\B^6_1})$.
\end{lem}
\begin{proof}
We readily calculate
\begin{align*}
\||.|^{2-\frac{1}{12}}f\|_{L^{12}(\B^6_1)}&=\left(\int_0^1\left|s^{2-\frac{1}{12}} f(s)\right|^{12}s^5 ds \right)^{\frac{1}{12}}=\left(\int_0^1\left|s^2 f(s)\right|^{12}s^4 ds \right)^{\frac{1}{12}}
\\
&\lesssim\left(\int_0^1\left|s^2 f(s)\right|^{12}s ds \right)^{\frac{1}{12}} 
=\||.|^2 f\|_{L^{12}(\B^2_1)}
\lesssim\||.|^2 f\|_{H^1(\B^2_1)}
\\
&\lesssim \left(\int_0^1 |f(s)|^2s^3+ |f'(s)|^2 s^5 ds \right)^{\frac{1}{2}}
\lesssim
\| f\|_{H^1(\B^6_1)}
\end{align*}
by Sobolev embedding and Lemma \ref{helplemoverrho}.
To show the second bound, we argue similarly to obtain 
\begin{align*}
\left\|\int_\rho^1 s^{2-\frac{1}{12}}f'(s)ds\right\|_{L^{12}_\rho(\B^6_1)}&=
\left\|\rho^{2} \rho^{-2+\frac{1}{3}}\int_\rho^1 s^{2-\frac{1}{12}}f'(s)ds\right\|_{L^{12}_\rho(\B^2_1)}
\\
&\lesssim \left\|\rho^{-2+\frac{1}{3}}\int_\rho^1 s^{2-\frac{1}{12}}f'(s)ds\right\|_{H^1_\rho(\B^6_1)}.
\end{align*}
Further,
\begin{align*}
\left\|\rho^{-2+\frac{1}{3}}\int_\rho^1 s^{2-\frac{1}{12}}f'(s)ds\right\|_{L^2_\rho(\B^6_1)}&\lesssim \left\|\rho^{-3+\frac{1}{4}}\int_\rho^1 s^{\frac{5}{2}}|f'(s)|ds\right\|_{L^2_\rho(\B^6_1)}
\\
 &\lesssim \|f\|_{H^1(\B^6_1)}\||.|^{-3+\frac{1}{4}}\|_{L^2(\B^6_1)} 
\\
&\lesssim \|f\|_{H^1(\B^6_1)}
\end{align*}
and 
\begin{align*}
\left\|\rho^{-2+\frac{1}{3}}\int_\rho^1 s^{2-\frac{1}{12}}f'(s)ds\right\|_{\dot{H}^1_\rho(\B^6_1)}&\lesssim \left\|\rho^{-3+\frac{1}{3}}\int_\rho^1 s^{2-\frac{1}{12}}|f'(s)|ds\right\|_{L^2_\rho(\B^6_1)}+\|f\|_{H^1(\B^6_1)} 
\\
&\lesssim
\left\|\rho^{-3+\frac{1}{12}}\int_\rho^1 s^{2+\frac{1}{24}}|f'(s)|ds\right\|_{L^2_\rho(\B^6_1)}+\|f\|_{H^1(\B^6_1)}
\\
&\lesssim \|f\|_{H^1(\B^6_1)}\int_0^1 s^{-1+\frac{1}{12}} ds\,\| |.|^{-3+\frac{1}{12}}\|_{L^2(\B^6_1)}+\|f\|_{H^1(\B^6_1)}
\\
&\lesssim \|f\|_{H^1(\B^6_1)}
\end{align*}
where the third $\lesssim$ follows from the Cauchy-Schwarz inequality.
For the third one, we note that the calculation used to derive the first estimate together with Lemma \ref{helplemoverrho} implies
\begin{align*}
\||.|^{1-\frac{1}{12}}f\|_{L^{12}(\B^6_1)}\lesssim \||.|^{-1} f\|_{H^1(\B^6_1)}\lesssim \|f\|_{H^2(\B^6_1)}.
\end{align*}
\end{proof}
\begin{lem}\label{teclem2}
	Let $\alpha \in (0,1)$. Then we have the estimate
	\begin{align*}
	\int_0^1s^{\alpha-1}|a+\log(1\pm s)|^{-\alpha} ds\lesssim |a|^{-\alpha}
	\end{align*}
	for any $a\in \R\setminus\{0\}$.
\end{lem}
\begin{proof}
	We only prove the - case as the + case can be shown analogously. 
For $a<0$ the estimate
\begin{align*}
\left|a+ \log(1-s) \right|^{-\alpha} \leq |a|^{-\alpha}
\end{align*}  
holds for all $s\in [0,1]$ and so the claim follows.
For $a>0$  we change variables according to $s=1-e^{ax}$ and compute
	\begin{align*}
	\int_0^1s^{\alpha-1}|a+\log(1-s)|^{-\alpha} ds &= \int_{-\infty}^0(1-e^{a x})^{\alpha-1}|a+ax|^{-\alpha} a e^{a x} dx
	\\
	&\lesssim |a|^{1-\alpha}\int_{-\frac{1}{2}}^0(1-e^{a x})^{\alpha-1}e^{ax} dx 
	\\
	&\quad +|a|^{1-\alpha}(1-e^{-\frac a2})^{\alpha-1}e^{-\frac a2}\int_{-2}^{-\frac{1}{2}} |1+x|^{-\alpha}dx 
	\\
	&\quad +|a|^{1-\alpha}\int^{-2}_{-\infty}(1-e^{a x})^{\alpha-1}e^{a x} dx.
	\end{align*}
	The claimed estimate now follows from a straightforward computation and the two identities
\[
-\partial_x \frac{(1-e^{ax})^{\alpha}}{a\alpha}=(1-e^{ax})^{\alpha-1}e^{ax} 
\]
and
\[(1-e^{-\frac{a}{2}})^{\alpha-1}e^{-\frac{a}{2}} \lesssim a^{\alpha-1}
\]
for all $a>0$.
\end{proof}
\begin{lem}\label{teclem3}
The estimate
	\begin{align*}
	\int_0^1 s^{-\frac{1}{2}}\left|a\pm \frac{1}{2}\log(1-s^2) \right|^{-\frac{1}{2}} ds\lesssim |a|^{-\frac{1}{2}}
	\end{align*}
	holds for all $a\in \R\setminus\{0\}$.
\end{lem}
\begin{proof}
We again  only prove the + case as the - case follows likewise.
 For negative $a$ the estimate
\begin{align*}
\left|a+ \frac{1}{2}\log(1-s^2) \right|^{-\frac{1}{2}} \leq |a|^{-\frac{1}{2}}
\end{align*}  
holds for $s \in [0,1]$ and so the claim follows immediately. For positive $a$
we change variables according to
\[s=\sqrt{1-e^{2ax}}
\]
and obtain
\begin{align*}
\int_0^1 s^{-\frac{1}{2}}\left|a\pm \frac{1}{2}\log(1-s^2) \right|^{-\frac{1}{2}} ds &=|a|^{\frac{1}{2}}\int_{-\infty}^0(1-e^{2ax})^{-\frac{3}{4}}|1+x|^{-\frac{1}{2}}e^{2ax} dx 
\\
&\lesssim |a|^{\frac{1}{2}}\int_{-\frac{1}{2}}^0(1-e^{2ax})^{-\frac{3}{4}} e^{2ax}dx 
\\
&\quad+ |a|^{\frac{1}{2}}(1-e^{-a})^{-\frac{3}{4}}e^{-4a}\int_{-2}^{-\frac{1}{2}}|1+x|^{-\frac{1}{2}} dx 
\\
&\quad+ |a|^{\frac{1}{2}}\int_{-\infty}^{-2}(1-e^{2ax})^{-\frac{3}{4}}e^{2ax} dx .
\end{align*}
Now, the claimed estimate follows from the identities
\[
-\partial_x \frac{2}{a}(1-e^{2ax})^{\frac{1}{4}}=(1-e^{2ax})^{-\frac{3}{4}}e^{2ax} 
\]
and
\[(1-e^{-a})^{-\frac{3}{4}}e^{-4a}\lesssim a^{-1}.
\]
\end{proof}
\subsection{Oscillatory integral bounds} 
With these technical lemmas out of the way we begin bounding the integral term in Eq.~\eqref{semirep}.
To do so, we suppose  $f\in C^2(\overline{\B^6_1})$ and investigate the expression
\[
\frac{1}{2\pi i}\lim_{\varepsilon\to 0^+}\lim_{N \to \infty} \int_{\varepsilon-i N}^{\varepsilon+ i N}e^{\lambda\tau}[\mathcal{R}(f)(\rho,\lambda)-\mathcal{R}_{\mathrm{f}}(f)(\rho,\lambda)] d\lambda.
\]
Additionally, as $F_\lambda$ contains the term $\lambda f_1$ we also have to obtain control over the integral
\[
\frac{1}{2\pi i}\lim_{\varepsilon\to 0^+}\lim_{N \to \infty} \int_{\varepsilon-i N}^{\varepsilon+ i N} \lambda  e^{\lambda\tau}[\mathcal{R}( f)(\rho,\lambda)-\mathcal{R}_{\mathrm{f}}(f)(\rho,\lambda)] d\lambda.
\]
To make our lives easier, we break up the difference $\mathcal{R}(f)-\mathcal{R}_{\mathrm{f}}(f)$  into smaller parts starting with
\begin{align*}
V_1(f)(\rho,\lambda):=u_0(\rho,\lambda)c(f,\lambda)-u_{\mathrm{f}_0}(\rho,\lambda)c_{\mathrm{f}}(f,\lambda).
\end{align*}
However, we first need a better description of $c(f,\lambda)$ and $c(f,\lambda)-c_{\mathrm{f}}(f,\lambda)$
\begin{lem}\label{lem:c}
We have that
\begin{align*}
c(f,\lambda)=f(1)\O(\langle\omega\rangle^{-\frac{5}{2}})
\end{align*}
and
\begin{align*}
c(f,\lambda)-c_{\mathrm{f}}(f,\lambda)=f(1)\O(\langle\omega\rangle^{-\frac{7}{2}}).
\end{align*}
\end{lem}
\begin{proof}
Recall that $c(f,\lambda)$ was given by
\[
c(f,\lambda)=-f(1)\left( \lim_{\rho \to 0}\frac{\rho^{4}u_1(\rho)}{2i(1-2\lambda)(1-\rho^2)^{\frac{1}{2}-\lambda}}+\int_0^1\frac{\partial_s(s^{4}u_1(s,\lambda))} {2i(1-2\lambda)(1-s^2)^{\frac{1}{2}-\lambda}} ds \right).
\]
It is immediate that 
\[\lim_{\rho \to 0}\frac{\rho^{4}u_1(\rho)}{2i(1-2\lambda)(1-\rho^2)^{\frac{1}{2}-\lambda}}=\O\left(\langle\omega\rangle^{-3}\right),
\]
since $u_1$ is given by
\begin{align} \label{eq:u1 again}
u_1(\rho,\lambda)&=\rho^{-\frac{5}{2}}(1-\rho^2)^{\frac{1}{4}-\frac{\lambda}{2}}\chi_\lambda(\rho)\left(\O(\rho^{-1}\langle\omega\rangle^{-\frac{3}{2}})+c_{2,4}(\lambda) b_2(\rho,\lambda)[1+\O(\rho^2\langle\omega\rangle^0)]\right)\nonumber
\\
&\quad+(1-\chi_\lambda(\rho))\frac{\rho^{-\frac{5}{2}}(1+\rho)^{\frac{3}{2}-\lambda}}{\sqrt{ a(\lambda)}}[1+\O(\rho^{-1}(1-\rho)\langle\omega\rangle^{-1})].
\end{align} 
To estimate the integral term, we first note that the representation \eqref{BesselTaylor} of $Y_2 $  implies $\partial_\rho (\rho^\frac{3}{2} b_2(\rho,\lambda))=\O(\rho^0\langle\omega\rangle^{-2})$. Using this fact we see that
\begin{align*}
\chi_\lambda(\rho)\partial_\rho\left[\rho^{\frac{3}{2}}(1-\rho^2)^{\frac{1}{4}-\frac{\lambda}{2}}\left(\O(\rho^{-1}\langle\omega\rangle^{-\frac{3}{2}})+b_2(\rho,\lambda)[1+\O(\rho^2\langle\omega\rangle^0)]\right)\right]=\chi_\lambda(\rho)\O(\rho^{-\frac{1}{2}}\langle\omega\rangle^{-\frac{3}{2}})
\end{align*} and so
we calculate that
\begin{align*}
\int_0^1\frac{\partial_s(s^{4}u_1(s,\lambda))} {2i(1-2\lambda)(1-s^2)^{\frac{1}{2}-\lambda}} ds
&= \int_0^1\chi_\lambda(s) \O(s^{-\frac{1}{2}}\langle\omega\rangle^{-\frac{5}{2}}) ds
\\
&\quad +\O(\langle\omega\rangle^{-\frac{3}{2}})\int_0^1(1-\chi_\lambda(s))(1-s^2)^{-\frac{1}{2}+\lambda} 
\\
&\quad\times \partial_s\left(s^{\frac{3}{2}}(1+s)^{\frac{3}{2}-\lambda}[1+\O(s^{-1}(1-s)\langle\omega\rangle^{-1})]\right)
ds.
\end{align*}
Thus, appropriate integrations by parts in the latter integral yield
\begin{align*}
\int_0^1\frac{\partial_s(s^{4}u_1(s,\lambda))} {2i(1-2\lambda)(1-s^2)^{\frac{1}{2}-\lambda}} ds=\O(\langle\omega\rangle^{-\frac{5}{2}})
\end{align*}
and consequently,
\begin{align*}
c(f,\lambda)=f(1)\O(\langle\omega\rangle^{-\frac{5}{2}}).
\end{align*}
Similarly one obtains
\begin{align*}
c(f,\lambda)-c_{\mathrm{f}}(f,\lambda)=f(1)\O(\langle\omega\rangle^{-\frac{7}{2}}).
\end{align*}

\end{proof}
\begin{lem}\label{decomp1}
We can decompose $V_1$ as 
\[V_1(f)(\rho,\lambda)= \rho^{-\frac{5}{2}} (1-\rho^2)^{\frac{1}{4}-\frac{\lambda}{2}}f(1)(G_1(\rho,\lambda)+G_2(\rho,\lambda))\] where
\begin{align*}
G_1(\rho,\lambda)&=\chi_\lambda(\rho)b_1(\rho,\lambda)[\O(\langle\omega\rangle^{-\frac{7}{2}})+\O(\rho^2\langle\omega\rangle^{-\frac{5}{2}})]\\
G_2(\rho,\lambda)&=(1-\chi_\lambda(\rho))h_1(\rho,\lambda)[\O(\langle\omega\rangle^{-\frac{7}{2}})+\O(\rho^0(1-\rho)\langle\omega\rangle^{-\frac{7}{2}})]
\\
&\quad+(1-\chi_\lambda(\rho))h_2(\rho,\lambda)[\O(\langle\omega\rangle^{-\frac{7}{2}})+\O(\rho^0(1-\rho)\langle\omega\rangle^{-\frac{7}{2}})].
\end{align*}
\end{lem}
\begin{proof}
Recall that 
\begin{align*}
u_0(\rho,\lambda)&=\rho^{-\frac{5}{2}}(1-\rho^2)^{\frac{1}{4}-\frac{\lambda}{2}}\bigg[\left(c_{1,3}(\lambda)-\frac{c_{2,3}(\lambda)}{c_{2,4}(\lambda)}c_{1,4}(\lambda)\right)\chi_\lambda(\rho)\psi_1(\rho,\lambda)
\\
&\quad+(1-\chi_\lambda(\rho))\left(\psi_3(\rho,\lambda)-\frac{c_{2,3}(\lambda)}{c_{2,4}(\lambda)}\psi_4(\rho,\lambda)\right)
\bigg]
\end{align*}
and similarly 
\begin{align*}
u_{\mathrm{f}_0}(\rho,\lambda)&=\rho^{-\frac{5}{2}}(1-\rho^2)^{\frac{1}{4}-\frac{\lambda}{2}}\bigg[\left(c_{\mathrm{f}_{1,3}}(\lambda)-\frac{c_{\mathrm{f}_{2,3}}(\lambda)}{c_{\mathrm{f}_{2,4}}(\lambda)}c_{\mathrm{f}_{1,4}}(\lambda)\right)\chi_\lambda(\rho)\psi_{\mathrm{f}_1}(\rho,\lambda)
\\
&\quad+(1-\chi_\lambda(\rho))\left(\psi_{\mathrm{f}_3}(\rho,\lambda)-\frac{c_{\mathrm{f}_{2,3}}(\lambda)}{c_{\mathrm{f}_{2,4}}(\lambda)}\psi_{\mathrm{f}_4}(\rho,\lambda)\right)
\bigg].
\end{align*}
To continue we remark that
\begin{align*}
V_1(f)(\rho,\lambda)=c(f,\lambda)(u_0(\rho,\lambda)-u_{\mathrm{f}_0}(\rho,\lambda))+u_{\mathrm{f}_0}(\rho,\lambda)(c(f,\lambda)-c_{\mathrm{f}}(f,\lambda))
\end{align*}
and by employing Lemma \ref{lem:c} one readily checks that 
\[
u_{\mathrm{f}_0}(\rho,\lambda)(c(f,\lambda)-c_{\mathrm{f}}(f,\lambda))
\]
is already of the same form as
\[ \rho^{-\frac{5}{2}} (1-\rho^2)^{\frac{1}{4}-\frac{\lambda}{2}}f(1)(G_1(\rho,\lambda)+G_2(\rho,\lambda)).\]
Next,
we investigate the difference $ u_0(\rho,\lambda)-u_{\mathrm{f}_0}(\rho,\lambda)$ on the support of $\chi_\lambda(\rho)$, i.e., we look at
\begin{align*}
\chi_\lambda(\rho)\left[\left(c_{1,3}(\lambda)-\frac{c_{2,3}(\lambda)}{c_{2,4}(\lambda)}c_{1,4}(\lambda)\right)\psi_1(\rho,\lambda)-\left(c_{\mathrm{f}_{1,3}}(\lambda)-\frac{c_{\mathrm{f}_{2,3}}(\lambda)}{c_{\mathrm{f}_{2,4}}(\lambda)}c_{\mathrm{f}_{1,4}}(\lambda)\right)\psi_{\mathrm{f}_1}(\rho,\lambda)\right].
\end{align*}
Using Lemmas \ref{connectioncoef} and \ref{connectioncoeffree} one directly verifies that
\begin{align*}
c_{1,3}(\lambda)-c_{\mathrm{f}_{1,3}}(\lambda)
=\O\left(\langle\omega\rangle^{-1}\right)
\end{align*}
and 
\begin{align*}
\frac{c_{\mathrm{f}_{2,3}}(\lambda)}{c_{\mathrm{f}_{2,4}}(\lambda)}c_{\mathrm{f}_{1,4}}(\lambda)-\frac{c_{2,3}(\lambda)}{c_{2,4}(\lambda)}c_{1,4}(\lambda)
=&\O(\langle\omega\rangle^{-1}).
\end{align*}
Thus, we obtain that
\begin{align*}
&\quad\chi_\lambda(\rho)\left[\left(c_{1,3}(\lambda)-\frac{c_{2,3}(\lambda)}{c_{2,4}(\lambda)}c_{1,4}(\lambda)\right)\psi_1(\rho,\lambda)-\left(c_{\mathrm{f}_{1,3}}(\lambda)-\frac{c_{\mathrm{f}_{2,3}}(\lambda)}{c_{\mathrm{f}_{2,4}}(\lambda)}c_{\mathrm{f}_{1,4}}(\lambda)\right)\psi_{\mathrm{f}_1}(\rho,\lambda)\right]
\\
&=
\chi_\lambda(\rho)\left[\O(\langle\omega\rangle^{-1})\psi_1(\rho,\lambda)+\left(c_{\mathrm{f}_{1,3}}(\lambda)-\frac{c_{\mathrm{f}_{2,3}}(\lambda)}{c_{\mathrm{f}_{2,4}}(\lambda)}c_{\mathrm{f}_{1,4}}(\lambda)\right)\left(\psi_1(\rho,\lambda)-\psi_{\mathrm{f}_1}(\rho,\lambda)\right)\right]
\\
&=\chi_\lambda(\rho)b_1(\rho,\lambda)[\O(\langle\omega\rangle^{-1})+\O(\rho^2\langle\omega\rangle^{0})].
\end{align*}
A similar computation on the support of $(1-\chi_\lambda(\rho))$ shows that we have 
\begin{align*}
u_0(\rho,\lambda)-u_{\mathrm{f}_0}(\rho,\lambda)&= \rho^{-\frac{5}{2}} (1-\rho^2)^{\frac{1}{4}-\frac{\lambda}{2}} \chi_\lambda(\rho)b_1(\rho,\lambda)[\O(\langle\omega\rangle^{-1})+\O(\langle\omega\rangle^{0}\rho^2)]
\\
&\quad+\rho^{-\frac{5}{2}} (1-\rho^2)^{\frac{1}{4}-\frac{\lambda}{2}}(1-\chi_\lambda(\rho))h_1(\rho,\lambda)[\O(\langle\omega\rangle^{-1})+\O(\rho^0(1-\rho)\langle\omega\rangle^{-1})]
\\
&\quad+\rho^{-\frac{5}{2}} (1-\rho^2)^{\frac{1}{4}-\frac{\lambda}{2}}(1-\chi_\lambda(\rho)) h_2(\rho,\lambda)[\O(\langle\omega\rangle^{-1})+\O(\rho^0(1-\rho)\langle\omega\rangle^{-1})]
\end{align*}
and so the desired decomposition follows, since we know from Lemma \ref{lem:c} that $$c(f,\lambda)=f(1)\O(\langle\omega\rangle^{-\frac{5}{2}}).$$
\end{proof}
\begin{lem}\label{symbol1}
$G_1$ and $G_2$ satisfy
\begin{align*}
\rho^{-\frac{5}{2}} (1-\rho^2)^{\frac{1}{4}-\frac{\lambda}{2}}G_1(\rho,\lambda)&=\chi_\lambda(\rho)(1-\rho^2)^{\frac{1}{4}-\frac{\lambda}{2}}\O(\rho^0\langle\omega\rangle^{-\frac{3}{2}})
\\
\rho^{-\frac{5}{2}} (1-\rho^2)^{\frac{1}{4}-\frac{\lambda}{2}}G_2(\rho,\lambda)&=
(1-\chi_\lambda(\rho))\rho^{-\frac{5}{2}}(1-\rho)^{\frac{3}{2}-\lambda}
[\O(\langle\omega\rangle^{-4})+
\O(\rho^{0}(1-\rho)\langle\omega\rangle^{-4})]
\\
&\quad+(1-\chi_\lambda(\rho))\rho^{-\frac{5}{2}}(1+\rho)^{\frac{3}{2}-\lambda}
[\O(\langle\omega\rangle^{-4})+
\O(\rho^{0}(1-\rho)\langle\omega\rangle^{-4})]
 \end{align*}
\end{lem}
\begin{proof}
This follows immediately from the symbol representations \eqref{BesselTaylor} and Lemma \ref{free ODE near 1}.
\end{proof}
Motivated by this, we define the operators $T_j$ for $j=1,2$ and $f\in C^2(\overline{\B^6_1})$ as
\begin{align*}
T_j(\tau)f(\rho):=&\frac{1}{2\pi i}\lim_{\varepsilon\to 0^+}\lim_{N \to \infty} \int_{\varepsilon-i N}^{\varepsilon+ i N}e^{\lambda\tau}
\rho^{-\frac{5}{2}} (1-\rho^2)^{\frac{1}{4}-\frac{\lambda}{2}}f(1)G_j(\rho,\lambda) d \lambda.
\end{align*}
Note that the limits exist for any $\rho\in (0,1)$ because the integrand is absolutely
convergent by Lemma \ref{symbol1}.
Furthermore, as $F_\lambda(\rho)$ also contains the term $\lambda f_1(\rho)$, we analogously define the operators $\dot{T}_j$ as
$$\dot{T}_j(\tau)f(\rho):=\frac{1}{2\pi i}\lim_{\varepsilon\to 0^+}\lim_{N \to \infty} \int_{\varepsilon-i N}^{\varepsilon+ i N}\lambda e^{\lambda\tau}
\rho^{-\frac{5}{2}} (1-\rho^2)^{\frac{1}{4}-\frac{\lambda}{2}}f(1)G_j(\rho,\lambda) d \lambda
$$
for $j=1,2$ and $f\in C^2(\overline{\B^6_1})$. Again, the limits
exist for any $\rho\in (0,1)$ by Lemma \ref{symbol1} because on the support of $\chi_\lambda$
we may trade powers of $\rho$ for decay in $\omega$ and thereby obtain
absolute convergence.
\begin{lem}\label{kernel1}
	Let $p\in [2,\infty]$ and $q\in [6,12]$ be such that $\frac{1}{p}+\frac{6}{q}=1$.
Then
\begin{equation*}
\|T_j(.)f\|_{L^p(\R_+)L^q(\B_1^6)}\lesssim\|f\|_{H^1(\B_1^6)}
\end{equation*}
and 
\begin{equation*}
\| \dot{T}_j(.)f\|_{L^p(\R_+)L^q(\B_1^6)}\lesssim\|f\|_{H^2(\B_1^6)}
\end{equation*}
for $j=1,2$ and all $f\in C^2(\overline{\B^6_1})$.
\end{lem}
\begin{proof}
Note first, that
for $j=1,2$ the operators $T_j$ satisfy
\begin{align*}
T_j(\tau)f(\rho) =&\frac{1}{2\pi}\int_\R e^{i\omega\tau}
\rho^{-\frac{5}{2}} (1-\rho^2)^{\frac{1}{4}-\frac{i\omega}{2}}f(1)G_j(\rho, i \omega) d \omega,
\end{align*}
since the integral is absolutely convergent.
Moreover,
\begin{align*}
|T_1(\tau)f(\rho)|&=\left|\frac{1}{2\pi}\int_\R e^{i \omega\tau}
\rho^{-\frac{5}{2}} (1-\rho^2)^{\frac{1}{4}-\frac{i \omega}{2}}f(1)G_1(\rho,i\omega) d \omega\right|
\\
&=\left|f(1)\int_\R e^{i \omega\tau}\chi_{i \omega}(\rho)(1-\rho^2)^{\frac{1}{4}-\frac{i \omega}{2}}\O(\rho^0\langle\omega\rangle^{-\frac{3}{2}}) d \omega\right|
\\
&\lesssim\left\langle\tau-\frac{1}{2}\log(1-\rho^2)\right\rangle^{-2}|f(1)|
\\
&\lesssim \langle\tau\rangle^{-2}\|f\|_{H^1(\B^6_1)},
\end{align*}
thanks to Lemma \ref{osci1}.
From this, we can immediately infer the estimate
\begin{equation*}
\|T_1(.)f\|_{L^p(\R_+)L^q(\B_1^6)}\lesssim\|f\|_{H^1(\B_1^6)}
\end{equation*}
for all Strichartz pairs  $(p,q)$ stated in the Lemma.
Next, by using the symbol forms
established in Lemma \ref{symbol1}
and Lemma \ref{osci3} we conclude that
\begin{align*}
|T_2(\tau)f(\rho)|&=\left|\frac{1}{2\pi}\int_\R e^{i \omega\tau}
\rho^{-\frac{5}{2}} (1-\rho^2)^{\frac{1}{4}-\frac{i \omega}{2}}f(1)G_2(\rho,i\omega) d \omega \right|
\\
&\lesssim\left(\langle\tau-\log(1+\rho)\rangle^{-2}+\langle\tau-\log(1-\rho)\rangle^{-2}\right)|f(1)|
\\ 
&\lesssim \langle\tau\rangle^{-2}\|f\|_{H^1(\B^6_1)}.
\end{align*}
Thus, the first claimed estimate has been proven. To establish the bounds on $\dot{T}_j$ we remark that
$V_1(f)(\rho,\lambda)$ satisfies
\begin{align*}
|\lambda V_1(f)(\rho,\lambda)|\lesssim |f(1)|\rho^{-\frac{5}{2}}\langle\lambda\rangle^{-3}
\end{align*}
for any $\rho\in (0,1)$.
Therefore, we can take the limit in the integral to obtain  
\begin{align*}
\dot{T}_1(\tau)f(\rho)&=\frac{i}{2\pi}\int_\R \omega e^{i\omega\tau}
\rho^{-\frac{5}{2}} (1-\rho^2)^{\frac{1}{4}-\frac{i \omega}{2}}f(1)G_1(\rho,i\omega ) d \omega
\\
&=\int_\R e^{i\omega\tau}
\chi_{i\omega}(\rho)(1-\rho^2)^{\frac{1}{4}-\frac{i \omega}{2}}f(1)\O(\rho^{-\frac{1}{12}}\langle\omega\rangle^{-\frac{1}{2}-\frac{1}{12}})d \omega.
\end{align*}
An application of Lemma \ref{osci2} then yields 
\begin{align*}
|\dot{T}_1(\tau)f(\rho)|&\lesssim \rho^{-\frac{1}{12}}|\tau-\frac{1}{2}\log(1-\rho^2)|^{-\frac{1}{2}+\frac{1}{12}}\langle \tau\rangle^{-2}\|f\|_{H^2(\B^6_1)}
\\
&\lesssim \rho^{-\frac{1}{12}}|\tau|^{-\frac{1}{2}+\frac{1}{12}}\langle \tau\rangle^{-2}\|f\|_{H^2(\B^6_1)},
\end{align*}
from which we infer that
\begin{align*}
\|\dot{T}_1(.)f\|_{L^2(\R_+)L^{12}(\B_1^6)}\lesssim\|f\|_{H^2(\B_1^6)}.
\end{align*}
In addition, the estimate 
\begin{align*}
\left|\dot{T}_1(\tau)f(\rho)\right|&=\left|\frac{i}{2\pi}\int_\R \omega e^{i\omega\tau}
\rho^{-\frac{5}{2}} (1-\rho^2)^{\frac{1}{4}-\frac{i \omega}{2}}f(1)G_1(\rho,i\omega ) d \omega\right|
\\
&=\left|\int_\R e^{i\omega\tau}
\chi_{i\omega}(\rho)(1-\rho^2)^{\frac{1}{4}-\frac{i \omega}{2}}f(1)\O(\rho^{-\frac{5}{6}}\langle\omega\rangle^{-\frac{1}{2}-\frac{5}{6}})d \omega\right|
\\
&\lesssim \langle\tau\rangle^{-2}\rho^{-\frac{5}{6}}\|f\|_{H^2(\B^6_1)}
\end{align*}
holds thanks to Lemma \ref{osci1}.
This implies
\begin{align*}
\|\dot{T}_1(.)f\|_{L^{\infty}(\R_+)L^{6}(\B_1^6)}\lesssim\|f\|_{H^2(\B_1^6)},
\end{align*}
and thus, the general estimate 
\begin{align*}
\|\dot{T}_1(.)f\|_{L^{p}(\R_+)L^{q}(\B_1^6)}\lesssim\|f\|_{H^2(\B_1^6)}
\end{align*}
for all admissible pairs $(p,q)$ follows from interpolation.
Finally,
$\dot{T}_2$ can be bounded in a similar manner as $\dot{T}_1$, by additionally making use of Lemma \ref{osci4}. 
\end{proof}
We continue by investigating
\begin{align*}
V_2(f)(\rho,\lambda):=&-u_0(\rho,\lambda)U_1(\rho,\lambda)f(\rho)+u_1(\rho,\lambda)U_0(\rho,\lambda)f(\rho)
\\
&+u_{\mathrm{f}_0}(\rho,\lambda)U_{\mathrm{f}_1}(\rho,\lambda)f(\rho)-u_{\mathrm{f}_1}(\rho,\lambda)U_{\mathrm{f}_0}(\rho,\lambda)f(\rho).
\end{align*}
\begin{lem}\label{decomp2}
We can decompose $V_2$ according to 
\[V_2(f)(\rho,\lambda)=  \rho^{-\frac{5}{2}} (1-\rho^2)^{\frac{1}{4}-\frac{\lambda}{2}}f(\rho) \sum_{j=3}^9G_j(\rho,\lambda)
\]
where
\begin{align*}
G_3(\rho,\lambda)&=\chi_\lambda(\rho)b_1(\rho,\lambda)\int_0^\rho s^{\frac{5}{2}}\frac{b_1(s,\lambda)\alpha_1(\rho,s,\lambda)+b_2(s,\lambda)\alpha_2(\rho,s,\lambda)}{(1-s^2)^{\frac{5}{4}-\frac{\lambda}{2}}}
\\
&\quad+\frac{\O(s^3\langle\omega\rangle^{-2})[1+\O(\rho^2\langle\omega\rangle^0)]}{(1-s^2)^{\frac{5}{4}-\frac{\lambda}{2}}} ds
\\
G_4(\rho,\lambda)&=(1-\chi_\lambda(\rho))h_2(\rho,\lambda)
\int_0^\rho \chi_\lambda(s)s^{\frac{5}{2}}\frac{b_1(s,\lambda)\beta_1(\rho,s,\lambda)+b_2(s,\lambda)\beta_2(\rho,s,\lambda)}{(1-s^2)^{\frac{5}{4}-\frac{\lambda}{2}}}
\\
&\quad+\chi_\lambda(s)\frac{\O(s^3\langle\omega\rangle^{-2})[1+\O(\rho^0(1-\rho)\langle\omega\rangle^{-1})]}{(1-s^2)^{\frac{5}{4}-\frac{\lambda}{2}}}ds
\\
G_5(\rho,\lambda)&=(1-\chi_\lambda(\rho))h_1(\rho,\lambda)
\int_0^\rho \chi_\lambda(s)s^{\frac{5}{2}}\frac{b_1(s,\lambda)\beta_3(\rho,s,\lambda)+b_2(s,\lambda)\beta_4(\rho,s,\lambda)}{(1-s^2)^{\frac{5}{4}-\frac{\lambda}{2}}}
\\
&\quad+\chi_\lambda(s)\frac{\O(s^3\langle\omega\rangle^{-2})[1+\O(\rho^0(1-\rho)\langle\omega\rangle^{-1})]}{(1-s^2)^{\frac{5}{4}-\frac{\lambda}{2}}}ds
\end{align*}
and
\begin{align*}
G_6(\rho,\lambda)&=(1-\chi_\lambda(\rho))h_1(\rho,\lambda)\int_0^\rho (1-\chi_\lambda(s))s^{\frac{5}{2}}\frac{h_2(s,\lambda)\gamma_1(\rho,s,\lambda)}{(1-s^2)^{\frac{5}{4}-\frac{\lambda}{2}}}ds
\\
G_7(\rho,\lambda)&=\chi_\lambda(\rho)b_1(\rho,\lambda)\int_0^\rho s^{\frac{5}{2}}\frac{b_1(s,\lambda)\alpha_3(\rho,s,\lambda)}{(1-s^2)^{\frac{5}{4}-\frac{\lambda}{2}}}ds
\\
&\quad+ \chi_\lambda(\rho)\O(\rho^{\frac{1}{2}}\langle\omega\rangle^{-2})\int_0^\rho s^{\frac{5}{2}}\frac{b_1(s,\lambda)[1+\O(s^2\langle\omega\rangle^0)]}{(1-s^2)^{\frac{5}{4}-\frac{\lambda}{2}}}ds
\\
&\quad+\chi_\lambda(\rho)b_2(\rho,\lambda)\int_0^\rho s^{\frac{5}{2}}\frac{b_1(s,\lambda)\alpha_4(\rho,s,\lambda)}{(1-s^2)^{\frac{5}{4}-\frac{\lambda}{2}}}  ds
\\
G_8(\rho,\lambda)&=(1-\chi_\lambda(\rho))h_2(\rho,\lambda)\int_0^\rho \chi_\lambda(s)s^{\frac{5}{2}}\frac{b_1(s,\lambda)\beta_5(\rho,s,\lambda) }{(1-s^2)^{\frac{5}{4}-\frac{\lambda}{2}}}
ds
\\
G_9(\rho,\lambda)&=(1-\chi_\lambda(\rho))h_2(\rho,\lambda))\int_0^\rho (1-\chi_\lambda(s))s^{\frac{5}{2}}\frac{h_1(s,\lambda)\gamma_2(\rho,s,\lambda)}{(1-s^2)^{\frac{5}{4}-\frac{\lambda}{2}}}
 ds,
\end{align*} 
with 
\begin{align*}
\alpha_j(\rho,s,\lambda)&=\O(\langle\omega\rangle^{-1})+\O(\rho^2\langle\omega\rangle^{0})+\O(s^2\langle\omega\rangle^{0})+\O(\rho^2s^2\langle\omega\rangle^{0})
\\
\beta_j(\rho,s,\lambda)&= 
\O(\langle\omega\rangle^{-1})+\O(\rho^0(1-\rho)\langle\omega\rangle^{-1})+\O(s^2\langle\omega\rangle^{0})+\O(\rho^0(1-\rho)s^2\langle\omega\rangle^{-1})
\\
\gamma_j(\rho,s,\lambda)&=
\O(\langle\omega\rangle^{-1})+\O(\rho^0(1-\rho)\langle\omega\rangle^{-1})+\O(s^0(1-s)\langle\omega\rangle^{-1})
\\
&\quad+\O(\rho^0(1-\rho)s^0(1-s)\langle\omega\rangle^{-2}).
\end{align*}

\end{lem}
\begin{proof}
Recall that $u_0=u_2-\frac{c_{2,3}}{c_{2,4}}u_1$,
where 
\begin{align*}
u_1(\rho,\lambda)&=\rho^{-\frac{5}{2}}(1-\rho^2)^{\frac{1}{4}-\frac{\lambda}{2}}\big[\chi_\lambda(\rho)(c_{1,4}(\lambda)\psi_1(\rho,\lambda)+c_{2,4}(\lambda)\psi_2(\rho,\lambda))\\
&\quad+\left(1-\chi_\lambda(\rho)\right)\psi_4(\rho,\lambda)\big]\\
u_2(\rho,\lambda)&= \rho^{-\frac{5}{2}}(1-\rho^2)^{\frac{1}{4}-\frac{\lambda}{2}}\big[\chi_\lambda(\rho)(c_{1,3}(\lambda)\psi_1(\rho,\lambda)+c_{2,3}(\lambda)\psi_2(\rho,\lambda))\\
&\quad+\left(1-\chi_\lambda(\rho)\right)\psi_3(\rho,\lambda)\big],
\end{align*}
with the $\psi_j$ as in Lemmas \ref{Besselsol} and \ref{Hankelsol}, and $c_{j,k}$ from Lemma \ref{connectioncoef}.
Furthermore we have that
\[
U_j(\rho,\lambda)=\int_0^\rho\frac{s^5 u_j(s,\lambda)}{(1-s^2)^{\frac{3}{2}-\lambda}}ds.
\]
Having reminded ourselves of these definitions we start by decomposing $-u_0(\rho,\lambda)U_1(\rho,\lambda)+u_{\mathrm{f}_0}(\rho,\lambda)U_{\mathrm{f}_1}(\rho,\lambda)$ on the support of $\chi_\lambda$. Upon neglecting the prefactor $\rho^{-\frac{5}{2}}(1-\rho^2)^{\frac{1}{4}-\frac{\lambda}{2}}$ for the time being, we want to decompose
\begin{align*}
&\quad\chi_\lambda(\rho)\Bigg[\left(c_{1,3}(\lambda)-\frac{c_{2,3}(\lambda)}{c_{2,4}(\lambda)}c_{1,4}(\lambda)\right)\psi_1(\rho,\lambda)
\\
&\quad\times\int_0^\rho s^{\frac{5}{2}}\frac{\chi_\lambda(s)\left(c_{1,4}(\lambda)\psi_1(\rho,\lambda)+c_{2,4}(\lambda)\psi_2(\rho,\lambda)\right)+(1-\chi_\lambda(s))\psi_4(s,\lambda)}{(1-s^2)^{\frac{5}{4}-\frac{\lambda}{2}}} ds
\\
&\quad-\left(c_{\mathrm{f}_{1,3}}(\lambda)-\frac{c_{\mathrm{f}_{2,3}}(\lambda)}{c_{\mathrm{f}_{2,4}}(\lambda)}c_{\mathrm{f}_{1,4}}(\lambda)\right)\psi_{\mathrm{f}_1}(\rho,\lambda)
\\
&\quad\times\int_0^\rho s^{\frac{5}{2}}\frac{\chi_\lambda(s)(c_{\mathrm{f}_{1,4}}(\lambda)\psi_{\mathrm{f}_1}(s,\lambda)+c_{_\mathrm{2,4}}(\lambda)\psi_{\mathrm{2}}(\rho,\lambda))+(1-\chi_\lambda(s))\psi_{\mathrm{f}_4}(s,\lambda)}{(1-s^2)^{\frac{5}{4}-\frac{\lambda}{2}}} ds\Bigg]
\\
&=
\chi_\lambda(\rho)\Bigg[\left(c_{1,3}(\lambda)-\frac{c_{2,3}(\lambda)}{c_{2,4}(\lambda)}c_{1,4}(\lambda)\right)\psi_1(\rho,\lambda)
\\
&\quad\times\int_0^\rho s^{\frac{5}{2}}\frac{c_{1,4}(\lambda)\psi_1(\rho,\lambda)+c_{2,4}(\lambda)\psi_2(\rho,\lambda)}{(1-s^2)^{\frac{5}{4}-\frac{\lambda}{2}}} ds
\\
&\quad-\left(c_{\mathrm{f}_{1,3}}(\lambda)-\frac{c_{\mathrm{f}_{2,3}}(\lambda)}{c_{\mathrm{f}_{2,4}}(\lambda)}c_{\mathrm{f}_{1,4}}(\lambda)\right)\psi_{\mathrm{f}_1}(\rho,\lambda)
\\
&\quad\times\int_0^\rho s^{\frac{5}{2}}\frac{c_{\mathrm{f}_{1,4}}(\lambda)\psi_{\mathrm{f}_1}(s,\lambda)+c_{_\mathrm{2,4}}(\lambda)\psi_{\mathrm{2}}(\rho,\lambda)}{(1-s^2)^{\frac{5}{4}-\frac{\lambda}{2}}} ds\Bigg],
\end{align*}
since $s\leq \rho$.
Thanks to the proof of Lemma \ref{decomp1} we already know that
\begin{align*}
c_{\mathrm{f}_{1,3}}(\lambda)-\frac{c_{\mathrm{f}_{2,3}}(\lambda)}{c_{\mathrm{f}_{2,4}}(\lambda)}c_{\mathrm{f}_{1,4}}(\lambda)
-\left(
c_{1,3}(\lambda)-\frac{c_{2,3}(\lambda)}{c_{2,4}(\lambda)}c_{1,4}\right)=\O\langle\omega\rangle^{-1}.
\end{align*}
Furthermore,
\begin{align*}
\psi_1(\rho,\lambda)-\psi_{\mathrm{f}_1}(\rho,\lambda)=b_1(\rho,\lambda)\O(\rho^2\langle\omega\rangle^0)
\end{align*}
and analogous equalities also hold for $\psi_j-\psi_{\mathrm{f}_j}$ with $j=2,3,4.$
Thus, a repeated usage of the identity 
\[
a_1b_1-a_2b_2=b_1(a_1-a_2)+a_2(b_1-b_2)
\] yields
\begin{align*}
&\quad\chi_\lambda(\rho)\Bigg[\left(c_{1,3}(\lambda)-\frac{c_{2,3}(\lambda)}{c_{2,4}(\lambda)}c_{1,4}(\lambda)\right)\psi_1(\rho,\lambda)
\\
&\quad\times\int_0^\rho s^{\frac{5}{2}}\frac{\chi_\lambda(s)\left(c_{1,4}(\lambda)\psi_1(\rho,\lambda)+c_{2,4}(\lambda)\psi_2(\rho,\lambda)\right)+(1-\chi_\lambda(s))\psi_4(s,\lambda)}{(1-s^2)^{\frac{5}{4}-\frac{\lambda}{2}}} ds
\\
&\quad-\left(c_{\mathrm{f}_{1,3}}(\lambda)-\frac{c_{\mathrm{f}_{2,3}}(\lambda)}{c_{\mathrm{f}_{2,4}}(\lambda)}c_{\mathrm{f}_{1,4}}(\lambda)\right)\psi_{\mathrm{f}_1}(\rho,\lambda)
\\
&\quad\times\int_0^\rho s^{\frac{5}{2}}\frac{\chi_\lambda(s)\left(c_{\mathrm{f}_{1,4}}(\lambda)\psi_{\mathrm{f}_1}(s,\lambda)+c_{_\mathrm{2,4}}(\lambda)\psi_{\mathrm{2}}(\rho,\lambda)\right)+(1-\chi_\lambda(s))\psi_{\mathrm{f}_4}(s,\lambda)}{(1-s^2)^{\frac{5}{4}-\frac{\lambda}{2}}} ds\Bigg]
\\
&=G_3(\rho,\lambda).
\end{align*}
The other multipliers $G_j$ can be obtained by applying the same considerations to the remaining terms.
\end{proof}
By using the symbol representations established in Lemmas \ref{Besselsol} and \ref{Hankelsol} this decomposition can be rewritten in a much more compact form as follows.\begin{lem}\label{V2symbol}
The multipliers $G_j$ satisfy
\begin{align*}
\rho^{-\frac{5}{2}} (1-\rho^2)^{\frac{1}{4}-\frac{\lambda}{2}}G_3(\rho,\lambda)&=\chi_\lambda(\rho)(1-\rho^2)^{\frac{1}{4}-\frac{\lambda}{2}}\int_0^\rho \frac{\O(\rho^0s\langle\omega\rangle^{-1})}{(1-s^2)^{\frac{5}{4}-\frac{\lambda}{2}}}  ds
\\
\rho^{-\frac{5}{2}} (1-\rho^2)^{\frac{1}{4}-\frac{\lambda}{2}}G_4(\rho,\lambda)&=(1-\chi_\lambda(\rho))\rho^{-\frac{5}{2}}(1+\rho)^{\frac{3}{2}-\lambda}
\\
&\quad\times\int_0^\rho \chi_\lambda(s)\frac{\O(s\langle\omega\rangle^{-\frac{5}{2}})\widehat{\beta}_1(\rho,s,\lambda)}{(1-s^2)^{\frac{5}{4}-\frac{\lambda}{2}}}
\\
&\quad+\chi_\lambda(s)\frac{\O(s^3\langle\omega\rangle^{-2})[1+\O(\rho^{-1}(1-\rho)\langle\omega\rangle^{-1})]}{(1-s^2)^{\frac{5}{4}-\frac{\lambda}{2}}}ds
\\
\rho^{-\frac{5}{2}} (1-\rho^2)^{\frac{1}{4}-\frac{\lambda}{2}}G_5(\rho,\lambda)&=(1-\chi_\lambda(\rho))\rho^{-\frac{5}{2}}(1-\rho)^{\frac{3}{2}-\lambda}
\\
&\quad\times\int_0^\rho \chi_\lambda(s)\frac{\O(s\langle\omega\rangle^{-\frac{5}{2}})\widehat{\beta}_2(\rho,s,\lambda)}{(1-s^2)^{\frac{5}{4}-\frac{\lambda}{2}}}
\\
&\quad+\chi_\lambda(s)\frac{\O(s^3\langle\omega\rangle^{-2})[1+\O(\rho^{-1}(1-\rho)\langle\omega\rangle^{-1})]}{(1-s^2)^{\frac{5}{4}-\frac{\lambda}{2}}}ds
\\
\rho^{-\frac{5}{2}} (1-\rho^2)^{\frac{1}{4}-\frac{\lambda}{2}}G_6(\rho,\lambda)&=(1-\chi_\lambda(\rho))\rho^{-\frac{5}{2}}(1-\rho)^{\frac{3}{2}-\lambda}
\\
&\quad\times\int_0^\rho (1-\chi_\lambda(s))\frac{s^{\frac{5}{2}}\O(\langle\omega\rangle^{-1})}{(1-s)^{\frac{3}{2}-\lambda}}\widehat{\gamma}_1(\rho,s,\lambda)ds
\\
\rho^{-\frac{5}{2}} (1-\rho^2)^{\frac{1}{4}-\frac{\lambda}{2}}G_7(\rho,\lambda)&=\chi_\lambda(\rho)(1-\rho^2)^{\frac{1}{4}-\frac{\lambda}{2}}\int_0^\rho \frac{\O(\rho^0s\langle\omega\rangle^{-1})}{(1-s^2)^{\frac{5}{4}-\frac{\lambda}{2}}}  ds
\\
\rho^{-\frac{5}{2}} (1-\rho^2)^{\frac{1}{4}-\frac{\lambda}{2}}G_8(\rho,\lambda)&=(1-\chi_\lambda(\rho))\rho^{-\frac{5}{2}}(1+\rho)^{\frac{3}{2}-\lambda}
\\
&\quad\times\int_0^\rho \chi_\lambda(s)\frac{\O(s\langle\omega\rangle^{-\frac{5}{2}})}{(1-s^2)^{\frac{5}{4}-\frac{\lambda}{2}}}
\widehat{\beta}_3(\rho,s,\lambda) ds
\\
\rho^{-\frac{5}{2}} (1-\rho^2)^{\frac{1}{4}-\frac{\lambda}{2}}G_9(\rho,\lambda)&=(1-\chi_\lambda(\rho))\rho^{-\frac{5}{2}}(1+\rho)^{\frac{3}{2}-\lambda}
\\
&\quad\times\int_0^\rho (1-\chi_\lambda(s))\frac{s^{\frac{5}{2}}\O(\langle\omega\rangle^{-1})}{(1+s)^{\frac{3}{2}-\lambda}}
\widehat{\gamma}_2(\rho,s,\lambda) ds,
\end{align*}
where 
\begin{align*}
\widehat{\beta}_j(\rho,s,\lambda)&=[1+\O(\rho^{-1}(1-\rho)\langle\omega\rangle^{-1})][O(\langle\omega\rangle^{-1})+\O(\rho^0(1-\rho)\langle\omega\rangle^{-1})
\\
&\quad+\O(s^2\langle\omega\rangle^{0})+\O(\rho^0(1-\rho)s^2\langle\omega\rangle^{-1})]
\\
\widehat{\gamma}_j(\rho,s,\lambda)&=[1+\O(\rho^{-1}(1-\rho)\langle\omega\rangle^{-1})][1+\O(s^{-1}(1-s)\langle\omega\rangle^{-1})]\gamma_j(\rho,s,\lambda),
\end{align*}
with $\gamma_j$ from Lemma \ref{decomp2}.
\end{lem}  
\begin{proof}
This is an immediate consequence of Lemma \ref{decomp2}, previously
established symbol representations, and support properties of the cut-offs.
\end{proof}
In analogy to what we did before we define the operators $T_j$ and $\dot{T}_j$ for $j=3,\dots,9$ and $f\in C^2(\overline{\B^6_1})$  as
\begin{align*}
T_j(\tau)f(\rho):=\frac{1}{2\pi i}\lim_{\varepsilon\to 0^+}\lim_{N \to \infty} \int_{\varepsilon-i N}^{\varepsilon+ i N}e^{\lambda\tau}
\rho^{-\frac{5}{2}} (1-\rho^2)^{\frac{1}{4}-\frac{\lambda}{2}}f(\rho)G_j(\rho,\lambda) d \lambda
\end{align*}
and
\begin{align*}
\dot{T}_j(\tau)f(\rho):=\frac{1}{2\pi i}\lim_{\varepsilon\to 0^+}\lim_{N \to \infty} \int_{\varepsilon-i N}^{\varepsilon+ i N}\lambda e^{\lambda\tau}
\rho^{-\frac{5}{2}} (1-\rho^2)^{\frac{1}{4}-\frac{\lambda}{2}}f(\rho)G_j(\rho,\lambda) d \lambda.
\end{align*}
That these limits exist for each $\tau\geq 0$ and $\rho\in (0,1)$
follows from the computations in the proof of Lemma \ref{kernel2}.

\begin{lem}\label{kernel2}
	Let $p\in [2,\infty]$ and $q\in [6,12]$ be such that $\frac{1}{p}+\frac{6}{q}=1$.
	Then
	\begin{equation*}
	\|T_j(.)f\|_{L^p(\R_+)L^q(\B_1^6)}\lesssim\|f\|_{H^1(\B_1^6)}
	\end{equation*}
	and 
	\begin{equation*}
	\| \dot{T}_j(.)f\|_{L^p(\R_+)L^q(\B_1^6)}\lesssim\|f\|_{H^2(\B_1^6)}
	\end{equation*}
	for $j=3,4,\dots,9$ and all $f\in C^2(\overline{\B^6_1})$.
\end{lem}
\begin{proof}
Note that we have  
\[
|V_2(f)(\rho,\lambda)|\lesssim\langle\lambda\rangle^{-2} |f(\rho)|
\]
and therefore
\begin{align*}
T_j(\tau)f(\rho)=\frac{1}{2\pi} \int_\R e^{i \omega\tau}
\rho^{-\frac{5}{2}} (1-\rho^2)^{\frac{1}{4}-\frac{\lambda}{2}}f(\rho)G_j(\rho,i\omega) d \omega
\end{align*}
holds for $j=3,..,9.$
Furthermore, we can also use Fubini's theorem to interchange the order of integration in all the operators $T_j$.
For $T_3$ this yields
\begin{align*}
T_3(\tau)f(\rho)=f(\rho)\int_0^{\rho}\int_\R e^{i\omega \tau} \O(\rho^{\frac{11}{12}}s^0\langle\omega\rangle^{-1-\frac{1}{12}})
\chi_{i\omega}(\rho) \frac{(1-\rho^2)^{\frac{1}{4}-\frac{i \omega}{2}}}{(1-s^2)^{\frac{5}{4}-\frac{i\omega}{2}}} d \omega ds.
\end{align*}
By using Lemma \ref{osci1} we then obtain
\begin{align*}
|T_3(\tau)f(\rho)|&\lesssim \left|\rho^{\frac{11}{12}}f(\rho) \int_0^{\rho}\langle\tau +\frac{1}{2}(\log(1-s^2)-\log(1-\rho^2))\rangle^{-2}ds\right|
\\
&\lesssim\left|\rho^{\frac{11}{12}}f(\rho) \int_0^{\rho}\langle\tau \rangle^{-2}ds\right|
\\
&\lesssim\rho^{2-\frac{1}{12}}|f(\rho)|\langle\tau \rangle^{-2}.
\end{align*}
 Moreover, we have that
\[|\dot{T}_3(\tau)f(\rho)|\lesssim \rho^{-1}|f(\rho)|\langle\lambda\rangle^{-2} \]
for $\rho\in(0,1)$
and so we can also take the $\varepsilon$ limit and change the order of integration. A calculation  similar to the one used to bound $|T_3(\tau)f(\rho)|,$ then shows
\begin{align*}
|\dot{T}_3(\tau)f(\rho)|\lesssim\rho^{1-\frac{1}{12}}|f(\rho)|\langle\tau \rangle^{-2}.
\end{align*}
Using Lemmas \ref{osci1} and \ref{osci3} the estimates
\begin{align*}
|T_j(\tau)f(\rho)|\lesssim\rho^{2-\frac{1}{12}}|f(\rho)|\langle\tau \rangle^{-2}
\end{align*}
and
\begin{align*}
|\dot{T}_j(\tau)f(\rho)|\lesssim\rho^{1-\frac{1}{12}}|f(\rho)|\langle\tau \rangle^{-2},
\end{align*}
for $j=4,5,7,8$ can be derived likewise.
For $T_6$ we have that
\begin{align*}
T_6(\tau)f(\rho)&=\rho^{-\frac{5}{2}}f(\rho)\int_0^{\rho} \int_\R e^{i \omega\tau}(1-\chi_{i \omega}(\rho))(1-\rho)^{\frac{3}{2}-i \omega}
\\
&\quad \times s^{\frac{5}{2}}\frac{1-\chi_{i \omega}(s)}{(1-s)^{\frac{3}{2}-i\omega}}\O(\langle\omega\rangle^{-1})
\widehat{\gamma}_1(\rho,s,i \omega)d \omega ds.
\end{align*}
Since $\widehat{\gamma}_6(.,.,i\omega)$ behaves like $\O(\langle\omega\rangle^{-1})$ in $\omega$,
Lemma \ref{osci3} yields
\begin{align*}
|T_6(\tau)f(\rho)|&\lesssim\rho^{-\frac{3}{2}}|f(\rho)|\int_0^\rho s^{\frac{5}{2}} \langle\tau-\log(1-\rho)+\log(1-s)\rangle^{-2}ds
\\
&\lesssim\rho^{-\frac{3}{2}}|f(\rho)|\langle\tau\rangle^{-2}\int_0^\rho s^{\frac{5}{2}}ds
\\
&\lesssim\rho^{2}|f(\rho)|\langle\tau\rangle^{-2}.
\end{align*}
An analogous calculation shows
$$
|T_9(\tau)f(\rho)|\lesssim\rho^{2}|f(\rho)|\langle\tau\rangle^{-2}.
$$
Next, as the $\lambda$ integrals in $\dot{T}_6$ and $\dot{T}_9$ appear
to be no longer absolutely convergent, we perform an integration by parts which yields
\begin{align*}
&\quad\int_0^\rho (1-\chi_\lambda(s))\frac{s^{\frac{5}{2}}\O(\langle\omega\rangle^{-1})}{(1-s)^{\frac{3}{2}-\lambda}}\widehat{\gamma}_1(\rho,s,\lambda) ds,
\\ 
&=(1-\chi_\lambda(\rho))\frac{\rho^{\frac{5}{2}}\O(\langle\omega\rangle^{-2})}{(1-\rho)^{\frac{1}{2}-\lambda}}\widehat{\gamma}_1(\rho,\rho,\lambda)
\\
&\quad +\O(\langle\omega\rangle^{-2})\int_0^\rho(1-s)^{-\frac{1}{2}+\lambda}\partial_s\left[(1-\chi_\lambda(s))s^{\frac{5}{2}}\widehat{\gamma}_1(\rho,s,\lambda)\right] ds\\
&=:B_6(\rho,\lambda)+I_6(\rho,\lambda).
\end{align*}
This now renders the $\lambda$ integral absolutely convergent and we can again take the limit $\varepsilon \to 0$. Note that the operator  corresponding to the integral term $ I_6(\rho,\lambda)$ is bounded by
\begin{align*}
\rho|f(\rho)|\langle\tau\rangle^{-2},
\end{align*}
which can be verified in a similar manner as the bound on $T_6$.
Furthermore,
\begin{align*}
&\quad\left|\rho^{-\frac{5}{2}}f(\rho)\int_\R \omega e^{i\omega \tau}((1-\chi_{i\omega}(\rho))(1-\rho)^{\frac{3}{2}-i\omega} B_6(\rho,i\omega) d \omega\right|
\\
&=\left|\rho^{-1}f(\rho)\int_\R \omega e^{i\omega \tau}(1-\chi_{i \omega}(\rho))^2 \rho(1-\rho)\O(\langle\omega\rangle^{-2}) \widehat{\gamma}_1(\rho,\rho,i \omega) d \omega\right|
\\
&\lesssim \rho|f(\rho)|\langle\tau \rangle^{-2},
\end{align*}
due to Lemma \ref{osci3}
and therefore
\begin{align*}
\left|\dot{T}_6(\tau)f\right|\lesssim \rho|f(\rho)|\langle\tau \rangle^{-2}.
\end{align*}
As $\dot{T}_9$ can be bounded by similar means, we conclude that
\begin{align*}
|T_j(\tau)f(\rho)|\lesssim\rho^{2-\frac{1}{12}}|f(\rho)|\langle\tau \rangle^{-2}
\end{align*}
and
\begin{align*}
|\dot{T}_j(\tau)f(\rho)|\lesssim\rho^{1-\frac{1}{12}}|f(\rho)|\langle\tau \rangle^{-2}
\end{align*}
holds for all $j=3,4,\dots,9$ and so we wrap up this proof by employing Lemma \ref{teclem1}.
\end{proof}
Finally, we can turn our attention to 
\begin{align*}
V_3(f)(\rho,\lambda)=&-u_0\int_\rho^{1}U_1(s,\lambda)f'(s)ds +u_1\int_0^\rho U_0(s,\lambda)f'(s)ds 
\\
&+u_{\mathrm{f}_0}\int_\rho^{1}U_{\mathrm{f}_1}(s,\lambda)f'(s)ds -u_{\mathrm{f}_1}\int_0^\rho U_{\mathrm{f}_0}(s,\lambda)f'(s)ds.
\end{align*}
As we did before, we also derive a useful decomposition of $V_3$.
\begin{lem}
We can decompose $V_3$ as
\[
V_3(f)(\rho,\lambda)=\rho^{-\frac{5}{2}} (1-\rho^2)^{\frac{1}{4}-\frac{\lambda}{2}}\sum_{j=10}^{18} G_j(f)(\rho,\lambda)
\]
with
\begin{align*}
G_{10}(f)(\rho,\lambda)&=\chi_{\lambda}(\rho)b_1(\rho,\lambda)\int_\rho^1f'(s)\int_0^s\chi_{\lambda}(t)t^{\frac{5}{2}}\frac{b_1(t,\lambda)\alpha_5(\rho,t,\lambda)+b_2(t,\lambda)\alpha_6(\rho,t,\lambda)}{(1-t^2)^{\frac{5}{4}-\frac{\lambda}{2}}}
\\
&\quad+
\chi_\lambda(t)\frac{\O(t^3\langle\omega\rangle^{-2})[1+\O(\rho^2\langle\omega\rangle^0)]}{(1-t^2)^{\frac{5}{4}-\frac{\lambda}{2}}} dt ds
\\
G_{11}(f)(\rho,\lambda)&=(1-\chi_{\lambda}(\rho))h_2(\rho,\lambda)\int_\rho^1f'(s)
\int_0^s\chi_{\lambda}(t)t^{\frac{5}{2}}\frac{b_1(t,\lambda)\beta_6(\rho,t,\lambda)+b_2(t,\lambda)\beta_7(\rho,t,\lambda)}{(1-t^2)^{\frac{5}{4}-\frac{\lambda}{2}}}
\\
&\quad+
\chi_\lambda(t)\frac{\O(t^3\langle\omega\rangle^{-2})[1+\O(\rho^0(1-\rho)\langle\omega\rangle^{-1})]}{(1-t^2)^{\frac{5}{4}-\frac{\lambda}{2}}} dt ds
\\
G_{12}(f)(\rho,\lambda)&=(1-\chi_{\lambda}(\rho))h_1(\rho,\lambda)\int_\rho^1f'(s)
\int_0^s\chi_{\lambda}(t)t^{\frac{5}{2}}\frac{b_1(t,\lambda)\beta_7(\rho,t,\lambda)+b_2(t,\lambda)\beta_8(\rho,t,\lambda)}{(1-t^2)^{\frac{5}{4}-\frac{\lambda}{2}}}
\\
&\quad+
\chi_\lambda(t)\frac{\O(t^3\langle\omega\rangle^{-2})[1+\O(\rho^0(1-\rho)\langle\omega\rangle^{-1})]}{(1-t^2)^{\frac{5}{4}-\frac{\lambda}{2}}} dt ds
\\
G_{13}(f)(\rho,\lambda)&=\chi_{\lambda}(\rho)b_1(\rho,\lambda)\int_\rho^1f'(s)\int_0^s(1-\chi_{\lambda}(t))t^{\frac{5}{2}}\frac{h_2(t,\lambda)\beta_{9}(t,\rho,\lambda)}{(1-t^2)^{\frac{5}{4}-\frac{\lambda}{2}}} dt ds
\\
G_{14}(f)(\rho,\lambda)&=(1-\chi_{\lambda}(\rho))h_2(\rho,\lambda)\int_\rho^1f'(s)\int_0^s(1-\chi_{\lambda}(t))t^{\frac{5}{2}}\frac{h_2(t,\lambda)\gamma_3(\rho,t,\lambda)}{(1-t^2)^{\frac{5}{4}-\frac{\lambda}{2}}} dt ds
\\
G_{15}(f)(\rho,\lambda)&=(1-\chi_{\lambda}(\rho))h_1(\rho,\lambda)\int_\rho^1f'(s)\int_0^s(1-\chi_{\lambda}(t))t^{\frac{5}{2}}\frac{h_2(t,\lambda)\gamma_4(\rho,t,\lambda)}{(1-t^2)^{\frac{5}{4}-\frac{\lambda}{2}}} dt ds
\\
G_{16}(f)(\rho,\lambda)&=\chi_{\lambda}(\rho)b_1(\rho,\lambda)\int_0^\rho f'(s)\int_0^st^{\frac{5}{2}}\frac{b_1(t,\lambda)\alpha_7(t,\rho,\lambda)}{(1-t^2)^{\frac{5}{4}-\frac{\lambda}{2}}} dt ds
\\
&\quad+\chi_{\lambda}(\rho)\O(\rho^{\frac{1}{2}}\langle\omega\rangle^{-2})\int_0^\rho f'(s)\int_0^st^{\frac{5}{2}}\frac{b_1(t,\lambda)[1+\O(t^2\langle\omega\rangle^0)]}{(1-t^2)^{\frac{5}{4}-\frac{\lambda}{2}}} dt ds
\\
&\quad+\chi_{\lambda}(\rho)b_2(\rho,\lambda) \int_0^\rho f'(s)\int_0^st^{\frac{5}{2}}\frac{b_1(t,\lambda)\alpha_8(\rho,t,\lambda)}{(1-t^2)^{\frac{5}{4}-\frac{\lambda}{2}}} dt ds
\\
G_{17}(f)(\rho,\lambda)&=(1-\chi_{\lambda}(\rho))h_2(\rho,\lambda) \int_0^\rho f'(s)\int_0^s\chi_{\lambda}(t)t^{\frac{5}{2}}\frac{b_1(t,\lambda)\beta_{10}(\rho,t,\lambda)}{(1-t^2)^{\frac{5}{4}-\frac{\lambda}{2}}} 
dt ds
\\
G_{18}(f)(\rho,\lambda)&=(1-\chi_{\lambda}(\rho))h_2(\rho,\lambda) \int_0^\rho f'(s)
\\
&\quad\times\int_0^s(1-\chi_\lambda(t))t^{\frac{5}{2}}\frac{h_1(t,\lambda)\gamma_5(\rho,t,\lambda)+h_2(t,\lambda)\gamma_6(\rho,t,\lambda)}{(1-t^2)^{\frac{5}{4}-\frac{\lambda}{2}}} dt ds,
\end{align*}
where as before
\begin{align*}
\alpha_j(\rho,t,\lambda)&=\O(\langle\omega\rangle^{-1})+\O(\rho^2\langle\omega\rangle^{0})+\O(t^2\langle\omega\rangle^{0})+\O(\rho^2t^2\langle\omega\rangle^{0})
\\
\beta_j(\rho,t,\lambda)&= 
\O(\langle\omega\rangle^{-1})+\O(\rho^0(1-\rho)\langle\omega\rangle^{-1})+\O(t^2\langle\omega\rangle^{0})+\O(\rho^0(1-\rho)t^2\langle\omega\rangle^{-1})
\\
\gamma_j(\rho,t,\lambda)&= 
\O(\langle\omega\rangle^{-1})+\O(\rho^0(1-\rho)\langle\omega\rangle^{-1})+\O(t^0(1-t)\langle\omega\rangle^{-1})
\\
&\quad+\O(\rho^0(1-\rho)t^0(1-t)\langle\omega\rangle^{-2}).
\end{align*}

\end{lem}
\begin{proof}
This decomposition follows in the same way as the one in Lemma \ref{decomp2}.
\end{proof}
It will again be useful for us to use the symbol representations.
\begin{lem}
We have that
\begin{align*}
\rho^{-\frac{5}{2}} (1-\rho^2)^{\frac{1}{4}-\frac{\lambda}{2}}G_{10}(f)(\rho,\lambda)&=\chi_{\lambda}(\rho)(1-\rho^2)^{\frac{1}{4}-\frac{\lambda}{2}}\int_\rho^1 f'(s)
\int_0^s\chi_{\lambda}(t)\frac{\O\left(\rho^0 t\langle\omega\rangle^{-1}\right)}{(1-t^2)^{\frac{5}{4}-\frac{\lambda}{2}}} dt ds
\\
\rho^{-\frac{5}{2}} (1-\rho^2)^{\frac{1}{4}-\frac{\lambda}{2}}G_{11}(f)(\rho,\lambda)&=(1-\chi_{\lambda}(\rho))\rho^{-\frac{5}{2}}(1+\rho)^{\frac{3}{2}-\lambda}\int_\rho^1f'(s)
\\
&\quad\times\int_0^s\chi_{\lambda}(t)\frac{\O(t\langle
\omega\rangle^{-\frac{5}{2}})\widehat{\beta}_4(\rho,t,\lambda)}{(1-t^2)^{\frac{5}{4}-\frac{\lambda}{2}}}
\\
&\quad+
\chi_\lambda(t)\frac{\O(t^3\langle\omega\rangle^{-2})[1+\O(\rho^{-1}(1-\rho)\langle\omega\rangle^{-1})]}{(1-t^2)^{\frac{5}{4}-\frac{\lambda}{2}}}
dt ds
\\
\rho^{-\frac{5}{2}} (1-\rho^2)^{\frac{1}{4}-\frac{\lambda}{2}}G_{12}(f)(\rho,\lambda)&=(1-\chi_{\lambda}(\rho))\rho^{-\frac{5}{2}}(1-\rho)^{\frac{3}{2}-\lambda}\int_\rho^1f'(s)
\\
&\quad\times\int_0^s\chi_{\lambda}(t)\frac{\O(t\langle
\omega\rangle^{-\frac{5}{2}})\widehat{\beta}_5(\rho,t,\lambda)}{(1-t^2)^{\frac{5}{4}-\frac{\lambda}{2}}}
\\
&\quad+
\chi_\lambda(t)\frac{\O(t^3\langle\omega\rangle^{-2})[1+\O(\rho^{-1}(1-\rho)\langle\omega\rangle^{-1})]}{(1-t^2)^{\frac{5}{4}-\frac{\lambda}{2}}}
dt ds
\end{align*}
and
\begin{align*}
\rho^{-\frac{5}{2}} (1-\rho^2)^{\frac{1}{4}-\frac{\lambda}{2}}G_{13}(f)(\rho,\lambda)&=\chi_{\lambda}(\rho)(1-\rho^2)^{\frac{1}{4}-\frac{\lambda}{2}}\int_\rho^1f'(s)
\\
&\quad\times\int_0^s(1-\chi_{\lambda}(t))t^{\frac{5}{2}}\frac{\O(\rho^0 \langle\omega \rangle^{\frac{3}{2}})}{(1-t)^{\frac{3}{2}-\lambda}}\widehat{\beta}_6(t,\rho,\lambda) dt ds
\\
\rho^{-\frac{5}{2}} (1-\rho^2)^{\frac{1}{4}-\frac{\lambda}{2}}G_{14}(f)(\rho,\lambda)&=(1-\chi_{\lambda}(\rho))\rho^{-\frac{5}{2}}(1+\rho)^{\frac{3}{2}-\lambda}\int_\rho^1f'(s)
\\
&\quad\times\int_0^s(1-\chi_{\lambda}(t))t^{\frac{5}{2}}\frac{\O(\langle\omega\rangle^{-1})}{(1-t)^{\frac{3}{2}-\lambda}} \widehat{\gamma}_3(\rho,t,\lambda)dt ds
\\
\rho^{-\frac{5}{2}} (1-\rho^2)^{\frac{1}{4}-\frac{\lambda}{2}}G_{15}(f)(\rho,\lambda)&=(1-\chi_{\lambda}(\rho))\rho^{-\frac{5}{2}}(1-\rho)^{\frac{3}{2}-\lambda}\int_\rho^1f'(s)
\\
&\quad\times\int_0^s(1-\chi_{\lambda}(t))t^{\frac{5}{2}}\frac{\O(\langle\omega\rangle^{-1})}{(1-t)^{\frac{3}{2}-\lambda}} \widehat{\gamma}_4(\rho,t,\lambda)dt ds
\end{align*}
and
\begin{align*}
\rho^{-\frac{5}{2}} (1-\rho^2)^{\frac{1}{4}-\frac{\lambda}{2}}G_{16}(f)(\rho,\lambda)&=\chi_{\lambda}(\rho)(1-\rho^2)^{\frac{1}{4}-\frac{\lambda}{2}}
\int_0^\rho f'(s)\int_0^s\chi_{\lambda}(t)\frac{\O\left(\rho^0 t\langle\omega\rangle^{-1}\right)}{(1-t^2)^{\frac{5}{4}-\frac{\lambda}{2}}} dt ds
\\
\rho^{-\frac{5}{2}} (1-\rho^2)^{\frac{1}{4}-\frac{\lambda}{2}}G_{17}(f)(\rho,\lambda)&=(1-\chi_{\lambda}(\rho))\rho^{-\frac{5}{2}}(1+\rho)^{\frac{3}{2}-\lambda}\int_0^\rho f'(s)
\\
&\quad\times\int_0^s\chi_{\lambda}(\rho)\frac{\O(t\langle\omega\rangle^{-\frac{5}{2}})}{(1-t^2)^{\frac{5}{4}-\frac{\lambda}{2}}} \widehat{\beta}_{7}(\rho,t,\lambda)
dt ds
\\
\rho^{-\frac{5}{2}} (1-\rho^2)^{\frac{1}{4}-\frac{\lambda}{2}}G_{18}(f)(\rho,\lambda)&=(1-\chi_{\lambda}(\rho))\rho^{-\frac{5}{2}}(1+\rho)^{\frac{3}{2}-\lambda}\int_0^\rho f'(s)\O(\langle\omega\rangle^{-1})
\\
&\quad\times\int_0^s(1-\chi_\lambda(t))t^{\frac{5}{2}}\left[\frac{\widehat{\gamma}_5(\rho,t,\lambda)}{(1+t)^{\frac{3}{2}-\lambda}}+\frac{\widehat{\gamma}_6(\rho,t,\lambda)}{(1-t)^{\frac{3}{2}-\lambda}}\right] dt ds
\end{align*}
where $\widehat{\beta}_j$ and $\widehat{\gamma}_j$ are defined as in Lemma \ref{V2symbol}.
\end{lem}
Once more, we define another set of operators $T_j$ and $\dot{T}_j$ for $j=10,\dots,18$ and $f\in C^2(\overline{\B^6_1}) $ as
\begin{align*}
T_j(\tau)f(\rho):=\frac{1}{2\pi i}\lim_{\varepsilon\to 0^+}\lim_{N \to \infty} \int_{\varepsilon-i N}^{\varepsilon+ i N}e^{\lambda\tau}
\rho^{-\frac{5}{2}} (1-\rho^2)^{\frac{1}{4}-\frac{\lambda}{2}}G_j(f)(\rho,\lambda) d \lambda
\end{align*}
and
\begin{align*}
\dot{T}_j(\tau)f(\rho):=\frac{1}{2\pi i}\lim_{\varepsilon\to 0^+}\lim_{N \to \infty} \int_{\varepsilon-i N}^{\varepsilon+ i N}\lambda e^{\lambda\tau}
\rho^{-\frac{5}{2}} (1-\rho^2)^{\frac{1}{4}-\frac{\lambda}{2}}G_j(f)(\rho,\lambda) d \lambda.
\end{align*}
\begin{lem}\label{kernel3}
	Let $p\in [2,\infty]$ and $q\in [6,12]$ be such that $\frac{1}{p}+\frac{6}{q}=1$.
	Then
	\begin{equation*}
	\|T_j(.)f\|_{L^p(\R_+)L^q(\B_1^6)}\lesssim\|f\|_{H^1(\B_1^6)}
	\end{equation*}
	and 
	\begin{equation*}
	\| \dot{T}_j(.)f\|_{L^p(\R_+)L^q(\B_1^6)}\lesssim\|f\|_{H^2(\B_1^6)}
	\end{equation*}
	for $j=10,11,\dots,18$ and all $f\in C^2(\overline{\B^6_1})$.
\end{lem}
\begin{proof}
We continue using the same strategy as before and note that \[|V_3(f)(\rho,\lambda)|\lesssim \langle \lambda \rangle^{-2}\rho^{-\frac{5}{2}}\|f\|_{W^{1,\infty}(\B^6_1)} .\] Thus, we can again take the limit $\varepsilon \to 0$ in the operators $T_j$ to obtain that 
\begin{align*}
T_j =\frac{1}{2\pi }\int_{\R}e^{i\omega\tau}
\rho^{-\frac{5}{2}} (1-\rho^2)^{\frac{1}{4}-\frac{i \omega}{2}}G_j(f)(\rho, i \omega) d \omega,
\end{align*}
for $j=10,11,\dots,18$.
For $T_{10}$ we also readily use Fubini's theorem to interchange the order of integration. This yields
\begin{align*}
T_{10}(\tau)f(\rho)=&\int_\rho^1  f'(s)\int_0^s \int_\R e^{i\omega \tau}\chi_{i \omega}(\rho) (1-\rho^2)^{\frac{1}{4}-\frac{i \omega}{2}}\chi_{i\omega}(t)\frac{\O(\rho^0 t^{\frac{11}{12}}\langle\omega\rangle^{-1-\frac{1}{12}})}{(1-t^2)^{\frac{5}{4}-\frac{i\omega}{2}}}d \omega dt ds
\end{align*}
and thus, as the $t$ integrand is only supported away from $1$,
\begin{align*}
|T_{10}(\tau)f(\rho)|&\lesssim\langle\tau\rangle^{-2} \int_\rho^1  |f'(s)|\int_0^s t^{\frac{11}{12}} dt ds
\\
&\lesssim \langle\tau\rangle^{-2}\int_\rho^1  |f'(s)|s^{2-\frac{1}{12}} ds,
\end{align*}
by Lemma \ref{osci1}.
Again, one can analogously obtain the estimate \[|\dot{T}_{10}(\tau)f(\rho)| \lesssim\langle\tau\rangle^{-2}\int_\rho^1  |f'(s)|s^{1-\frac{1}{12}} ds.\]
We can also slightly modify this strategy to obtain the bounds 
\[|T_j(\tau)f(\rho)|\lesssim \langle\tau\rangle^{-2}\int_\rho^1  |f'(s)|s^{2-\frac{1}{12}} ds\]
and
\[|\dot{T}_j(\tau)f(\rho)| \lesssim \langle\tau\rangle^{-2}\int_\rho^1  |f'(s)|s^{1-\frac{1}{12}} ds\]
for $j=16,17$.
For $j=11$ we have
\begin{align*}
T_{11}(\tau)f(\rho)&=\rho^{-\frac{5}{2}}\int_\rho^1f'(s)\int_0^s\int_\R e^{i\omega\tau}(1-\chi_{i\omega}(\rho))(1+\rho)^{\frac{3}{2}-i\omega}\chi_{\lambda}(t)
\\
 &\quad\times\bigg[\frac{\O(t\langle
\omega\rangle^{-\frac{5}{2}})\widehat{\beta}_4(\rho,t,\lambda)}{(1-t^2)^{\frac{5}{4}-\frac{\lambda}{2}}} 
\\
&\quad+
\frac{\O(t^3\langle\omega\rangle^{-2})[1+\O(\rho^{-1}(1-\rho)\langle\omega\rangle^{-1})]}{(1-t^2)^{\frac{5}{4}-\frac{i \omega}{2}}}\bigg]
d\omega dt ds
\end{align*}
and since the support of the $t $ integrand is again away from the endpoint $1$,  applying  Lemma \ref{osci3} yields
 \begin{align*}
\left|T_{11}(\tau)f(\rho)\right|&\lesssim \langle\tau \rangle^{-2}\int_\rho^1 |f'(s)| \int_0^s t dt ds
\\
&\lesssim \langle\tau \rangle^{-2} \int_\rho^1|f'(s)| s^2 ds.
 \end{align*}
In the same manner, we obtain the estimate
\begin{align*}
|\dot{T}_{11}(\tau)f(\rho)|\lesssim \langle \tau \rangle^{-2}\int_\rho^1 s|f'(s)| ds.
\end{align*}
The same strategy can again be used to obtain the bounds
\begin{align*}
|T_{12}(\tau)f(\rho)|\lesssim \langle \tau \rangle^{-2}\int_\rho^1 s^2|f'(s)| ds
\end{align*}
and
\begin{align*}
|\dot{T}_{12}(\tau)f(\rho)|\lesssim& \langle \tau \rangle^{-2}\int_\rho^1 s|f'(s)| ds.
\end{align*}
Thus, an application of Lemma \ref{teclem1} yields the desired Strichartz estimates for $j=10,11,12,16,17$.
For $j=13$, we integrate by parts to obtain
\begin{align*}
&\quad\int_0^s(1-\chi_{i\omega}(t))t^{\frac{5}{2}}\frac{\O(\rho^0 \langle\omega \rangle^{\frac{3}{2}})}{(1-t)^{\frac{3}{2}-i\omega}}\widehat{\beta}_6(t,\rho,i\omega) dt 
\\
&=(1-\chi_{i\omega}(s))s^{\frac{5}{2}}\frac{\O(\rho^0\langle\omega \rangle^{\frac{1}{2}})}{(1-t)^{\frac{1}{2}-i\omega}} \widehat{\beta}_6(s,\rho,\lambda)
+\O(\rho^0\langle\omega\rangle^{\frac{1}{2}})\int_0^s\frac{\partial_t\left[(1-\chi_{i\omega}(t))t^{\frac{5}{2}}\widehat{\beta}_6(t,\rho,i\omega)\right]}{(1-t)^{\frac{1}{2}-i\omega}} dt.
\end{align*}
So
\begin{align*}
T_{13}(\tau)f(\rho)&=\int_\R e^{\tau i\omega}\chi_{i \omega}(\rho)(1-\rho^2)^{\frac{1}{4}-\frac{i\omega}{2}}
\\
&\quad\times \int_\rho^1f'(s)\int_0^s(1-\chi_{i\omega}(t))t^{\frac{5}{2}}\frac{\O(\rho^0\langle\omega \rangle^{\frac{3}{2}})}{(1-t)^{\frac{3}{2}-i\omega}}\widehat{\beta}_6(t,\rho,i\omega) dt ds d \omega
\end{align*}
satisfies
\begin{align*}
T_{13}(\tau)f(\rho)&=\int_\rho^1 s^{\frac{5}{2}}f'(s)\int_\R e^{i\omega \tau}\chi_{i \omega}(\rho)(1-\rho^2)^{\frac{1}{4}-\frac{i \omega}{2}} \frac{(1-\chi_{i\omega}(s))}{(1-s)^{\frac{1}{2}-i\omega }}
\\
&\quad\times\O(\rho^{-\frac{11}{24}}\langle\omega\rangle^{\frac{1}{24}})\widehat{\beta}_6(s,\rho,i\omega) d \omega ds
\\
&\quad+\int_\rho^1 f'(s)\int_0^s \int_\R e^{i\omega \tau}\chi_{i \omega}(\rho)(1-\rho^2)^{\frac{1}{4}-\frac{i \omega}{2}} (1-t)^{-\frac{1}{2}+i\omega} \O(\rho^{-\frac{11}{24}}\langle\omega\rangle^{\frac{1}{24}})
\\
&\quad\times\partial_t\left[(1-\chi_{i\omega}(t))t^{\frac{5}{2}}\widehat{\beta}(t,\rho,i \omega)\right]d \omega dt ds \\
&=:T^1_{13}(\tau)f(\rho)+T^2_{13}(\tau)f(\rho).
\end{align*}
By using that $\widehat{\beta}_6(.,.,i\omega)$ behaves like $\O(\langle\omega\rangle^{-1})$ in $\omega$, Lemma \ref{osci2}, Minkoswki's inequality, and Lemma \ref{teclem3}, we obtain that
\begin{align*}
\|T_{13}^1(\tau)f\|_{L^{12}(\B^6_1)}&\lesssim \int_0^1 s^{\frac{5}{2}}|f'(s)|(1-s)^{-\frac{1}{2}}\langle\tau-\log(1-s)\rangle^{-2}
\\
&\quad\times\|\rho^{-\frac{11}{24}}|\tau-\frac{1}{2}\log(1-\rho^2)+\log(1-s)|^{-\frac{1}{24}}\|_{L_\rho^{12}(\B^6_1)} ds
\\
&\lesssim \int_0^1 s^{\frac{5}{2}}|f'(s)|(1-s)^{-\frac{1}{2}}\langle\tau+\log(1-s)\rangle^{-2}
|\tau+\log(1-s)|^{-\frac{1}{24}}ds.
\end{align*}
Thus, the computation in the proof of Proposition 4.6 in \cite{DonRao20} shows that
\begin{align*}
\|T_{13}^1(.)f\|_{L^2(\R_+) L^{12}(\B^6_1)}\lesssim \|f\|_{H^1(\B^6_1)}.
\end{align*}
A similar calculation shows
\begin{align*}
\|T_{13}^2(.)f\|_{L^2(\R_+) L^{12}(\B^6_1)}\lesssim \|f\|_{H^1(\B^6_1)}
\end{align*}
and as the other endpoint estimates for $T_{13}^1$ and $T_{13}^2$ can be obtained analogously, an interpolation argument yields
\begin{align*}
\|T_{13}(.)f\|_{L^p(\R_+) L^q(\B^6_1)}\lesssim \|f\|_{H^1(\B^6_1)}
\end{align*}
for any admissible pair $(p,q)$.
To derive the estimate
\begin{align}
\|\dot{T}_{13}(.)f\|_{L^p(\R_+) L^{q}(\B^6_1)}\lesssim \|f\|_{H^2(\B^6_1)}
\end{align}
we first note that we can again take the limit $\varepsilon\to 0$ and interchange the order of integration. The claim then essentially follows by performing one more integration by parts.  We illustrate this for the boundary part $\dot{T}^1_{13}(\tau)f(\rho)$, given  by 
\begin{align}\label{dotT113}
\dot{T}^1_{13}(\tau)f(\rho)&=\int_\rho^1 s^{\frac{5}{2}}f'(s)\int_\R i \omega e^{i\omega \tau}\chi_{i \omega}(\rho)(1-\rho^2)^{\frac{1}{4}-\frac{i \omega}{2}} \frac{(1-\chi_{i\omega}(s))}{(1-s)^{\frac{1}{2}-i\omega }}\nonumber
\\
&\quad\times\O(\rho^{-\frac{11}{24}}\langle\omega\rangle^{\frac{1}{24}})\widehat{\beta}_6(s,\rho,i\omega) d \omega ds\nonumber
\\
&=-f'(\rho)\int_\R  e^{i\omega \tau}\chi_{i \omega}(\rho)(1-\rho^2)^{\frac{1}{4}-\frac{i \omega}{2}} (1-\chi_{i\omega}(\rho))(1-\rho)^{\frac{1}{2}+i\omega }\nonumber
\\
&\quad\times \O(\rho^{\frac{5}{2}-\frac{11}{24}}\langle\omega\rangle^{\frac{1}{24}})\widehat{\beta}_6(\rho,\rho,i \omega) d \omega \nonumber
\\
&\quad+ \int_\rho^1 \int_\R e^{i\omega \tau}\chi_{i \omega}(\rho)(1-\rho^2)^{\frac{1}{4}-\frac{i \omega}{2}} (1-s)^{\frac{1}{2}+i\omega }
\O(\rho^{-\frac{11}{24}}\langle\omega\rangle^{\frac{1}{24}}) \nonumber
\\
&\quad \times\partial_s\left[(1-\chi_{i\omega}(s))s^{\frac{5}{2}}f'(s)\widehat{\beta}(s,\rho,i\omega)\right]d \omega ds
\end{align} 
We now readily calculate, that the boundary term in \eqref{dotT113} can be bounded in absolute value by
$\rho^{2}|f'(\rho)|$, while the integral term can by controlled by the same means as $T_{13}$.
Consequently, an application of Lemma \ref{teclem1} establishes the estimates on $\dot{T}_{13}$ and 
we continue with $T_{14}$. For $T_{14}$ we can do the same manipulations as with $T_{13}$ which yields
\begin{align*}
T_{14}(\tau)f(\rho)&= \rho^{-\frac{5}{2}}\int_\rho^1 f'(s)\int_\R e^{i\tau\omega}(1-\chi_{i\omega}(\rho))(1+\rho)^{\frac{3}{2}-i \omega}
\\
&\quad\times\frac{(1-\chi_{i \omega}(s))s^{\frac{5}{2}}\O(\langle\omega\rangle^{-2})}{(1-s)^{\frac{1}{2}-i \omega}}\widehat{\gamma}_3(\rho,s,i \omega) d\omega ds 
\\
&\quad+ \rho^{-\frac{5}{2}}\int_\rho^1 f'(s) \int_0^s\int_\R e^{i\tau\omega}(1+\rho)^{\frac{3}{2}-i \omega}(1-\chi_{i\omega}(\rho))\O(\langle\omega\rangle^{-2})
\\
&\quad\times(1-t)^{-\frac{1}{2}+i \omega}\partial_t\left[(1-\chi_{i \omega}(t))t^{\frac{5}{2}}\widehat{\gamma}_3(\rho,t,i \omega)\right]d \omega dt ds
\\
&=:T_{14}^1(\tau)f(\rho)+T_{14}^2(\tau)f(\rho).
\end{align*}
Now, the $s$-integrand in $T^1_{14}$
given by \[
f'(s)\int_\R e^{i\tau\omega}(1-\chi_{i\omega}(\rho))(1+\rho)^{\frac{3}{2}-i \omega}
\frac{(1-\chi_{i \omega}(s))s^{\frac{5}{2}}\O(\langle\omega\rangle^{-2})}{(1-s)^{\frac{1}{2}-i \omega}}\widehat{\gamma}_3(\rho,s,i \omega) d\omega=:I_{14}^1(\rho,s,\tau)
\]
satisfies the two estimates
\begin{align*}
\left|I_{14}^1(\rho,s,\tau)\right|\lesssim \rho^{2} s^{\frac{5}{2}}|f'(s)|(1-s)^{-\frac{1}{2}}
\langle\tau-\log(1+\rho)+\log(1-s)\rangle^{-2}ds
\end{align*}
and
\begin{align*}
|I_{14}^1(\rho,s,\tau)|&\lesssim \rho^{3} s^{\frac{5}{2}}|f'(s)|(1-s)^{-\frac{1}{2}}
\\
&\quad\times\langle\tau-\log(1+\rho)+\log(1-s)\rangle^{-2}|\tau-\log(1+\rho)+\log(1-s)|^{-1}
ds,
\end{align*}
thanks to Lemmas \ref{osci3} and \ref{osci4}.
From these estimates, we infer that
\begin{align*}
|T_{14}^1(\tau)f(\rho)|&\lesssim \rho^{-\frac{11}{24}}\int_0^1 s^{\frac{5}{2}}|f'(s)|(1-s)^{-\frac{1}{2}}
\\
&\quad\times\langle\tau+\log(1-s)\rangle^{-2}|\tau-\log(1+\rho)+\log(1-s)|^{-\frac{1}{24}} ds.
\end{align*}
So, we can again use Minkoswki's inequality as well as Lemma \ref{teclem2} to conclude that
\begin{align*}
\|T_{14}^1(\tau)f\|_{L^{12}(\B^6_1)}\lesssim&\int_0^1 s^{\frac{5}{2}}|f'(s)|(1-s^2)
\langle\tau+\log(1-s)\rangle^{-2}|\tau+\log(1-s)|^{-\frac{1}{24}} ds .
\end{align*}
Again, by slightly modifying our argument above, we obtain the full estimate
\begin{align*}
\|T_{14}(.)f\|_{L^p(\R_+) L^{q}(\B^6_1)}\lesssim \|f\|_{H^1(\B^6_1)}
\end{align*}
and one more integration by parts combined with similar reasoning also yields
\begin{align}
\|\dot{T}_{14}(.)f\|_{L^p(\R_+) L^{q}(\B^6_1)}\lesssim \|f\|_{H^2(\B^6_1)}.
\end{align}
As the Strichartz estimates for $T_{15}$ and $\dot{T}_{15}$
can be obtained similarly, only $T_{18}$ and $\dot{T}_{18}$ remain to be dealt with.
Therefore, note that
\begin{align*}
T_{18}(\tau)f(\rho)&=\rho^{-\frac{5}{2}}\int_0^\rho f'(s)\int_0^s\int_\R e^{i \omega \tau }(1-\chi_{i \omega}(\rho))(1-\rho)^{\frac{3}{2}-i\omega}(1-\chi_{i \omega}(t))t^{\frac{5}{2}}
\\
&\quad\times\big[(1+t)^{-\frac{3}{2}+i \omega}\O(\langle\omega\rangle^{-1})\widehat{\gamma}_5(\rho,t,i\omega)
\\
&\quad+(1-t)^{-\frac{3}{2}+i \omega} \O(\langle\omega\rangle^{-1})\widehat{\gamma}_6(\rho,t,i\omega)\big] d\omega d t ds.
\end{align*}

Thus, as $\widehat{\gamma}_j(.,.,i \omega) $ behaves like $\O(\langle\omega\rangle^{-1})$ in $\omega$, applying  Lemma \ref{osci3} leads to the estimate
\begin{align*}
|T_{18}(\tau)f(\rho)|&\lesssim\rho^{-\frac{3}{2}}\int_0^\rho|f'(s)|s^{\frac{5}{2}}\int_0^s
\langle\tau-\log(1-\rho)+\log(1-t)\rangle^{-2}
\\
&\quad+\langle\tau-\log(1-\rho)+\log(1+t)\rangle^{-2} dt ds
\\
&\lesssim\langle \tau \rangle^{-2}\rho^{-\frac{1}{2}}\int_0^\rho |f'(s)|s^{\frac{3}{2}}\int_0^s 1 dt ds
\\
&\lesssim \langle \tau \rangle^{-2}\|f\|_{H^1(\B^6_1)},
\end{align*}
where the last step follows from the Cauchy-Schwarz inequality.
Finally, to control $\dot{T}_{18}$, one has to perform an integration by parts in the $t$ integral and then do a similar calculation.
\end{proof}
Due to these Lemmas we can now easily prove a first set of Strichartz estimates for the semigroup $\Sf$.
\begin{prop}\label{strichatzestimate1}
Let $p\in [2,\infty]$ and $q\in [6,12]$ be such that $\frac{1}{p}+\frac{6}{q}=1$. Then the estimate
\begin{equation*}
\left\| [\Sf(\tau)(\I-\Pf)\ff]_1\right\|_{L^p_\tau(\R_+)L^q(\B_1^6)}\lesssim\|(\I-\Pf)\ff\|_{\mathcal{H}}
\end{equation*}
holds true for all $\ff\in \mathcal{H}$. Furthermore, we also have
\begin{equation*}
\left\| \int_0^\tau [\Sf(\tau-\sigma)(\I-\Pf)\hfh(\sigma,.)]_1 d \sigma\right\|_{L^p_\tau(I)L^q(\B_1^6)} \lesssim\|(\I-\Pf)\hfh\|_{L^1(I)\mathcal{H}}
\end{equation*}
for any interval $0\in I\subset[0,\infty) $ and all $\hfh \in L^1(\R_+,\mathcal{H}).$
\end{prop}
\begin{proof}
Due to the identity
\begin{align*}
[\Sf(\tau) (\I-\Pf)\ff]_1(\rho)&=[\Sf_0(\tau)(\I-\Pf)\ff]_1(\rho)
\\
&\quad+\frac{1}{2\pi i}\lim_{N \to \infty} \int_{\varepsilon-i N}^{\varepsilon+ i N}e^{\lambda\tau}(\mathcal{R}(F_\lambda)(\rho,\lambda)-\mathcal{R}_{\mathrm{f}}(F_\lambda)(\rho,\lambda)) d\lambda,
\end{align*}
the homogeneous estimate follows from the ``free'' Strichartz estimates established in Lemma \ref{Strichart} combined with Lemmas \ref{kernel1}, \ref{kernel2} and \ref{kernel3} via a density argument.
The inhomogeneous one can then be established as previously done in the proof of Lemma \ref{Strichart}.
\end{proof}
\section{More Strichartz estimates}
The Strichartz estimates from the previous section will not be sufficient to control the full nonlinear equation. However, we can remedy this by also deriving the Strichartz estimate
\begin{equation}\label{strest}
\|[\Sf(\tau)(\I-\Pf)\ff]_1\|_{L^2_\tau(\R_+)\dot{W}^{1,4}(\B^6_1)}\lesssim
\|(\I-\Pf)\ff\|_{\mathcal H}.
\end{equation}
Our strategy to derive this estimate works for all pairs $(p,q)$ with $p\in [2,\infty)$ and $q\in(3,4]$. However, as we only need the pair $(p,q)=(2,4)$ we only prove the estimate \eqref{strest}. For the endpoint pair $(p,q)=(\infty,3)$ one would have to take into account the parity of the $\O(\langle\omega\rangle^{-1})$ terms, which appear throughout this paper.
To establish \eqref{strest}, we will have to interchange limits and differentiation with oscillatory integrals which are not absolutely convergent. To do so we will make use of the following lemma and variations thereof.
\begin{lem}\label{difinterchange}
Let $f(\omega)=\O\left(\langle\omega\rangle^{-1-\alpha}\right)$ with $\alpha > 0$.
Then
\begin{align*}
\partial_a \int_\R e^{i a \omega } f(\omega) d \omega= i \int_\R \omega e^{i a \omega}f(\omega) d \omega
\end{align*}
for $a \in \R\setminus\{0\}$.
\end{lem}
\begin{proof}
This can be shown as Lemma 5.1 in \cite{Don17}.
\end{proof}
We will also need one more technical Lemma.
\begin{lem}\label{teclem5}
The estimates
\begin{align*}
\left\||.|^{\frac{3}{4}}f\right\|_{L^{4}(\B^6_1)}\lesssim \|f\|_{H^1(\B^6_1)}
\end{align*}
and 
\begin{align*}
\left\|\int_\rho^1 s^{\frac{3}{4}}f'(s)ds\right\|_{L^{4}_\rho(\B^6_1)}\lesssim \|f\|_{H^1(\B^6_1)}
\end{align*}
hold true for all $f\in C^1(\overline{\B^6_1})$.
Further, also the estimate
\begin{align*}
\left\|f\right\|_{L^{4}(\B^6_1)}\lesssim \|f\|_{H^2(\B^6_1)}
\end{align*}
is true for all $f\in C^2(\overline{\B^6_1})$.
\end{lem}
\begin{proof}
This follows in the same way as Lemma \ref{teclem1}.
\end{proof}
Suppose again that $\ff \in C^3\times C^2(\overline{\B^6_1})$ and set $\widetilde{\ff}=(\I-\Pf)\ff$. Then, by using variations of Lemma \ref{difinterchange}, we obtain
\begin{align*}
\partial_\rho[\Sf(\tau)\widetilde{\ff}]_1(\rho)=\partial_\rho[\Sf_0(\tau)\widetilde{\ff}]_1(\rho)+\frac{1}{2\pi i}\lim_{N \to \infty}\int_{\varepsilon-iN}^{\varepsilon+ i N} e^{\lambda \tau}\partial_\rho[\mathcal{R}(F_\lambda)(\rho,\lambda)-\mathcal{R}_{\mathrm{f}}(F_\lambda) (\rho,\lambda) ]d \lambda
\end{align*}
for any $\varepsilon >0$.
To be more precise, we cannot apply Lemma \ref{difinterchange} straight away, but instead have to perform integrations by parts to obtain enough decay in $\omega$. After that, we can apply the Lemma and undo the integrations by parts.
Recall that $\mathcal{R}(F_\lambda)(\rho,\lambda)$ is of the form
\begin{align}
\mathcal{R}(F_\lambda)(\rho,\lambda)&= u_0(\rho,\lambda)\left(c(F_\lambda,\lambda)-U_1(\rho,\lambda)F_\lambda(\rho)-\int_\rho^{1}U_1(s,\lambda)F'_\lambda(s)ds \right)\nonumber
\\
&\quad+u_1(\rho,\lambda)\left(U_0(\rho,\lambda)F_\lambda(\rho)- \int_0^\rho U_0(s,\lambda)F'_\lambda(s) d s\right)
\end{align}
and therefore 
\begin{align}\label{u'}
\mathcal{R}(F_\lambda)'(\rho,\lambda)&= u'_0(\rho,\lambda)\left(c(F_\lambda,\lambda)-U_1(\rho,\lambda)F_\lambda(\rho)-\int_\rho^{1}U_1(s,\lambda)F'_\lambda(s)ds \right)\nonumber
\\
&\quad+u_1'(\rho,\lambda)\left(U_0(\rho,\lambda)F_\lambda(\rho)- \int_0^\rho U_0(s,\lambda)F'_\lambda(s) d s\right).
\end{align}
Analogously,
\begin{align*}
\mathcal{R}_{\mathrm{f}}(F_\lambda)'(\rho,\lambda)=& u_{\mathrm{f}_0}'(\rho,\lambda)\left(c_{\mathrm{f}}(F_\lambda,\lambda)-U_{\mathrm{f}_1}(\rho,\lambda)F_\lambda(\rho)-\int_\rho^{1}U_{\mathrm{f}_1}(s,\lambda)F'_\lambda(s)ds \right)
\\
&+u_{\mathrm{f}_1}'(\rho,\lambda)\left(U_{\mathrm{f}_0}(\rho,\lambda)F_\lambda(\rho)- \int_0^\rho U_{\mathrm{f}_0}(s,\lambda)F'_\lambda(s) d s\right).
\end{align*}
We will now proceed in a similar fashion as before.
Motivated by this let $f\in C^2(\overline{\B^6_1})$ and consider
\[V_1'(f)(\rho,\lambda):
=u'_0(\rho,\lambda)c(f,\lambda)-u'_{\mathrm{f}_0}(\rho,\lambda)c_{\mathrm{f}}(f,\lambda).
\]
\begin{lem}\label{decompV'}
We can decompose $V_1'$ as
\[V_1'(f)(\rho,\lambda)=f(1)\left(H_1(\rho,\lambda)+H_2(\rho,\lambda)\right)
\]
with
\begin{align*}
H_1(\rho,\lambda)=&\left[\rho^{-1}+\rho(1-\rho^2)^{-1}\O(\langle\omega\rangle) \right]\chi_\lambda(\rho)(1-\rho^2)^{\frac{1}{4}-\frac{\lambda}{2}}\O(\rho^0\langle\omega\rangle^{-\frac{3}{2}})
\\
H_2(\rho,\lambda)=&\left[\O(\rho^{-1}(1-\rho)^{-1})+(1-\rho)^{-1}\O(\langle\omega\rangle)\right](1-\chi_{\lambda}(\rho))\rho^{-\frac{5}{2}}(1-\rho)^{\frac{3}{2}-\lambda}
\\
&\times[\O(\langle\omega\rangle^{-4})+
\O(\rho^{0}(1-\rho)\langle\omega\rangle^{-4})]
\\
&+\left[\O(\rho^{-1})+(1+\rho)^{-1}\O(\langle\omega\rangle)\right](1-\chi_{\lambda}(\rho))\rho^{-\frac{5}{2}}(1+\rho)^{\frac{3}{2}-\lambda}
\\
&\times[\O(\langle\omega\rangle^{-4})+
\O(\rho^{0}(1-\rho)\langle\omega\rangle^{-4})]
\\
&+(1-\chi_{\lambda}(\rho))\rho^{-\frac{5}{2}}(1+\rho)^{\frac{3}{2}-\lambda}\O(\rho^{-1}(1-\rho)^0\langle\omega\rangle^{-4}).
\end{align*}
\end{lem}
\begin{proof}
This follows immediately from the decomposition in Lemma \ref{decomp1} and the corresponding symbol forms given in Lemma \ref{symbol1}.
\end{proof}
In accordance with our previous strategy, we define the operators $S_j $ and $\dot{S}_j$ for $j=1,2$ and $f\in C^2(\overline{\B^6_1})$  as
\begin{align*}
S_j(\tau)f(\rho):=\frac{1}{2\pi i}\lim_{\varepsilon\to 0^+}\lim_{N \to \infty} \int_{\varepsilon-i N}^{\varepsilon+ i N}e^{\lambda\tau}
f(1)H_j(\rho,\lambda) d \lambda
\end{align*}
and
\begin{align*}
\dot{S}_j(\tau)f(\rho):=\frac{1}{2\pi i}\lim_{\varepsilon\to 0^+}\lim_{N \to \infty} \int_{\varepsilon-i N}^{\varepsilon+ i N}\lambda e^{\lambda\tau} f(1)
H_j(\rho,\lambda) d \lambda.
\end{align*}

\begin{lem}\label{Sf1}
The estimates
	\begin{align*}
	\|S_j(.)f\|_{L^2(\R_+)L^4(\B^6_1)}\lesssim& \|f\|_{H^1(\B^6_1)}
	\end{align*}
	and 
	\begin{align*}
	\|\dot{S}_j(.)f\|_{L^2(\R_+)L^4(\B^6_1)}\lesssim& \|f\|_{H^2(\B^6_1)}
	\end{align*}
	hold for $j=1,2$ and all $f \in C^2(\overline{\B^6_1})$.
\end{lem}
\begin{proof}
	This Lemma can be established in a very similar fashion as Lemma \ref{kernel1}, by making use of Lemma \ref{teclem5}.
\end{proof}
We proceed in the same manner and use the previous decomposition of $V_2(f,\lambda)$ to decompose
\begin{align*}
V'_2(f)(\rho,\lambda):=&-u_0'(\rho,\lambda)U_1(\rho,\lambda)f(\rho)+u_1'(\rho,\lambda)U_0(\rho,\lambda)f(\rho)
\\
&+u_{\mathrm{f}_0}'(\rho,\lambda)U_{\mathrm{f}_1}(\rho,\lambda)f(\rho)-u_{\mathrm{f}_1}'(\rho,\lambda)U_{\mathrm{f}_0}(\rho,\lambda)f(\rho).
\end{align*}
\begin{lem}\label{lem:decmopv_2'}
We can decompose $V_2'$ as
\begin{align*}
V_2'(f)(\rho,\lambda)=f(\rho)\sum_{j=3}^{9}  H_j(\rho,\lambda)
\end{align*}
with
\begin{align*}
H_3(\rho,\lambda)=&\left[\O(\rho^{-1})+\rho(1-\rho^2)^{-1}\O(\langle\omega\rangle)\right]\widetilde{G}_3(\rho,\lambda) 
\\
H_4(\rho,\lambda)=&\left[\O(\rho^{-1})+(1+\rho)^{-1}\O(\langle\omega\rangle)\right]\widetilde{G}_4(\rho,\lambda)+\widehat{G}_4(\rho,\lambda)
\\
H_5(\rho,\lambda)=&\left[\O(\rho^{-1})+(1-\rho)^{-1}\O(\langle\omega\rangle)\right]\widetilde{G}_5(\rho,\lambda)+\widehat{G}_5(\rho,\lambda)
\\
H_6(\rho,\lambda)=&\left[\O(\rho^{-1})+(1-\rho)^{-1}\O(\langle\omega\rangle)\right] \widetilde{G}_6(\rho,\lambda)+\widehat{G}_6(\rho,\lambda)
\\
H_7(\rho,\lambda)=&\left[\O(\rho^{-1})+\rho(1-\rho^2)^{-1}\O(\langle\omega\rangle)\right]\widetilde{G}_7(\rho,\lambda)
\\
H_8(\rho,\lambda)=&\left[\O(\rho^{-1})+(1+\rho)\O(\langle\omega\rangle)\right] 
\widetilde{G}_8(\rho,\lambda)+\widehat{G}_8(\rho,\lambda)
\\
H_9(\rho,\lambda)=&\left[\O(\rho^{-1})+(1+\rho)\O(\langle\omega\rangle)\right] \widetilde{G}_9(\rho,\lambda)+\widehat{G}_9(\rho,\lambda),
\end{align*}
where all the $\widetilde{G}_j $ are of the same forms as $\rho^{-\frac{5}{2}}(1-\rho^2)^{\frac{1}{4}-\frac{\lambda}{2}}G_j(\rho,\lambda)$
and where the $\widehat{G}_j$ are the terms obtained when the $\rho$ derivative hits one of the perturbative parts $\widehat{\beta}_j,\widehat{\gamma}_j$ (or $\O(\rho^0(1-\rho)\langle\omega\rangle^{-1})$ in the case of $j=3,4$) as given in Lemma \ref{V2symbol}.
\end{lem}
We continue by defining operators $S_j$ and $\dot{S}_j$ as
\begin{align*}
S_j(\tau)f(\rho)=\frac{1}{2\pi i}\lim_{\varepsilon\to 0^+}\lim_{N \to \infty}\int_{\varepsilon-i N}^{\varepsilon+ i N}e^{\lambda\tau}
f(\rho)H_j(\rho,\lambda) d \lambda,
\end{align*}

\begin{align*}
\dot{S}_j(\tau)f(\rho)=\frac{1}{2\pi i}\lim_{\varepsilon\to 0^+}\lim_{N \to \infty} \int_{\varepsilon-i N}^{\varepsilon+ i N}\lambda e^{\lambda\tau}
f(\rho)H_j(\rho,\lambda) d \lambda
\end{align*}
for $j=3,\dots,9$ and    $f\in C^2(\overline{\B^6_1})$
\begin{lem}
The estimates
	\begin{align*}
	\|S_j(.)f\|_{L^2(\R_+)L^4(\B^6_1)}\lesssim& \|f\|_{H^1(\B^6_1)}
	\end{align*}
	and 
	\begin{align*}
	\|\dot{S}_j(.)f\|_{L^2(\R_+)L^4(\B^6_1)}\lesssim& \|f\|_{H^2(\B^6_1)}
	\end{align*}
hold for $j=3,4,\dots,9$ and $f\in C^2(\overline{\B_1^6}).$
\end{lem}
\begin{proof}
For $j=3,4,5,7,8$ we can use the arguments which we already used in the proofs of Lemmas \ref{kernel1} and \ref{kernel2} combined with Lemma \ref{teclem5} to derive both sets of Strichartz estimates.
For $S_6$, we remark, that the considerations in the proof of Lemma \ref{kernel2} together with Lemma \ref{teclem5}
imply the estimate 
\begin{align*}
\|S_6(.)f\|_{L^2(\R_+)L^4(\B^6_1)}\lesssim \|f\|_{H^1(\B^6_1)}.
\end{align*}
To deal with $\dot{S}_6$, we first note that after an integration by parts in the $s$-integral, we obtain
$$|\lambda H_6(\rho,\lambda)|\lesssim \rho^{-1}\langle\omega\rangle^{-1}$$
for $\rho\in(0,1)$ and $ \omega\in \R$.
Hence, a variant of Lemma \ref{difinterchange} (to be more precise we again first have to integrate by parts, then take the limit and then undo the integration) yields
\begin{align*}
\lim_{\varepsilon\to 0^+}\lim_{N \to \infty} \int_{\varepsilon-i N}^{\varepsilon+ i N}\lambda e^{\lambda\tau}
f(\rho)H_6(\rho,\lambda) d \lambda=\int_\R i \omega e^{i \omega\tau}
f(\rho)H_6(\rho,i\omega) d \omega.
\end{align*}
Recall that after an integration by parts, the integral term in the multiplier $G_6$ takes the form
\begin{align*}
B_6(\rho,\lambda)+I_6(\rho,\lambda)=&(1-\chi_\lambda(\rho))\rho^{\frac{5}{2}}\frac{\O(\langle\omega\rangle^{-2})}{(1-\rho)^{\frac{1}{2}-\lambda}}\widehat{\gamma}_1(\rho,\rho,\lambda)
\\
&-\O(\langle\omega\rangle^{-2})\int_0^\rho(1-s)^{-\frac{1}{2}+\lambda}\partial_s\left[(1-\chi_\lambda(s))s^{\frac{5}{2}}\widehat{\gamma}_1(\rho,s,\lambda)\right] ds.
\end{align*}
Therefore, a second integration by parts yields
\begin{align*}\label{g5}
H_6(\rho,\lambda)=&\left[\O(\rho^{-1})+\O((1-\rho)^{-1}\langle\omega\rangle)\right](1-\chi_\lambda(\rho))\rho^{-\frac{5}{2}}(1-\rho)^{\frac{3}{2}-\lambda}
\\
&\times\bigg[(1-\chi_\lambda(\rho))\rho^{\frac{5}{2}}\O(\langle\omega\rangle^{-2})(1-\rho)^{-\frac{1}{2}+\lambda}\widehat{\gamma}_1(\rho,\rho,\lambda)
\\
&-\O(\langle\omega\rangle^{-3})(1-\rho)^{\frac{1}{2}+\lambda}\partial_\rho\left[(1-\chi_\lambda(\rho))\rho^{\frac{5}{2}}\widehat{\gamma}_1(\rho,\rho,\lambda)\right]\\
&+\O(\langle\omega\rangle^{-3})\int_0^\rho(1-s)^{\frac{1}{2}+\lambda}\partial_s^2\left[(1-\chi_\lambda(s))s^{\frac{5}{2}}\widehat{\gamma}_1(\rho,s,\lambda)\right] ds
\bigg]+\widehat{G}_6(\rho,\lambda)
\\
=&:B_{6,1}(\rho,\lambda)+B_{6,2}(\rho,\lambda)+I_{6,1}(\rho,\lambda)+\widehat{G}_6(\rho,\lambda).
\end{align*}
We now use Lemma \ref{osci4} to obtain
\begin{align*}
\left|f(\rho)\int_\R \omega e^{i \omega \tau}B_{6,1}(\rho,i \omega) d \omega\right|\lesssim |f(\rho)||\tau|^{-\frac{1}{4}}\langle\tau\rangle^{-2}.
\end{align*}
Using this estimate we can easily infer that
\begin{align*}
\left\|f(\rho)\int_\R \omega e^{i \omega \tau}B_{6,1}(\rho,i \omega) d \omega\right\|_{L^2(\R_+)L^4(\B^6_1)}\lesssim \|f\|_{H^2(\B^6_1)} .
\end{align*}
 As the the oscillatory integral corresponding to $B_{6,2}(\rho,i \omega)$, $I_{6,1}(\rho,i \omega)$ and $\widehat{G}_6(\rho,\lambda)$ can be bounded similarly, we obtain
\begin{align*}
\|\dot{S}_6(.)f\|_{L^2(\R_+)L^4(\B^6_1)}\lesssim \|f\|_{H^2(\B^6_1)}.
\end{align*}
Since $S_9$ and $\dot{S}_9$ can be bounded likewise we conclude the proof of this Lemma.
\end{proof}
Our very last decomposition is the one of 
\begin{align*}
V_3'(f)(\rho,\lambda):=&-u_0'(\rho,\lambda)\int_\rho^{1}U_1(s,\lambda)f'(s)ds +u_1'(\rho,\lambda)\int_0^\rho U_0(s,\lambda)f'(s)ds 
\\
&+u'_{\mathrm{f}_0}(\rho,\lambda)\int_\rho^{1}U_{\mathrm{f}_1}(s,\lambda)f'(s)ds -u'_{\mathrm{f}_1}(\rho,\lambda)\int_0^\rho U_{\mathrm{f}_0}(s,\lambda)f'(s)ds.
\end{align*}
\begin{lem}
We can decompose $V'_3(f)(\rho,\lambda)$ as
\[
V'_3(f)(\rho,\lambda)=\sum_{j=10}^{18} H_j(f)(\rho,\lambda)
\]
with
\begin{align*}
H_{10}(f)(\rho,\lambda)=&\left[\O(\rho^{-1}) +\rho(1-\rho^2)^{-1} \O\left(\langle\omega\rangle\right)\right]\widetilde{G}_{10}(f)(\rho,\lambda)
\\
H_{11}(f)(\rho,\lambda)=&\left[\O(\rho^{-1})+(1+\rho)^{-1}\O(\langle\omega\rangle)\right]\widetilde{G}_{11}(f)(\rho,\lambda)+\widehat{G}_{11}(\rho,\lambda)
\\
H_{12}(f)(\rho,\lambda)=&\left[\O(\rho^{-1})+(1-\rho)^{-1}\O(\langle\omega\rangle)\right]\widetilde{G}_{12}(f)(\rho,\lambda)+\widehat{G}_{12}(\rho,\lambda)
\\
H_{13}(f)(\rho,\lambda)=&\left[\O(\rho^{-1}) +\rho(1-\rho^2)^{-1}\O( \langle\omega\rangle)\right]\widetilde{G}_{13}(f)(\rho,\lambda)
\\
H_{14}(f)(\rho,\lambda)=&\left[\O(\rho^{-1})+(1+\rho)^{-1}\O(\langle\omega\rangle)\right]
\widetilde{G}_{14}(f)(\rho,\lambda)+\widehat{G}_{14}(\rho,\lambda)
\\
H_{15}(f)(\rho,\lambda)=&\left[\O(\rho^{-1})+(1-\rho)^{-1}\O(\langle\omega\rangle)\right]
\widetilde{G}_{15}(f)(\rho,\lambda)+\widehat{G}_{15}(\rho,\lambda)
\\
H_{16}(f)(\rho,\lambda)=&\left[\O(\rho^{-1}) +\rho(1-\rho^2)^{-1} \O(\langle\omega\rangle)\right]\widetilde{G}_{16}(f)(\rho,\lambda)
\\
H_{17}(f)(\rho,\lambda)=&
\left[\O(\rho^{-1}) +(1+\rho)^{-1} \O\left(\langle\omega\rangle\right)\right]\widetilde{G}_{17}(f)(\rho,\lambda)+\widehat{G}_{17}(\rho,\lambda)
\\
H_{18}(f)(\rho,\lambda)=&\left[\O(\rho^{-1}) +(1+\rho)^{-1} \O(\left\langle\omega\rangle\right)\right]\widetilde{G}_{18}(f)(\rho,\lambda)+\widehat{G}_{18}(\rho,\lambda),
\end{align*} 
and where all the kernels $\widetilde{G}_j(f)(\rho,\lambda)$ are of the same form as $\rho^{-\frac{5}{2}}(1-\rho^2)^{\frac{1}{4}-\frac{\lambda}{2}}G_j(f)(\rho,\lambda)$ and the $\widehat{G}_j$ are defined as in Lemma \ref{lem:decmopv_2'}.
\end{lem}
One last time we define operators $S_j$ and $\dot{S}_j$ for $j=10,\dots,18$ and $f\in C^2(\overline{\B^6_1})$ as
\begin{align*}
S_j(\tau)f(\rho):=\frac{1}{2 \pi i}\lim_{\varepsilon\to 0^+}\lim_{N \to \infty} \int_{\varepsilon-i N}^{\varepsilon+ i N}e^{\lambda\tau}
H_j(f)(\rho,\lambda) d \lambda
\end{align*}
and
\begin{align*}
\dot{S}_j(\tau)f(\rho):=\frac{1}{2\pi i}\lim_{\varepsilon\to 0^+}\lim_{N \to \infty} \int_{\varepsilon-i N}^{\varepsilon+ i N}\lambda e^{\lambda\tau}
H_j(f)(\rho,\lambda) d \lambda.
\end{align*}
\begin{lem}
The estimates
	\begin{align*}
	\|S_j(.)f\|_{L^2(\R_+)L^4(\B^6_1)}\lesssim& \|f\|_{H^1(\B^6_1)}
	\end{align*}
	and 
	\begin{align*}
	\|\dot{S}_j(.)f\|_{L^2(\R_+)L^4(\B^6_1)}\lesssim& \|f\|_{H^2(\B^6_1)}
	\end{align*}
hold for $j=10,\dots,18$ and all $f \in C^2(\overline{\B^6_1})$.
\end{lem}
\begin{proof}
For $j=10,13,16,17$ we can derive both Strichartz estimates in almost the same fashion as we did in Lemma \ref{kernel3}.
For the remaining indices we can use an adapted version of the strategy which we used to bound $S_6$ and $\dot{S}_6$.
\end{proof}
Thanks to these last few Lemmas, we easily obtain the final result of this section.
\begin{prop}\label{strichatzestimate2}
The estimate
	\begin{equation*}
	\left\| [\Sf(\tau)(\I-\Pf)\ff]_1\right\|_{L^2_\tau(\R_+)\dot{W}^{1,4}(\B_1^6)}\lesssim\|(\I-\Pf)\ff\|_{\mathcal{H}}
	\end{equation*}
	holds true for all $\ff\in \mathcal{H}$. Furthermore, we also have
	\begin{equation*}
	\left\| \int_0^\tau [\Sf(\tau-\sigma)(\I-\Pf)\hfh(\sigma,.)]_1 d \sigma\right\|_{L^2_\tau(I)\dot{W}^{1,4}(\B_1^6)} \lesssim\|(\I-\Pf)\hfh\|_{L^1(I)\mathcal{H}},
	\end{equation*}
	for any $0\in I\subset [0,\infty)$  and  $\hfh \in L^1(\R_+,\mathcal{H}).$
	Finally, for the second component of the semigroup $\Sf$ we have the estimates
\begin{equation*}
	\left\| [\Sf(\tau)(\I-\Pf)\ff]_2\right\|_{L^2_\tau(\R_+)L^4(\B_1^6)}\lesssim\|(\I-\Pf)\ff\|_{\mathcal{H}}
	\end{equation*}
and
	\begin{equation*}
	\left\| \int_0^\tau [\Sf(\tau-\sigma)(\I-\Pf)\hfh(\sigma,.)]_2 d \sigma\right\|_{L^2_\tau(I)L^4(\B_1^6)} \lesssim\|(\I-\Pf)\hfh\|_{L^1(I)\mathcal{H}}
	\end{equation*}
	for $\ff\in \mathcal{H}$, $0\in I\subset [0,\infty)$, and $\hfh \in L^1(\R_+,\mathcal{H}).$
\end{prop}
\begin{proof}
The estimates on the first component can again be achieved as before. For the bounds on the second component we observe that by \eqref{coordinate} we have that
\begin{align*}
[\Sf(\tau)(\I-\Pf)\ff]_2(\rho)= (\partial_\tau +\rho\partial_\rho+1)[\Sf(\tau)(\I-\Pf)\ff]_1(\rho).
\end{align*}
As $\partial_\tau $ produces a factor of $\lambda$, it follows that the operators obtained this way are comparable to $S_j, T_j$ and $\dot{S}_j, \dot{T}_j$.
\end{proof}
\section{Nonlinear theory and fixed point arguments}
We can now finally turn our attention to the full equation
\begin{align}\label{integraleq2}
\Phi(\tau)=\Sf(\tau)\uf+\int_0^\tau \Sf(\tau-\sigma)\Nf(\Phi(\sigma)) d \sigma.
\end{align}
Recall that $\Nf$ was of the form
\begin{align*}
\Nf (\uf)(\rho) :=\begin{pmatrix}
0\\
N(\psi_{*_1}+u_1)(\rho)-N(\psi_{*_1})(\rho)-\frac{48}{(\rho^2+2)^2}u_1(\rho)
\end{pmatrix},
\end{align*}
where 
\begin{align*}
N(u)(\rho)=-\frac{3\sin(2\rho u(\rho) )-6\rho u(\rho)}{2\rho^3}
\end{align*}
and
\begin{equation*}
\Psi_*(\rho)=\begin{pmatrix}
&\frac{2 \arctan\left(\frac{\rho}{\sqrt{2}}\right)}{\rho}\\
&\frac{2\sqrt{2}}{\rho^2+2}
\end{pmatrix}.
\end{equation*}
The first result of this section will be a useful bound on $\Nf$, which, combined with the previously established Strichartz estimates, will allow us to also control the nonlinearity.
 \begin{lem}\label{nonlineare}
 The nonlinearity $\Nf$ satisfies the two estimates
\begin{align*}
\|\Nf(\uf)\|_{\mathcal{H}}&\lesssim \|u_1\|_{L^{12}(\B_1^6)}^2+\|u_1\|_{L^9(\B_1^6)}^3+\|u_1\|_{L^8(\B_1^6)}^4+\|u_1'\|_{L^4(\B_1^6)}^2
\end{align*}
and
\begin{align*}
\|\Nf(\uf)-\Nf(\vf)\|_{\mathcal{H}}\lesssim
\|u_1-v_1\|_{L^{12}(\B_1^6)}&\big(\|u_1\|_{L^{12}(\B^6_1)}+\|v_1\|_{L^{12}(\B^6_1)}\\
&+\|u_1\|_{L^{8}(\B^6_1)}^2+\|v_1\|_{L^{8}(\B^6_1)}^2\\
&+\|u_1'\|_{L^4(\B^6_1)}(1+\|u_1\|_{L^6(\B_1^6)}+\|v_1\|_{L^6(\B_1^6)})
\\
&+\|u_1\|_{L^{\frac{36}{5}}(\B^6_1)}^3+\|v_1\|_{L^{\frac{36}{5}}(\B^6_1)}^3
\big)
\\
+\,\|u_1'-v_1'\|_{L^4(\B_1^6)}&\big(\|v_1\|_{L^{12}(\B_1^6)}+\|v_1\|_{L^8(\B_1^6)}^2\big)
\end{align*}
for all $\uf, \vf \in C^\infty\times C^\infty(\overline{\B^6_1})$.
\end{lem}
\begin{proof}
We define the auxiliary function $n:(0,1)\times\R \to \R,$ as
\begin{align*}
n(\rho,x)=&-\frac{3\sin\left(4\arctan\left(\frac{\rho}{\sqrt{2}}\right)+2\rho x\right)-6\rho x}{2\rho^3}+\frac{3\sin\left(4\arctan\left(\frac{\rho}{\sqrt{2}}\right)\right)}{2\rho^3}
-\frac{48}{(\rho^2+2)^2}x.
\end{align*}
By Taylor's Theorem we have that \begin{align}\label{Taylor}
n(\rho,x)=&-\frac{12\sqrt{2}(\rho^2-2)}{(\rho^2+2)^2}x^2+6\int_0^x \cos\left(4\arctan\left(\frac{\rho}{\sqrt{2}}\right)+2\rho t\right)(x-t)^2 dt
\end{align}
and thus
\begin{align*}
|n(\rho,x)|\lesssim |x|^2+|x|^3.
\end{align*}
From the above representation of $n$, it also follows that
\begin{align}\label{esti1}
|\partial_x n(\rho,x)|
\lesssim& |x|+|x|^2
\end{align}
and therefore
\begin{align}
|n(\rho,x)-n(\rho,y)|\lesssim \left(|x|+|y|+|x|^2+|y|^2\right)|x-y|
\end{align}
for all $\rho \in (0,1)$ and all $x,y \in \R$.
From these estimates and Hölder's inequality we deduce that
$$
\|[\Nf(\uf)]_2\|_{L^2(\B^6_1)}\lesssim \| u_1\|_{L^6(\B^6_1)}^3+\|u_1\|_{L^4(\B^6_1)}^2
$$
and
\begin{align*}
\|[\Nf(\uf)-\Nf(\vf)]_2\|_{L^2(\B^6_1)}&\lesssim \|u_1-v_1\|_{L^4(\B^6_1)}\\
&\quad\big(\|u_1\|_{L^4(\B^6_1)}+\|v_1\|_{L^4(\B^6_1)}+\|u_1\|_{L^8(\B^6_1)}^2+\|v_1\|_{L^8(\B^6_1)}^2\big)\\
&\lesssim \|u_1-v_1\|_{L^{12}(\B^6_1)}
\big(\|u_1\|_{L^{12}(\B^6_1)}+\|v_1\|_{L^{12}(\B^6_1)}\\
&\quad+\|u_1\|_{L^{8}(\B^6_1)}^2+\|v_1\|_{L^{8}(\B^6_1)}^2\big).
\end{align*}
Now we turn to the $\dot{H}^1$ seminorm.
A direct calculation using \eqref{Taylor} provides the estimate
\begin{align}\label{esti2}
|\partial_\rho n(\rho,x)|& \lesssim |x|^2+|x|^4
\end{align}
 This, together with \eqref{esti1} yields
\begin{align*}
\| [\N(\uf)]_2\|_{\dot{H}^1(\B^6_1)}&\lesssim \|u_1\|_{L^4(\B^6_1)}^2+\|u_1\|_{L^8(\B^6_1)}^4+\|u_1u_1'\|_{L^2(\B^6_1)}+\|u_1^2u_1'\|_{L^2(\B^6_1)}
\\
&\lesssim \|u_1\|_{L^4(\B^6_1)}^2+\|u_1\|_{L^8(\B^6_1)}^4+\|u_1'\|_{L^4(\B^6_1)}^2.
\end{align*}
Hence only the Lipschitz bound for the $\dot{H}^1$ seminorm remains to be shown.
To this end, we first investigate the expression $\partial_x \partial_\rho n(\rho,x)$. By once more employing the representation \eqref{Taylor}, we derive
\begin{align}\label{estimate n3}
|\partial_x \partial_\rho n(\rho,x)|&\lesssim  |x|+|x|^3 \nonumber\\
|\partial_x^2 n(\rho,x)|&\lesssim 1+|x|.
\end{align}
Moreover,
\begin{align*}
\|[\Nf(\uf)-\Nf(\vf)]_2\|_{\dot{H}^1(\B^6_1)}&\leq \|\partial_1n(|.|,u_1)-\partial_1 n(|.|,v_1)\|_{L^2(\B_1^6)}
\\
&\quad+
\|u_1'\partial_2 n(|.|, u_1)-v_1'\partial_2 n(|.|,v_1)\|_{L^2(\B_1^6)}\\
&=: N_1+N_2.
\end{align*}
Thanks to \eqref{estimate n3} and Hölder's inequality we have that
\begin{align*}
N_1&\lesssim \||u_1-v_1|\left( |u_1|+|u_1|^3 +|v_1|+|v_1|^3\right)\|_{L^2(\B_1^6)}
\\
&\lesssim \|u_1-v_1\|_{L^{12}(\B^6_1)}
\big(\|u_1\|_{L^{12}(\B^6_1)}+\|v_1\|_{L^{12}(\B^6_1)}+\|u_1\|_{L^{\frac{36}{5}}(\B^6_1)}^3+\|v_1\|_{L^{\frac{36}{5}}(\B^6_1)}^3
\big)
\end{align*}
and finally for $N_2$,
\begin{align*}
N_2&\leq \|u_1'\left[\partial_2 n(|.|, u_1)-\partial_2 n(|.|,v_1)\right]\|_{L^2(\B_1^6)}+\|(u_1'-v_1')\partial_2 n(|.|,v_1)\|_{L^2(\B_1^6)}
\\
&\lesssim \|u_1-v_2\|_{L^{12}(\B_1^6)}\|u_1'\|_{L^4(\B_1^6)}\left(1+\|u_1\|_{L^6(\B_1^6)}+\|v_1\|_{L^6(\B_1^6)} \right)
\\
&\quad+\|u_1'-v_1'\|_{L^4(\B_1^6)}\left(\|v_1\|_{L^{12}(\B_1^6)}+\|v_1\|_{L^8(\B_1^6)}^2\right).
\end{align*}
\end{proof}

In accordance with our previously proved Strichartz estimates, we now define the Strichartz space $\mathcal{X}$ as the completion of $C^\infty_c \times C^\infty_c([0,\infty)\times \overline{\B^6_1})$, where $C_c^\infty$ denotes smooth functions with compact support, with respect to the norm
\begin{align*}
\|(\phi_1,\phi_2)\|_{\X}&:=\|\phi_1\|_{L^{2}(\R_+)L^{12}(\B_1^6)}+\|\phi_1\|_{L^{3}(\R_+)L^{9}(\B_1^6)}+\|\phi_1\|_{L^{4}(\R_+)L^{8}(\B_1^6)}
+
\|\phi_1\|_{L^{6}(\R_+)L^{\frac{36}{5}}(\B_1^6)}
\\
&\quad+\|\phi_1\|_{L^{\infty}(\R_+)L^{6}(\B_1^6)}+\|\phi_1\|_{L^{2}(\R_+)\dot{W}^{1,4}(\B_1^6)}+\|\phi_2\|_{L^{2}(\R_+)L^{4}(\B_1^6)}.
\end{align*}
We will also make use of the abbreviation
\[\X_\delta:=\{\Phi \in \X: \|\Phi\|_\X \leq \delta\}.
\]
Next, for any $\uf \in \mathcal{H}$ and $\Phi\in C^\infty_c \times C^\infty_c([0,\infty)\times \overline{\B^6_1}) $ we define
\begin{align*}
\K_{\uf}(\Phi)(\tau):= \Sf(\tau)\uf+\int_0^\tau \Sf(\tau-\sigma)\Nf(\Phi(\sigma))d \sigma-e^{\tau}\Cf(\Phi,\uf),
\end{align*}
where $$
\Cf(\Phi,\uf):=\Pf\left(\uf+\int_0^\infty e^{-\sigma}\Nf(\Phi(\sigma))d \sigma\right).
$$
The useful thing about this integral operator is that it will satisfy the requirements of the Banach fixed point theorem under certain conditions.
\begin{lem}\label{integralbound}
Let $\uf \in \mathcal{H}$ and  $\Phi\in C^\infty_c \times C^\infty_c([0,\infty)\times \overline{\B^6_1})$. Then $\K_\uf(\Phi)\in \mathcal{X}$. Moreover, we have the estimate
\begin{align*}
\|\K_\uf(\Phi)\|_{\X}\lesssim \|\uf\|_{\mathcal{H}}+\|\Phi\|_{\X}^2+\|\Phi\|_{\X}^4.
\end{align*}
\end{lem}
\begin{proof}
Let $\Phi \in  C^\infty_c \times C^\infty_c([0,\infty)\times \overline{\B^6_1})$.
Propositions \ref{strichatzestimate1} and \ref{strichatzestimate2} imply that 
\begin{align*}
\| (\I-\Pf)\K_{\uf}(\Phi)(\tau)\|_{\X}\lesssim \| \uf\|_{\mathcal{H}}+\int_0^\tau\|N(\Phi(\sigma))\|_\mathcal{H} d\sigma
\end{align*}
and by Lemma \ref{nonlineare}, we obtain
\begin{align*}
\int_0^\tau\|N(\Phi(\sigma))\|_\mathcal{H} d\sigma &\lesssim
\int_0^\tau \| \phi_1(\sigma,.)\|_{L^{12}(\B_1^6)}^2+\| \phi_1(\sigma,.)\|_{L^9(\B_1^6)}^3+\| \phi_1(\sigma,.)\|_{L^8(\B_1^6)}^4
\\
&\quad+\| \phi_1(\sigma,.)\|_{\dot{W}^{1,4}(\B_1^6)}^2 d \sigma\\
&\lesssim\|\phi_1\|_{L^2(\R_+)L^{12}(\B_1^6)}^2+\|\phi_1\|_{L^3(\R_+)L^9(\B_1^6)}^3+\| \phi_1\|_{L^4(\R_+)L^8(\B_1^6)}^4
\\
&\quad+\| \phi_1\|_{L^2(\R_+)\dot{W}^{1,4}(\B_1^6)}^2\\
&\lesssim \|\Phi\|_{\X}^2+\|\Phi\|_{\X}^4.
\end{align*}

Next, note that 
$$\Pf \K_{\uf}(\Phi)=-\int_\tau^\infty e^{\tau-\sigma}\Pf\Nf(\Phi(\sigma))d \sigma$$ 
and as $\rg\Pf= \Span\{\gf\}$, there exists a unique $\widetilde{\gf}\in\mathcal{H}$
with 
$$\Pf \ff=\left(\ff,\widetilde{\gf}\right)_{\mathcal{H}}\gf$$
for all $\ff\in \mathcal{H}$.
Thus, $\|[\Pf \ff]_j\|_{\dot W^{k,q}(\B_1^6)}\lesssim
|(\ff,\widetilde{\gf})_{\mathcal H}|\|g_j\|_{\dot W^{k,q}(\B_1^6)}\lesssim \|\ff\|_{\mathcal H}$.
For $q\in [6,12]$ we then obtain
\begin{align*}
&\quad\|[\Pf \K_{\uf}(\Phi)(\tau)]_1\|_{L^q(\B^6_1)}+\|[\Pf \K_{\uf}(\Phi)(\tau)]_1\|_{\dot{W}^{1,4}(\B^6_1)}+\|[\Pf \K_{\uf}(\Phi)(\tau)]_2\|_{L^4(\B^6_1)}\\
&\lesssim \int_\tau^\infty e^{\tau-\sigma}\|\Nf(\Phi(\sigma))\|_{\mathcal{H}} d \sigma\\
&=\int_\R 1_{(-\infty,0]}(\tau-\sigma)e^{\tau-\sigma} 1_{[0,\infty)}(\sigma)\|\Nf(\Phi(\sigma))\|_{\mathcal{H}} d \sigma
\end{align*}
and so an application of Young's inequality yields
\begin{align*}
\|\Pf \K_\uf (\Phi)\|_{L^p(\R_+)L^q(\B_1^6)}
  &\lesssim \|1_{(-\infty,0]}(.)e^{(.)}\|_{L^p(\R)}
    \int_0^\infty\|\Nf(\Phi(\sigma))\|_{\mathcal{H}} d \sigma 
\\
&\lesssim \|\Phi\|_{\X}^2+\|\Phi\|_{\X}^4
\end{align*}
and analogously for the other terms.
\end{proof}
We also derive a corresponding local Lipschitz estimate.
\begin{lem}\label{integrallip}
Let $\uf \in \mathcal{H}$. Then the estimate
\begin{align*}
\|\K_{\uf}(\Psi)-\K_{\uf}(\Phi)\|_\X\lesssim \left(\|\Psi\|_\X+\|\Psi\|_\X^3+\|\Phi\|_\X+\|\Phi\|_\X^3\right) \|\Psi-\Phi\|_\X
\end{align*}
holds true for all $\Psi, \Phi \in  C^\infty_c \times C^\infty_c([0,\infty)\times \overline{\B^6_1})$.
\end{lem}
\begin{proof}
Again, we start by investigating $\left(\I-\Pf \right)\K_{\uf}=:\widetilde{\K}_{\uf}.$ Lemma \ref{nonlineare} implies
\begin{align*}
\|\widetilde{\K}_{\uf}(\Psi)-\widetilde{\K}_{\uf}(\Phi)\|_\X \lesssim& \int_0^\tau \|\Nf_{\uf}(\Psi(\sigma,.))-\Nf(\Phi(\sigma,.))\|_\mathcal{H} d \sigma\\
&\lesssim
\int_0^\tau\|\psi_1(\sigma,.)-\phi_1(\sigma,.)\|_{L^{12}(\B_1^6)}\Big(\|\psi_1(\sigma,.)\|_{L^{12}(\B^6_1)}+\|\phi_1(\sigma,.)\|_{L^{12}(\B^6_1)}
\\
&\quad+\|\psi_1(\sigma,.)\|_{L^{8}(\B^6_1)}^2+\|\phi_1(\sigma,.)\|_{L^{8}(\B^6_1)}^2
\\
&\quad+\|\psi_1(\sigma,.)\|_{\dot{W}^{1,4}(\B^6_1)}(1+\|\psi_1(\sigma,.)\|_{L^6(\B_1^6)}+\|\phi_1(\sigma,.)\|_{L^6(\B_1^6)})
\\
&\quad+\|\psi_1(\sigma,.)\|_{L^{\frac{36}{5}}(\B^6_1)}^3+\|\phi_1(\sigma,.)\|_{L^{\frac{36}{5}}(\B^6_1)}^3
\Big)
\\
&\quad+\|\psi_1(\sigma,.)-\phi_1(\sigma,.)\|_{\dot{W}^{1,4}(\B_1^6)}
\\
&\quad\times\big(\|\phi_1(\sigma,.)\|_{L^{12}(\B_1^6)}+\|\phi_1(\sigma,.)\|_{L^8(\B_1^6)}^2\big) d \sigma
\\
&\lesssim
\|\psi_1-\phi_1\|_{L^2(\R_+)L^{12}(\B_1^6)}\Big(\|\psi_1\|_{L^2(\R_+)L^{12}(\B_1^6)}
+\|\phi_1\|_{L^2(\R_+)L^{12}(\B_1^6)}
\\
&\quad+\|\psi_1\|_{L^4(\R_+)L^{8}(\B_1^6)}^2+
\|\phi_1\|_{L^4(\R_+)L^{8}(\B_1^6)}^2+\|\psi_1\|_{L^2(\R_+)\dot{W}^{1,4}(\B_1^6)}
\\
&\quad\times\big[1+\|\psi_1\|_{L^\infty(\R_+)L^{6}(\B_1^6)}+\|\phi_1\|_{L^\infty(\R_+)L^{6}(\B_1^6)}\big]
\\
&\quad+\|\psi_1\|_{L^6(\R_+)L^{\frac{36}{5}}(\B_1^6)}^3+\|\phi_1\|_{L^6(\R_+)L^{\frac{36}{5}}(\B_1^6)}^3\Big)
\\
&\quad+\|\psi_1-\phi_1\|_{L^2(\R_+)\dot{W}^{1,4}(\B_1^6)}\Big(\|\phi_1\|_{L^2(\R_+)L^{12}(\B_1^6)
}\\
&\quad+\|\phi_1\|_{L^4(\R_+)L^{8}(\B_1^6)}^2\Big)
\\
&\lesssim \left(\|\Psi\|_\X+\|\Psi\|_\X^3+\|\Phi\|_\X+\|\Phi\|_\X^3\right) \|\Psi-\Phi\|_\X.
\end{align*}
Analogously to the proof of Lemma \ref{integralbound}, we have that 
\begin{align*}
&\quad\|[\Pf \K_{\uf}(\Phi)(\tau)-\Pf \K_{\uf}(\Psi)(\tau)]_1\|_{L^q(\B^6_1)}+\|[\Pf \K_{\uf}(\Phi)(\tau)-\Pf \K_{\uf}(\Psi)(\tau)]_1\|_{\dot{W}^{1,4}(\B^6_1)}
\\
&\quad+\|[\Pf \K_{\uf}(\Phi)(\tau)-\Pf \K_{\uf}(\Psi)(\tau)]_2\|_{L^2(\B^6_1)}
\\
&\lesssim \int_\tau^\infty e^{\tau-\sigma}\|\Nf(\Phi(\sigma))\|_{\mathcal{H}} d \sigma
\\
&=\int_\R 1_{(-\infty,0]}(\tau-\sigma)e^{\tau-\sigma} 1_{[0,\infty)}(\sigma)\|\Nf(\Phi(\sigma,.)-\Psi(\sigma,.))\|_{\mathcal{H}} d \sigma
\end{align*}
for all $q\in [6,12]$.
Therefore, the same calculations as before yield \begin{align*}
\|\Pf \K_{\uf}(\Phi)-\Pf \K_{\uf}(\Psi)]\|_\X
\lesssim& \int_0^\infty\|\N(\Phi(\sigma,.))-\Nf(\Psi(\sigma,.))\|_{\mathcal{H}} d \sigma
\\
\lesssim&
\left(\|\Psi\|_\X+\|\Psi\|_\X^3+\|\Phi\|_\X+\|\Phi\|_\X^3\right) \|\Psi-\Phi\|_\X.
\end{align*}
\end{proof}
Lemmas \ref{integralbound} and \ref{integrallip} imply that $\K_\uf$
extends to an operator on $\X$.

\begin{lem}\label{solutionexistence}
  There exist $\delta, c>0$ such that for each $\uf\in \mathcal H$
  with $\|\uf\|_{\mathcal H}\leq \frac{\delta}{c}$ there exists a
  unique $\Phi_{\uf}\in \X_\delta$ that satisfies
$$
\Phi_{\uf}=\K_{\uf}(\Phi_{\uf}).
$$
\end{lem}
\begin{proof}
The claim follows from Lemmas \ref{integralbound}, \ref{integrallip},
and the contraction mapping principle.
\end{proof}
\subsection{Variation of Blowup time}
The next step is to show that we can choose a time $T$ close to $1$ such that the correction term vanishes.
Therefore, we recall that the prescribed initial data 
$$
\Phi(0)=(\phi_1(0,.),\phi_2(0,.)),
$$
are given by
\begin{align*}
	\phi_1(0,\rho)&=\psi_1(0,\rho)-\frac{2\arctan\left(\frac{\rho}{\sqrt{2}}\right)}{\rho}= Tf(T\rho)-\frac{2\arctan\left(\frac{\rho}{\sqrt{2}}\right)}{\rho},
	\\
	\phi_2(0,\rho)&=\psi_2(0,\rho)-\frac{2\sqrt{2}}{2+\rho^2}= T^2 g(T\rho)-\frac{2\sqrt{2}}{2+\rho^2}.
\end{align*}
Furthermore, in similarity coordinates the initial data of the blowup
solution $u^1_*$ are given by
\begin{align*}
\psi^1_{1_*}(0,\rho)=\frac{2\arctan\left(\frac{T \rho}{\sqrt{2}}\right)}{\rho}, \qquad \psi^1_{2_*}(0,\rho)=\frac{T^2 2\sqrt{2}}{2+T^2\rho^2}.
\end{align*}
Motivated by the appearance of the blowup time $T$ in the initial data, we define the operator $\Uf:[1-\delta,1+\delta]\times H^2 \times H^1(\B^6_{1+\delta}) \to \mathcal{H}$ as
\begin{align*}
\Uf(T,\vf)(\rho)=(Tv_1(T\rho),T^2 v_2(T\rho))+(\psi^1_{1_*}(0,\rho),\psi^1_{2_*}(0,\rho))
-\left(\frac{2\arctan(\frac{\rho}{\sqrt2})}{\rho},\frac{2\sqrt{2}}{2+\rho^2}
\right).
\end{align*}
Note that as long as $\delta$ stays small enough, this is a continuous map and
\begin{align*}
\Phi(0)=\Uf\left(T,\left(f-\frac{2\arctan\left(\frac{(.)}{\sqrt{2}}\right)}{(.)},g-\frac{2\sqrt{2}}{2+(.)^2}\right)\right).
\end{align*}
It is also clear that $\Uf(1,\textbf{0})=\textbf{0}$.
Further, for $\delta$ sufficiently small 
the estimate 
\begin{align*}
\|\Psi^1_*-\Psi_*\|_{H^2\times H^1(\B^6_{1+\delta})}\lesssim |1-T|
\end{align*}
holds for all $T\in [1-\delta,1+\delta]$ and so
\begin{align*}
\|\Uf(T,\vf)\|_{\mathcal{H}} \lesssim \|\vf\|_{H^2\times H^1(\B^6_{1+\delta}) }
+|1-T|
\end{align*}
for all $ T \in [1-\delta,1+\delta]$.
\begin{lem}\label{exlem2}
There exist constants $M \geq 1$ and $\delta >0$ such that if $\vf \in H^2\times H^1(\B^6_{1+\delta})$ satisfies $\|\vf\|_{H^2\times H^1(\B^6_{1+\delta})}\leq \frac{\delta}{M}$, then there exists a unique $T^* \in [1-\delta,1+\delta]$ and a $\Phi \in \X_\delta$ with 
$ \Phi=\K_{\Uf(T^*,\vf)}(\Phi)$ and $\Cf(\Phi,\Uf(T^*,\vf))= \normalfont{\textbf{0}}$.
\end{lem}
\begin{proof}
Since 
$$
\partial_T \left(\frac{2\arctan(\frac{T \rho}{\sqrt{2}})}{\rho},\frac{T^2 2\sqrt{2}}{2+T^2\rho^2}\right)\bigg|_{T=1}=2 \sqrt{2}\gf(\rho),
$$
the claim follows by an application of Brower's fixed point theorem, see
the proof of Lemma 6.5 in \cite{Don17} for the details.
\end{proof}
This allows us to give rigorous meaning to the notion of solutions in our topology.
\begin{defi}\label{def:solutionstrichartz}
  Let
  \[ \Gamma^T:=\{(t,r)\in [0,T)\times [0,\infty): r\leq T-t\}. \]
  We say that $u: \Gamma^T\to \R$ is a Strichartz solution of
\[ \left(\partial_t^2-\partial_r^2-\frac{5}{r}\partial_r\right)
  u(t,r) +\frac{3 \sin(2ru(t,r))- 6 r u (t,r)}{2
    r^3}=0 \]
if $\Phi:=\Psi-\Psi_*$, with
\[ \Psi(\tau,\rho):=
  \begin{pmatrix}
    \psi(\tau,\rho) \\ (1+\partial_\tau+\rho\partial_\rho)\psi(\tau,\rho)
  \end{pmatrix},\qquad \psi(\tau,\rho):=Te^{-\tau}u(T-Te^{-\tau}, Te^{-\tau}\rho),
\]
  belongs to $\mathcal{X}$ and satisfies
\begin{align*}
\Phi=\K_{\Phi(0)}(\Phi)
\end{align*}
and $\Cf(\Phi,\Phi(0))=\textup{\textbf{0}}$.
\end{defi}
Next, we show the uniqueness of such solutions.
\begin{lem}
Let $\uf\in \mathcal{H}$. Then the map $\K_{\uf}:\X\to \X$
has at most one fixed point $\Phi \in \X$ that satisfies $\Cf(\Phi,\uf)=\textup{\textbf{0}}$.
\end{lem}

\begin{proof}
  Let $\uf \in \mathcal{H}$ and assume that we are given two fixed
  points $\Psi,\Phi\in \X$ with
  $\Cf(\Psi,\uf)=\Cf(\Phi,\uf)=\textup{\textbf{0}}$. Let
  $\tau_0\geq 0$ be arbitrary and note that for any $\delta>0$ there
  exists an $N_\delta\in\mathbb N$ such that
\begin{equation}\label{smallnes assumption}
\|\Psi\|_{\X(I_n)}+\|\Phi\|_{\X(I_n)}\leq \delta,
\end{equation}
for all $n\in \{0,1,2,\dots,N_\delta-1\}$,
 where 
\begin{align*}
\|\Psi\|_{\X(I_n)}:=
  &\|\psi_1\|_{L^2(I_n)L^{12}(\B^6_1)}+\|\psi_1\|_{L^4(I_n)L^8(\B^6_1)}+\|\psi_1\|_{{L^6(I_n)L^{\frac{36}{5}}(\B^6_1)}}+\|\psi_1\|_{{L^2(I_n)\dot{W}^{1,4}(\B^6_1)}}
  \\
  &+\|\psi_2\|_{L^2(I_n)L^4(B_1^6)}
\end{align*}
and $I_n:=[n\frac{\tau_0}{N_\delta}, (n+1)\frac{\tau_0}{N_\delta}]$.
Such an $N_\delta$ can always be found because the map $\tau \mapsto
\|f\|_{L^p(0,\tau)}: [0,\infty)\to \R$ is continuous for every $f\in
L^p(0,\infty)$ and $p\in [1,\infty)$. 
From the calculations done in the proof of Lemma \ref{integrallip} we then see that
\begin{align*}
&\quad\|\psi_1(\tau)-\phi_1(\tau)\|_{L^{12}(\B^6_1)}+\|\psi_1(\tau)-\phi_1(\tau)\|_{\dot{W}^{1,4}(\B^6_1)}+\|\psi_2(\tau)-\phi_2(\tau)\|_{L^4(\B^6_1)}
\\
&\lesssim \int_0^\tau e^{\tau-\sigma} \|\Nf(\Phi(\sigma))-\Nf(\Psi(\sigma))\|_{\mathcal{H}}d \sigma
\end{align*}
and so
\begin{align*}
&\quad\|\psi_1-\phi_1\|_{L^2(I_0)L^{12}(\B^6_1)}+\|\psi_1-\phi_1\|_{L^2(I_0)\dot{W}^{1,4}(\B^6_1)}+\|\psi_2-\phi_2\|_{L^2(I_0)L^4(\B^6_1)}
\\
&\lesssim \left\|\int_0^\tau e^{\tau-\sigma} \|\Nf(\Phi(\sigma))-\Nf(\Psi(\sigma))\|_{\mathcal{H}}d \sigma\right\|_{L^2_\tau(I_0)}
\\
&\lesssim  e^{\tau_0} \int_0^{\tau_0/N_\delta}\|\Nf(\Phi(\sigma))-\Nf(\Psi(\sigma))\|_{\mathcal{H}}d \sigma
\\
&\lesssim  e^{\tau_0} \Big[
\|\psi_1-\phi_1\|_{L^2(I_0)L^{12}(\B_1^6)}\big(\|\psi_1\|_{L^2(I_0)L^{12}(\B^6_1)}+\|\phi_1\|_{L^2(I_0)L^{12}(\B^6_1)}
\\
&\quad+\|\psi_1\|_{L^4(I_0)L^{8}(\B^6_1)}^2+\|\phi_1\|_{L^4(I_0)L^{8}(\B^6_1)}^2+\|\psi_1\|_{L^6(I_0)L^{\frac{36}{5}}(\B^6_1)}^3+\|\phi_1\|_{L^6(I_0)L^{\frac{36}{5}}(\B^6_1)}^3
\\
&\quad+\|\psi_1\|_{L^2(I_0)\dot W^{1,4}(\B^6_1)}(1+\|\psi_1\|_{L^\infty(\R_+)L^6(\B_1^6)}+\|\phi_1\|_{L^\infty(\R_+)L^6(\B_1^6)})
\big)
\\
&\quad+\|\psi_1-\phi_1\|_{L^2(I_0)\dot W^{1,4}(\B_1^6)}\big(\|\phi_1\|_{L^2(I_0)L^{12}(\B_1^6)}+\|\phi_1\|_{L^4(I_0)L^8(\B_1^6)}^2\big)\Big].
\end{align*}
By choosing delta small enough, we obtain
\begin{align*}
&\quad\|\psi_1-\phi_1\|_{L^2(I_0)L^{12}(\B^6_1)}+\|\psi_1-\phi_1\|_{L^2(I_0)\dot{W}^{1,4}(\B^6_1)}+\|\psi_2-\phi_2\|_{L^2(I_0)L^{4}(\B^6_1)}
\\
&\leq \frac{1}{2}\left(\|\psi_1-\phi_1\|_{L^2(I_0)L^{12}(\B^6_1)}+\|\psi_1-\phi_1\|_{L^2(I_0)\dot{W}^{1,4}(\B^6_1)}+\|\phi_2-\psi_2\|_{L^2(I_0)L^{4}(\B^6_1)}\right).
\end{align*}
Hence, it follows that $\Psi$ and $\Phi$ agree as elements of $\X(I_0)$.
Proceeding inductively, we conclude that $\Psi= \Phi$ in $\X([0,\tau_0])$
and as $\tau_0$ was chosen arbitrarily the claim follows.
\end{proof}
\subsection{Persistence of regularity and proof of Theorem \ref{stability}}
As we have shown the existence and uniqueness of Strichartz solutions in the sense of Definition \ref{def:solutionstrichartz} for data close enough to the blowup solution, we will now show that such a solution is in fact smooth if we prescribe smooth initial data.
The first step in achieving this is the following result for the semigroup $\Sf$.
\begin{lem}\label{reg0}
Let $2\leq k \in \mathbb{N}$. Then there exists a constant $a_k>0$ such that the restriction of $\Sf$ to $H^k\times H^{k-1}(\B^6_1)$ satisfies
\begin{align*}
\| \Sf(\tau) \uf\|_{H^k\times H^{k-1}(\B^6_1)}\lesssim_k e^{a_k\tau}\|\uf\|_{H^k\times H^{k-1}(\B^6_1)}
\end{align*}
for all $\tau \geq 0$ and all $\uf \in H^k\times H^{k-1}(\B^6_1)$.
\end{lem}
\begin{proof}
As $\Lf'$ is a bounded operator on $H^k\times H^{k-1}(\B^6_1)$ for any $2\leq k \in \mathbb N $, the claim follows if we show that $\Sf_0$ satisfies
\begin{align*}
\| \Sf_0(\tau) \uf\|_{H^k\times H^{k-1}(\B^6_1)}\lesssim_k\|\uf\|_{H^k\times H^{k-1}(\B^6_1)}
\end{align*}
for any $\uf \in H^k \times H^{k-1}(\B^6_1)$.
To show this, we argue similarly as in the proof of Lemma \ref{Strichart}. Thus, assume $T\in \left[\frac{1}{2},\frac{3}{2}\right]$ and let $k\geq 2$ be a fixed natural number. Moreover, by a density argument it suffices to show the claim for $\uf \in C^k\times C^{k-1}(\overline{\B^6_1}).$ Next, let $\Af_T$ be the scaling operator $\Af_T:H^k\times H^{k-1}(\B^6_T)\to H^k\times H^{k-1}(\B^6_1)$,
$$
\Af_T\ff= (Tf_1(T.),T^2f_2(T.))
$$
and note that both $\Af_T$ and its inverse are uniformly bounded with respect to $T$.
 Now, let $u$ be the solution of the Cauchy problem
\begin{align*}
\begin{cases}
\left(\partial_t^2- \partial_r^2-\frac{5}{r}\partial_r \right)u(t,r)=0\\
u[0]=\Af^{-1}_T\uf
\end{cases}
\end{align*}
 on $\Gamma^T$. Then, by higher energy conservation we know that
 \begin{align*}
 \|u(t,.)\|_{H^k(\B^6_{T-t})}+ \|\partial_t u(t,.)\|_{H^{k-1}(\B^6_{T-t})}\lesssim \|u[0]\|_{H^k\times H^{k-1}(\B^6_T)}.
 \end{align*}
Furthermore, $[\Sf_0(\tau) \uf]_1(\rho)= Te^{-\tau}u(T-Te^{-\tau},Te^{-\tau}\rho)$ which we use to obtain
 \begin{align*}
 \|[\Sf_0(\tau) \uf]_1\|_{H^k(\B^6_1)} &\lesssim \|u(T-Te^{-\tau},Te^{-\tau}.)\|_{H^k(\B^6_1)}
 \lesssim \|u(T-Te^{-\tau},.)\|_{H^k(\B^6_{T-t})}
 \\
 &\lesssim
 \|u[0]\|_{H^k\times H^{k-1}(\B^6_T)}
 \lesssim  \|\uf\|_{H^k\times H^{k-1}(\B^6_1)}.
 \end{align*}

In the same manner one obtains the estimate 
\begin{align*}
 \|[\Sf_0(\tau)\uf]_2\|_{H^{k-1}(\B^6_1)}  \lesssim  \|\uf\|_{H^k\times H^{k-1}(\B^6_1)}
\end{align*}
 and since $k\geq 2$ was arbitrary the claim follows.
\end{proof}

\begin{lem}\label{nonlinearH3}
Let $3 \leq k\in \mathbb N.$ Then the nonlinearity $\Nf$ maps $H^k\times H^{k-1}(\B^6_1)$ to itself.
\end{lem}
\begin{proof}
For $4\leq k\in \mathbb N$ the claim follows immediately from the Banach algebra property of $H^k(\B^6_1)$.
For $k=3$
 it suffices to show that 
\begin{align*}
\widehat{\Nf}(\uf)(\rho):=\begin{pmatrix}
0
\\
N(u_1)(\rho)
\end{pmatrix}
=\begin{pmatrix}
0
\\
-\frac{3\sin(2\rho u_1(\rho))-6\rho u_1(\rho)}{2\rho^3}
\end{pmatrix}
\end{align*} 
is in $H^3\times H^2(\B^6_1)$ whenever $\uf\in H^3\times H^2(\B^6_1)$. This follows from the equation
\begin{align*}
\Nf(\uf)=\widehat{\Nf}(\Psi_*+\uf)+\widehat{\Nf}(\Psi_*)+\Lf'\uf,
\end{align*}
and the fact that $\Psi_* \in C^\infty \times C^\infty(\overline{\B^6_1})$.
Making use of Taylor's theorem, we see that
\begin{align}
n(\rho,x):=-\frac{3\sin(2\rho x)-6\rho x}{2\rho^3}=6\int_0^{x}
  \cos(2t\rho)( x-t)^2 dt=6x^3\int_0^1 \cos(2t\rho x)(1-t)^2 dt
\end{align}
and thus, we need to show that $\rho\mapsto n(\rho,u(\rho))$
is an element of $H^2(\B^6_1)$ for any $u\in H^3(\B^6_1)$.
For the $H^1$ norm we see that
\begin{align*}
\|n(|.|,u)\|_{H^1(\B^6_1)}&\lesssim \|u^3\|_{L^2(\B^6_1)}+\|u^4\|_{L^2(\B^6_1)}+\| u^2 u'\|_{L^2(\B^6_1)} \\
&\lesssim 1+\|u\|_{H^3(\B^6_1)}^4+\| u^2\|_{L^4(\B^6_1)}\| u'\|_{L^4(\B^6_1)}
\\
&\lesssim 1+\|u\|_{H^3(\B^6_1)}^4,
\end{align*}
while
\begin{align*}
\| n(|.|,u)\|_{\dot{H}^2(\B^6_1)}&=\bigg\|\int_0^{u(\rho)} 2 u''(\rho)\cos(2t\rho)( u(\rho)-t)+2 u'(\rho)^2\cos(2t\rho)
\\
&\quad-8u'(\rho)\sin(2t\rho)2t( u(\rho)-t)-4\cos(2t\rho)t^2( u(\rho)-t)^2 dt
\\
&\quad+\frac{5}{\rho}\left(\int_0^{u(\rho)} 2 u'(\rho)\cos(2t\rho)( u(\rho)-t) -2\sin(2t\rho)t( u(\rho)-t)^2 dt\right)\bigg\|_{L^2_\rho(\B^6_1)}.
\end{align*}
Note, that
\begin{align*}
&\quad\left\|\int_0^{u(\rho)} 2 u''(\rho)\cos(2t\rho)( u(\rho)-t)dt+\frac{5}{\rho}\int_0^{u(\rho)} 2 u'(\rho)\cos(2t\rho)( u(\rho)-t)dt\right\|_{L^2_\rho(\B^6_1)}
\\
&=\left\|\left(u''(\rho)+\frac{5}{\rho}u'(\rho)\right)\int_0^{u(\rho)} 2 \cos(2t\rho)( u(\rho)-t)dt\right\|_{L^2_\rho(\B^6_1)}
\\
&\leq \|u\|_{\dot{W}^{2,3}(\B^6_1)}\left\|\int_0^{u(\rho)} 2 \cos(2t\rho)( u(\rho)-t)dt\right\|_{L^6_\rho(\B^6_1)}\lesssim \|u\|_{H^3(\B_1^6)}^3
\end{align*}
due to Hölder's inequality and Sobolev embedding.
Further, by employing Lemma \ref{helplemoverrho} we derive the estimate
\begin{align*}
\left\|\frac{5}{\rho}\int_0^{u(\rho)}\sin(2t\rho)2t( u(\rho)-t)^2 dt
\right\|_{L^2_\rho(\B^6_1)}&\lesssim \||.|^{-1}u^4\|_{L^2(\B^6_1)}\lesssim \|u^4\|_{H^1(\B^6_1)}
\\
&\lesssim \|u\|_{W^{1,3}(\B^6_1)}\|u^3\|_{L^6(\B^6_1)}\lesssim \|u\|_{H^3(\B^6_1)}^4
\end{align*}
and as the other terms can be bounded likewise we conclude this proof.
\end{proof}

\begin{lem}
Let $3\leq k\in \mathbb N$. Then $\Nf$ satisfies the Lipschitz bound
\begin{align*}
\|\Nf(\uf)-\Nf(\vf)\|_{H^k\times H^{k-1}(\B^6_1)}\lesssim \left(1+\|\uf\|_{H^k\times H^{k-1}(\B^6_1)}^{k+1}+\|\vf\|_{H^k\times H^{k-1}(\B^6_1)}^{k+1}\right)\|\uf-\vf\|_{H^k\times H^{k-1}(\B^6_1)}
\end{align*}
for all $\uf,\vf \in H^k\times H^{k-1}(\B^6_1)$.
\end{lem}
\begin{proof}
For $k\geq 4$, the claim again follows immediately from the Banach
algebra property of $H^k(\B^6_1)$
and for $k=3$ it can be proved like Lemma \ref{nonlinearH3}.
\end{proof}
With this we can now define another notion of a solution of Eq.~\eqref{startingeq2}.
\begin{defi}
We say that $u: \Gamma^T\to \R$ is a local $H^k$ solution of
\[ \left(\partial_t^2-\partial_r^2-\frac{5}{r}\partial_r\right)
  u(t,r) +\frac{3 \sin(2ru(t,r))- 6 r u (t,r)}{2
    r^3}=0 \]
if $\Phi:=\Psi-\Psi_*$, with
\[ \Psi(\tau,\rho):=
  \begin{pmatrix}
    \psi(\tau,\rho) \\ (1+\partial_\tau+\rho\partial_\rho)\psi(\tau,\rho)
  \end{pmatrix},\qquad \psi(\tau,\rho):=Te^{-\tau}u(T-Te^{-\tau}, Te^{-\tau}\rho),
\]
  belongs to $C([0,\tau_0), H^k\times H^{k-1}(\B^6_1))$ and satisfies
\begin{align*}
\Phi(\tau)=\Sf(\tau)\Phi(0)+\int_0^\tau \Sf(\tau-\sigma)\Nf(\Phi(\sigma)) d \sigma
\end{align*}
for all $0 \leq \tau < \tau_0$.
Finally, if $\tau_0$ can be chosen $\infty$, we refer to $u$ as a global $H^k$ solution.
\end{defi}

Using Lemmas \ref{reg0} and \ref{nonlinearH3} we can now mimic the proof of Lemma \ref{lem:locexk} to obtain a local existence result in $H^k\times H^{k-1}$ for $k\geq 3$.
\begin{lem}\label{lem:localexistHk}
    Let $M>0$ and $3\leq k\in \mathbb{N}$ fixed. Then there exists
    a $\tau_0>0$ such that the following holds. Let $(f,g)\in H^{k}\times
    H^{k-1}(\B^6_T)$ with $\|(f,g)\|_{H^{k}\times H^{k-1}(\B^6_T)}\leq M$. Then there exists a unique local $H^k$ solution $u$ of ~\eqref{startingeq2} in the
lightcone
$\Gamma^T$ with $(u(0,.),\partial_0 u(0,.))=(f,g)$.
Moreover, if the associated $\Phi$ is inextendible beyond time $\tau_0$, we must have
    \[ \sup_{\tau\in [0,\tau_0)}\|\Phi(\tau)\|_{H^k\times H^{k-1} (\B^6_1)}=\infty. \]
  \end{lem}
  Using this local existence result we can now show that a local $H^3$
  solution in the lightcone $\Gamma^T$ will be a global one, provided
  that it is also a Strichartz solution.
\begin{lem}\label{reg1}
Suppose we are given initial data $(f,g)\in H^3 \times H^2(\B^6_T)$, such that there exists an associated Strichartz solution $u$ with $(u(0,.),\partial_0 u(0,.))=(f,g)$. Then $u$ is also a global $H^3$ solution.
\end{lem}
\begin{proof}
From Lemma \ref{lem:localexistHk} we already know that there exists a unique local $H^3$ solution $u$. Assume now that the corresponding $\Phi$ is inextendible beyond $\tau_0$. Moreover, for any bounded interval $I\subset [0,\infty)$ that contains 0 we let
$\mathcal{X}(I)$ be the completion of $C^\infty \times C^\infty(I\times \overline{\B^6_1})$ with respect to the norm
\begin{align*}
\|(\phi_1,\phi_2)\|_{\X(I)}&:=\|\phi_1\|_{L^{2}(I)L^{12}(\B_1^6)}+\|\phi_1\|_{L^{3}(I)L^{9}(\B_1^6)}+\|\phi_1\|_{L^{4}(I)L^{8}(\B_1^6)}
+
\|\phi_1\|_{L^{6}(I)L^{\frac{36}{5}}(\B_1^6)}
\\
&\quad+\|\phi_1\|_{L^{\infty}(I)L^{6}(\B_1^6)}+\|\phi_1\|_{L^{2}(I)\dot{W}^{1,4}(\B_1^6)}+\|\phi_2\|_{L^{2}(I)L^{4}(\B_1^6)}.
\end{align*}
Then, by Sobolev embedding $\Phi$ belongs to the space $\X([0,\tau])$ for every $0<\tau<\tau_0$ and therefore by our uniqueness theory has to coincide with the Strichartz solution on the time interval $[0,\tau_0)$.
Next, from Lemma \ref{reg0} we know that the estimate
\begin{align*}
\|\Sf(\tau)\uf\|_{H^3\times H^{2}(\B^6_1)}\lesssim e^{a_3\tau}\|\uf\|_{H^3\times H^{2}(\B^6_1)}
\end{align*}
holds for any $\uf\in H^3\times H^{2}(\B^6_1)$ and $\tau\geq0$. With this at hand, the main task now is to establish control over the nonlinearity. To do so we again investigate the expression
\begin{align*}
n(|.|,u)=\int_0^{u} \cos(2t|.|)( u-t)^2 dt
\end{align*}
for $u \in H^3(\B^6_1)$.
Similarly as before, we obtain
\begin{align}\label{n:H2norm}
\left\|n(|.|,u)\right\|_{H^1(\B^6_1)}&\leq \|u\|_{L^6(\B^6_1)}^3+\|u\|_{L^8(\B^6_1)}^4+\|u\|_{\dot{W}^{1,6}(\B^6_1)}\|u\|_{L^6(\B^6_1)}^2\nonumber
\\
&\lesssim \|u\|_{H^3(\B^6_1)}\left( \|u\|_{L^{12}(\B^6_1)}^2+\|u\|_{L^9(\B^6_1)}^3\right).
\end{align}
Recall, that
\begin{align*}
\| n(|.|,u)\|_{\dot{H}^2(\B^6_1)}&= \bigg\|\int_0^{u(\rho)} 2 u''(\rho)\cos(2t\rho)( u(\rho)-t)+2 u'(\rho)^2\cos(2t\rho)
\\
&\quad-8u'(\rho)\sin(2t\rho)2t( u(\rho)-t)-\cos(2t\rho)2t^2( u(\rho)-t)^2 dt
\\
&\quad+\frac{5}{\rho}\left(\int_0^{u(\rho)} 2 u'(\rho)\cos(2t\rho)( u(\rho)-t) -\sin(2t\rho)2t( u(\rho)-t)^2 dt\right)\bigg\|_{L^2_\rho(\B^6_1)}.
\end{align*}
Hölder's inequality implies that
\begin{align*}
&\quad\left\|\int_0^{u(\rho)} 2 u''(\rho)\cos(2t\rho)( u(\rho)-t)dt+\frac{5}{\rho}\int_0^{u(\rho)} 2 u'(\rho)\cos(2t\rho)( u(\rho)-t)dt\right\|_{L^2_\rho(\B^6_1)}
\\
&\leq \| u\|_{\dot{W}^{2,3}(\B^6_1)}\left\|\int_0^{u(\rho)} 2 \cos(2t\rho)( u(\rho)-t)dt\right\|_{L^6_\rho(\B^6_1)}\lesssim \|u\|_{H^3(\B^6_1)}\|u\|_{L^{12}(\B^6_1)}^2.
\end{align*}
Next, from the calculations in Lemma \ref{nonlinearH3} we also infer that 
\begin{align*}
\left\|\frac{5}{\rho}\int_0^{u(\rho)}\sin(2t\rho)2t( u(\rho)-t)^2 dt
\right\|&\lesssim \|u^4\|_{H^1(\B^6_1)}\lesssim \|u\|_{{W}^{1,6}(\B^6_1)}\|u^3\|_{L^3(\B^6_1)}
\\
&\lesssim  \|u\|_{H^3(\B^6_1)}\|u\|_{L^9(\B^6_1)}^3 
\end{align*}
while doing an integration by parts yields
\begin{align*}
&\quad\left\|\int_0^{u(\rho)}\cos(2t\rho)2t^2( u(\rho)-t)^2 dt
\right\|_{L^2_\rho(\B^6_1)}
\\
&\leq \left\|\int_0^{u(\rho)}\rho^{-1}\sin(2t\rho)\left[t( u(\rho)-t)^2+ t^2( u(\rho)-t) \right]dt
\right\|_{L^2_\rho(\B^6_1)}
\\
&\lesssim \||.|^{-1} u^4\|_{L^2(\B^6_1)}\lesssim \|u\|_{H^3(\B^6_1)}\|u\|_{L^9(\B^6_1)}^3 .
\end{align*}
As the remaining summands can be bounded likewise we derive that
\begin{align*}
\|n(|.|,u)\|_{H^2(\B^6_1)}\lesssim \|u\|_{H^3(\B^6_1)}\left(\|u\|_{L^9(\B^6_1)}^3 +\|u\|_{L^{12}(\B^6_1)}^2+\|u\|_{\dot{W}^{1,4}(\B^6_1)}^2\right)
\end{align*}
and therefore
\begin{align*}
\|\widehat{\Nf}(\uf)\|_{H^3\times H^2(\B^6_1)}\lesssim \|u_1\|_{H^3(\B^6_1)}\left(\|u_1\|_{L^9(\B^6_1)}^3 +\|u_1\|_{L^{12}(\B^6_1)}^2+\|u_1\|_{\dot{W}^{1,4}(\B^6_1)}^2\right)
\end{align*}
for any $\uf \in H^3 \times H^2(\B^6_1)$ and with $\widehat{\Nf}$ as defined in the proof of Lemma \ref{nonlinearH3}.
Moreover, we know that
\begin{align*}
\Nf(\uf)=\widehat{\Nf}(\Psi_*+\uf)+\widehat{\Nf}(\Psi_*)+\Lf'\uf
\end{align*}
and so we see that
\begin{align*}
\|\Nf(\uf)\|_{H^3\times H^2(\B^6_1)}\lesssim \left(1+\|u_1\|_{H^3(\B^6_1)}\right)\left(1+\|u_1\|_{L^9(\B^6_1)}^3 +\|u_1\|_{L^{12}(\B^6_1)}^2+\|u_1\|_{\dot{W}^{1,4}(\B^6_1)}^2\right)
\end{align*}
for any $\uf \in H^3 \times H^2(\B^6_1)$.
With this at hand we readily estimate
\begin{align*}
\|\Phi(\tau)\|_{H^3\times H^{2}(\B^6_1)}&\lesssim e^{a_3\tau}\|\Phi(0) \|_{H^3\times H^{2}(\B^6_1)} +e^{a_3\tau}\int_0^\tau \|\Nf(\Phi(\sigma))\|_{H^3\times H^{2}(\B^6_1)} d \sigma
\\
&\lesssim
e^{a_3\tau}\|\Phi(0) \|_{H^3\times H^{2}(\B^6_1)}+e^{a_3\tau}\int_0^\tau \left(1+\|\Phi(\sigma)\|_{H^3\times H^2(\B^6_1)}\right)
\\
&\quad\times\left(1+\|\phi_1(\sigma)\|_{L^9(\B^6_1)}^3+\|\phi_1(\sigma)\|_{L^{12}(\B^6_1)}^2+\|\phi_1(\sigma)\|_{\dot{W}^{1,4}(\B^6_1)}^2\right)d \sigma
\end{align*}
for any $0\leq\tau < \tau_0$. Hence, Gronwall's inequality implies
\begin{align*}
\|\Phi(\tau)\|_{H^3\times H^{2}(\B^6_1)}&\lesssim  e^{a_3\tau_0}(\tau_0+1+\|\Phi\|_{\mathcal{X}([0,\tau_0))}^3 +\|\Phi(0)\|_{H^3\times H^{2}(\B^6_1)} )
\\
&\quad\times\exp \bigg(e^{a_3\tau_0}\int_0^\tau \|\phi_1(\sigma)\|_{L^9(\B^6_1)}^3
+\|\phi_1(\sigma)\|_{L^{12}(\B^6_1)}^2+\|\phi_1(\sigma)\|_{\dot{W}^{1,4}(\B^6_1)}^2 d \sigma \bigg)
\\
&\lesssim
e^{a_3\tau_0}(\tau_0+1+\|\Phi\|_{\mathcal{X}([0,\tau_0))}^3+\|\Phi(0) \|_{H^3\times H^{2}(\B^6_1)}) 
\\
&\quad\times\exp\left(e^{a_3\tau_0}(1+ \|\Phi\|_{\mathcal{X}([0,\tau_0))}^3)\right),
\end{align*}
for all $0\leq \tau<\tau_0$.
So, we see that $\|\Phi(\tau)\|_{H^3\times H^2(\B^6_1)}$ can be uniformly bounded on $[0,\tau_0)$ which contradicts the blowup alternative.
\end{proof}
The same strategy can be applied to further improve the regularity to $H^4\times H^{3}(\B^6_1)$.
\begin{lem}\label{reg2}
Suppose we are given initial data $(f,g)\in H^4 \times H^3(\B^6_T)$, such that there exists an associated  Strichartz solution $u$ with $(u(0,.),\partial_0 u(0,.))=(f,g)$. Then $u$ is also a global $H^4$ solution.
\end{lem}

\begin{lem}\label{lem:Hkexistence}
Suppose we are given initial data $(f,g)\in C^\infty \times C^\infty(\overline{\B^6_T})$, such that there exists an associated Strichartz solution $u$ with $(u(0,.),\partial_0 u(0,.))=(f,g)$. Then $u$ is a global $H^k$ solution for every $3 \leq k \in \mathbb{N}$.
\end{lem}
\begin{proof}
From Lemma \ref{lem:localexistHk} we already know that smooth initial
data leads to a local $H^k $ solution for any $3\leq k \in \mathbb{N}$ which additionally is a global $H^4$  solution. Fix now a natural number $k\geq 5$. As before, we consider the associated $\Phi \in C([0,\tau_0),H^k\times H^{k-1}(\B^6_1))$ and assume that $\Phi \in  C([0,\infty),H^{k-1}\times H^{k-2}(\B^6_1))$. Suppose now that $\Phi$ is inextendible beyond $\tau_0$, for some $\tau_0>0$ as an element of  $C([0,\tau_0),H^{k}\times H^{k-1}(\B^6_1))$.
As $k\geq 5$ we can use the Banach algebra property of $H^{k-1}(\B^6_1)$ to compute that
\begin{align*}
\|\Nf(\uf)\|_{H^k\times H^{k-1}(\B^6_1)}\lesssim 1+\|\uf\|_{H^{k-1}\times H^{k-2}(\B^6_1)}^{k+2}
\end{align*}
for all $\uf\in H^k\times H^{k-1}(\B^6_1)$.
We then use Lemma \ref{reg0} to calculate 
\begin{align*}
\|\Phi(\tau)\|_{H^k\times H^{k-1}(\B^6_1)}&\lesssim e^{a_k\tau}\left(\|\Phi(0) \|_{H^k\times H^{k-1}(\B^6_1)}+\int_0^\tau \|\Nf(\Phi(\sigma))\|_{H^k\times H^{k-1}(\B^6_1)} d \sigma\right)
\\
&\lesssim e^{a_k\tau} \bigg(\|\Phi(0) \|_{H^k\times H^{k-1}(\B^6_1)}
+ \int_0^\tau 1+\|\Phi(\sigma)\|_{H^{k-1}\times H^{k-2}(\B^6_1)}^{k+2} d \sigma\bigg).
\end{align*}
Hence, the $H^k\times H^{k-1}$ norm of $\Phi(\tau)$ stays uniformly bounded on $[0,\tau_0)$, which contradicts the blowup alternative. Proceeding inductively, we obtain $\Phi\in C([0,\infty),H^k\times H^{k-1}(\B^6_1))$
for every $3\leq k \in \mathbb N$.
\end{proof}
\begin{prop}
Suppose we are given initial data $(f,g)\in C^\infty \times C^\infty(\overline{\B^6_T})$, such that there exists an associated Strichartz solution $u$ with $(u(0,.),\partial_0 u(0,.))=(f,g)$. Then $u$ is a smooth solution in the lightcone $\Gamma^T$.
\end{prop}
\begin{proof}
Lemma \ref{lem:Hkexistence} implies that the associated $\Phi$ satisfies $\Phi\in C([0,\infty),C^\infty\times C^\infty(\overline{\B^6_1}))$.
Furthermore, as $\uf$ and $\Phi(\tau)(.)$ are smooth for any $\tau\geq 0$ we know that
\begin{align*}
\partial_\tau \Phi(\tau)= \Sf(\tau) \Lf \uf +\Nf(\Phi(\tau))+\int_0^\tau \Sf(\tau-\sigma)\Lf\N(\Phi(\sigma))d\sigma
\end{align*}
and hence we see that 
\begin{align*}
\partial_\tau \Phi\in C([0,\infty),C^\infty\times C^\infty(\overline{\B^6_1})).
\end{align*}
However, as the right hand side contains no terms involving $\partial_\tau$ we can conclude inductively that $\Phi$ satisfies 
$\Phi\in C^\infty([0,\infty),C^\infty\times C^\infty(\overline{\B^6_1}))$.
Hence, we infer that $\partial_\rho^\ell\partial_\tau^k
\Phi(\tau)(\rho)$ can naturally be identified with an element of
$C([0,\infty)\times \overline{\B^6_1},\R^2)$ for any $k,\ell \in
\mathbb{N}_0$. Consequently, an improved  version of Schwarz's Theorem
(see \cite{Rud76}, p.~235, Theorem 9.41) ensures that $\Phi$ and subsequently $u$ is indeed smooth.
\end{proof}
Now, we are finally able to prove Theorem \ref{stability}.
\begin{proof}[Proof of Theorem \ref{stability}]
Let $\delta >0$ be small enough and choose $M\geq0$ sufficiently large.
Further, let $\vf=(f,g)-u^1_*[0] \in C^\infty \times C^\infty(\overline{\B^6_{1+\delta}})$  be such that
\begin{align*}
\|(f,g)-u^1_*[0]\|_{H^2\times H^1(\B^6_{1+\delta})}\leq \frac{\delta}{M}.
\end{align*}
Then, by our previous results, there exists a unique smooth solution with that initial data. Denote this solution by $u$, its associated function in similarity coordinates by $\Phi \in \X_\delta$, and let $T$ be the blowup time. We calculate
\begin{align*}
\delta^2&\geq\| \phi_1\|_{L^2(\R_+)L^{12}(\B_1^6)}^2= \int_0^\infty \left\|\psi(\tau,.)-2|.|^{-1}\arctan\left(2^{-\frac{1}{2}}|.|\right)\right\|^2_{L^{12}(\B_1^6)}d \tau\\
&=\int_0^T \left\|\psi(-\log(T-t)+\log T,.)-2|.|^{-1}\arctan\left(2^{-\frac{1}{2}}|.|\right)\right\|^2_{L^{12}(\B_1^6)}\frac{dt}{T-t}\\
&=\int_0^T \bigg\|(T-t)^{-1}\psi(-\log(T-t)+\log T,\frac{.}{T-t})
\\
&\quad-2|.|^{-1}\arctan\left(2^{-\frac{1}{2}}\frac{|.|}{T-t}\right)\bigg\|^2_{L^{12}(\B_{T-t}^6)} dt
\\
&=\int_0^T \left\|u(t,.)-u^T_*(t,r)\right\|^2_{L^{12}(\B_{T-t}^6)}dt
\end{align*}
and similarly
\begin{align*}
\delta^2&\geq\| \phi_1\|_{L^2(\R_+)\dot{W}^{1,4}(\B_1^6)}^2= \int_0^\infty \left\|\psi(\tau,.)-2|.|^{-1}\arctan\left(2^{-\frac{1}{2}}|.|\right)\right\|^2_{\dot{W}^{1,4}(\B_1^6)}d \tau\\
&=\int_0^T \left\|\psi(-\log(T-t)+\log T,.)-2|.|^{-1}\arctan\left(2^{-\frac{1}{2}}|.|\right)\right\|^2_{\dot{W}^{1,4}(\B_1^6)}\frac{dt}{T-t}\\
&=\int_0^T \bigg\|(T-t)^{-1}\psi(-\log(T-t)+\log T,\frac{.}{T-t})
\\
&\quad-2|.|^{-1}\arctan\left(2^{-\frac{1}{2}}\frac{|.|}{T-t}\right)\bigg\|^2_{\dot{W}^{1,4}(\B_{T-t}^6)}dt
\\
&=\int_0^T \left\|u(t,.)-u^T_*(t,r)\right\|^2_{\dot{W}^{1,4}(\B_{T-t}^6)}dt.
\end{align*}
\end{proof}
\section{The main theorem}
Establishing Theorem \ref{maintheorem} will now be a comparatively simple task.
We first state another useful version of Hardy's inequality.
\begin{lem}\label{Hardy}
 Let  $d=4,6$ and $R_0>0$. Then we have the bounds
  \begin{align*}
    \int_0^R |f(r)|^2 r^{d-3} dr&\lesssim \|f\|_{H^1(\B_R^d)}^2 \\
    \int_0^R |f'(r)|^2 r^{d-3} dr &\lesssim \|f\|_{H^2(\B_R^d)}^2
  \end{align*}
  for all $R\geq R_0$ and all $f\in C^2(\overline{\B_R^d})$.
\end{lem}
\begin{proof}
This can be shown ad verbum as Lemma \ref{helplemoverrho}.
\end{proof}
Using this lemma, we can establish the following result.
\begin{lem}\label{equivnorms}
The equivalence of norms
\[
\||.|f\|_{H^2(\R^4)}\simeq\|f\|_{H^2(\R^6)},
\]
holds for any $f\in C^2_c(\R^4)$.
Further, also 
\[
\||.|f\|_{H^2(\B_R^4)}\simeq\|f\|_{H^2(\B^6_R)},
\]
holds for $f\in C^2(\overline{\B_R^4})$ and $R\in \left[\frac{1}{2},\frac{3}{2}\right]$.
\end{lem}
\begin{proof}
We start by noting that
\[
\||.|f\|_{L^2(\R^4)}=\|f\|_{L^2(\R^6)}.
\]
Next, we calculate
\begin{align*}
\||.|f\|_{\dot{H}^1(\R^4)}^2=\int_0^\infty |\partial_r(r f(r))|^2 r^3
  dr\leq\int_0^\infty |f'(r)|^2 r^5 + |f(r)|^2r^3 dr\lesssim
\|f\|_{H^1(\R^6)}^2
\end{align*}
due to Lemma \ref{Hardy}.
Analogously, we have that 
\begin{align*}
\||.|f\|_{\dot{H^2}(\R^4)}^2&=\int_0^\infty \left|f''(r) r+ 5f'(r)+\frac{3f(r)}{r}\right|^2 r^3dr
\\
&\lesssim \int_0^\infty |f''(r)|^2 r^5+ 5|f'(r)|^2 r^3+|f(r)|^2 r dr
\lesssim \|f\|_{H^2(\R^6)}^2
\end{align*}
 by Lemmas \ref{helplemoverrho} and \ref{Hardy}. Thus, we have shown 
\[
\||.|f\|_{H^2(\R^4)}\lesssim\|f\|_{H^2(\R^6)}.
\]
To establish the other inequality, we once more employ Lemma \ref{Hardy} to estimate
\begin{align*}
\|f\|_{\dot{H}^1(\R^6)}^2=\int_0^\infty |f'(r)|^2 r^5 dr \lesssim \int_0^\infty |\partial_r(f(r)r)|^2 r^3 + |f(r)|^2 r^3 dr\lesssim \||.|f\|_{H^1(\R^4)}^2.
\end{align*}
Finally, to bound the $\dot{H}^2(\R^6)$ norm of $f$ in terms of $\||.|f\|_{H^2(\R^4)}$, we
first note that an integration by parts shows that 
\begin{align*}
\int_0^\infty |f(r)|^2 r dr\leq \int_0^\infty |f'(r)|^2 r^3 dr=\|f\|_{\dot{H}^1(\R^4)}^2.
\end{align*}
Now, as the Fourier transform of a radial function is again radial, we have that
\begin{align*}
\|f\|_{\dot{H}^1(\R^4)}^2&\simeq\||.|\mathcal{F}(f)\|_{L^2(\R^4)}^2\lesssim \||.|^2\mathcal{F}(f)'\|_{L^2(\R^4)}^2
\\
&\lesssim\||.|^2\mathcal{F}(|.|f)\|_{L^2(\R^4)}^2\simeq \||.|f\|_{\dot{H}^2(\R^4)}^2
\end{align*}
and thus 
\begin{align*}
\int_0^\infty |f(r)|^2 r dr+\int_0^\infty |f'(r)|^2 r^3 dr\lesssim \||.|f\|_{H^2(\R^4)}^2.
\end{align*}
Using this, we obtain that 
\begin{align*}
\|f\|_{\dot{H}^2(\R^6)}^2&=\int_0^\infty \left|f''(r)+\frac{5}{r}f'(r)\right|^2 r^5 dr \lesssim  \||.|f\|_{H^2(\R^4)}^2+\int_0^\infty |f(r)|^2 r +|f'(r)|^2 r^3 dr
\\
&\lesssim \||.|f\|_{H^2(\R^4)}^2
\end{align*}
which establishes the other inequality.
The claim on balls can be derived by the same means combined with an application of Lemma \ref{lem:extension}.
\end{proof}
Next, we will show that by combining Theorems \ref{away from the center} and \ref{stability} we obtain a smooth solution of Eq. \eqref{startingeq2} on $[0,T)\times \R^6$.

\begin{lem}\label{globalsol1}
Let $(f,g)\in C^\infty \times C^\infty(\R^6)$ be smooth initial data that satisfy the assumptions of Theorem \ref{stability} and let $T$ be the associated blowup time. Then Eq.~\eqref{startingeq2} has a unique smooth solution $u:[0,T)\times \R^6 \to \R$, which satisfies $u[0]=(f,g)$.
\end{lem}
\begin{proof}
From Theorems \ref{away from the center} and \ref{stability} we obtain that there exist two solutions corresponding to that initial data, one on
$$\{(t,r)\in [0,\infty)^2: r>t\}$$ and a second one on $\Gamma^T:= \{(t,r)\in [0,T)\times [0,\infty): r\leq T-t \}$ for some $T$ close to $1$ and with $r=|x|$.
As both solutions are smooth, they have to coincide on their overlap. Thus, we obtain a solution on $$\Gamma_T \cup   \{(t,r)\in [0,\infty)^2: r>t\}.$$
Recall now that we defined the domains
\[
\Gamma_{t_0}^{t_1}(r_0):=\{(t,r):t\in[t_0,t_1), r\in[r_0-(t_1-t),r_0+(t_1-t)]\}
\]
and for any such ``lightcone'' we call the set $ \{t_0\}\times  [r_0-(t_1-t_0),r_0+(t_1-t_0)]$ its base.
The idea now is to cover the set $[0,T-\varepsilon]\times \R^6$ for any $\varepsilon>0$ small, with lightcones, such that we can extend our solution from one lightcone to the next. To this end consider $\Gamma_k:=\Gamma_{T-k\frac{\varepsilon}{2}-\varepsilon}^{3T-k\frac{\varepsilon}{2}-\varepsilon}(2T+\frac{\varepsilon}{2})$,  for $k\in \mathbb{N}_0$ and $\varepsilon>0$ small. Then, there exists a minimal $k_\varepsilon \in \mathbb{N}$ such that the base of $\Gamma_{k_\varepsilon}$ is contained in 
$$ \Gamma^T \cup \{(t,r)\in [0,\infty)^2: r>t\}.$$
Moreover, for $0\leq k < k_\varepsilon$ we have that 
the base of $\Gamma_k$ is a subset of
 $$ \Gamma_{k+1} \cup \Gamma^T \cup \{(t,r)\in [0,\infty)^2: r>t\}.$$
Therefore, the local existence theory in lightcones implies that we obtain a smooth solution on $$\Gamma^T\cup  \{(t,r)\in [0,\infty)^2: r>t\} \cup \left(\bigcup_{k=0}^{k_\varepsilon} \Gamma_k\right).$$
 However, by construction we have that $[0,T-\varepsilon]\times\R^6$ is contained in this set, and as $\varepsilon>0$ can be chosen arbitrarily small the claim follows.
\end{proof}
Combining this Lemma with our local wellposed theory in lightcones now yields the desired existence on 
$\Omega_{T}^6:=\left ([0, \infty)\times \mathbb
    R^6\right )\setminus
  \left \{(t,x)\in [T,\infty)\times \mathbb R^6: |x|\leq t-T\right
  \}.$

\begin{prop}\label{globalsol}
Let $(f,g)\in C^\infty\times C^\infty(\R^6)$ be smooth initial data
that satisfy the assumptions of Theorem \ref{stability} and let $T$ be
the associated blowup time. Then Eq.~\eqref{startingeq2} has a unique
smooth solution on $\Omega_T^6$ that satisfies $u[0]=(f,g)$.
\end{prop}
Having constructed solutions on all of $\Omega_T^6$, we now turn our focus back to the original wave maps equation \eqref{wavemaps}. For a smooth solution $u$ of Eq. \eqref{startingeq2} we obtain a solution of Eq. \eqref{wavemaps} by setting 
\[U(t,x)=\begin{pmatrix}
\sin(|x|u(t,x))\frac{x}{|x|}\\
\cos(|x|u(t,x))
\end{pmatrix}
\]
and a Taylor expansion shows that $U$ is indeed smooth.
Recall, that the prescribed initial data were of the form
$\left(U(0,x),\partial_0U(0,x)\right)=(F,G)$ with 
 \begin{align} \label{initialdataform}
  F(x)=
    \begin{pmatrix}
      \sin(|x|f(x))\frac{x}{|x|} \\
      \cos(|x|f(x))
    \end{pmatrix},\qquad
    G(x)=
    \begin{pmatrix}
      \cos(|x|f(x))g(x)x \\
      -\sin(|x|f(x))|x|g(x)
    \end{pmatrix}
\end{align}
  for smooth, radial functions $f,g: \mathbb R^4\to\mathbb R$.
  For notational convenience, we also set $U_*^1[0]=(F_*,G_*)$ and $u_*^1[0]=(f_*,g_*)$,
  and finally, for $\delta > 0$ we define the set 
  \begin{align*}
  \mathcal{F}_{\delta}:=\left\{ f\in
    C^\infty\left(\overline{\B_{1+\delta}^4}\right): |x|f(x)\in
    \left[-\tfrac{3}{2},\tfrac{3}{2}\right] \mbox{ for all } x\in \Bz \right\}
  \end{align*}
\begin{lem} \label{symmetryreduction1}
Let $ \delta_0 >0$  be sufficiently small and for $f\in \mathcal{F}_{\delta_0}$ define
\[F_f(x):=\begin{pmatrix}
\sin(|x|f(x))\frac{x}{|x|}\\
\cos(|x|f(x))
\end{pmatrix}.
\]
Then the estimate
\begin{align*} 
\|f-f_*\|_{H^2(\B^6_{1+\delta})}\lesssim \|F_f-F_*\|_{H^2(\B^4_{1+\delta})}
\end{align*}
holds for all $f\in \mathcal{F}_{\delta_0}$ and all $\delta \in (0,\delta_0)$.
\end{lem} 
\begin{proof}
From $f\in \mathcal{F}_{\delta_0}$, we can infer that $\cos(|x|f(|x|))\geq c$ for all $x\in \B^4_{1+\delta}$
and some constant $c>0$. Furthermore, by choosing $\delta_0$ sufficiently small we also ensure that $f_*\in\mathcal{F}_{\delta_0}$.  Moreover, we have the explicit representation
\begin{align*}
|x|f(x)=\arctan\left(\sum_{i=1}^4 \frac{x_i}{|x|} \frac{F_{f_i}(x)}{F_{f_5}(x)}\right)=\arctan\left(\frac{\sin(|x|f(x))}{\cos(|x|f(x))}\right)
\end{align*}
for any $f\in \mathcal{F}_{\delta_0}$.
Since $ \frac{\sin(|.|f)}{\cos(|.|f)}$ is bounded by assumption and $\arctan \in C^\infty(\R)$, we can immediately conclude that
\begin{align*}
\||.|(f-f_*)\|_{H^2(\Bz)}^2& \lesssim \left\|\frac{\sin(|.|f)}{\cos(|.|f)}-\frac{\sin(|.|f_*)}{\cos(|.|f_*)}\right\|_{H^2(\Bz)}^2
\\
&\lesssim \|\sin(|.|f)-\sin(|.|f_*)\|_{H^2(\Bz)}^2+ \|\cos(|.|f)-\cos(|.|f_*)\|_{H^2(\Bz)}^2.
\end{align*}
By employing Lemma \ref{equivnorms} we can thus wind up this proof by showing that
\begin{align*}
\|\sin(|.|f)-\sin(|.|f_*)\|_{H^2(\Bz)}^2\lesssim \sum_{i=1}^4\left\|\frac{x_i}{|x|}\left[\sin(|x|f(x))-\sin(|x|f_*(x))\right]\right\|_{H^2_x(\Bz)}^2.
\end{align*}
From the fact that
\begin{align*}
\frac{|x|}{|x|}\leq  \sum_{i=1}^4\frac{|x_i|}{|x|}
\end{align*}
and the triangle inequality we can immediately infer that
\begin{align*}
\|\sin(|.|f)-\sin(|.|f_*)\|_{L^2(\Bz)}^2\lesssim \sum_{i=1}^4\left\|\frac{x_i}{|x|}\left[\sin(|x|f(x))-\sin(|x|f_*(x))\right]\right\|_{L^2_x(\Bz)}^2.
\end{align*}
Let now $h$ be the radial representative of $|.|f(.)$ and calculate
\begin{align*}
\partial_j \frac{x_i}{|x|}\sin(h(|x|))=
h'(|x|) \cos(h(|x|))\frac{x_i x_j}{|x|^2}+\sin(h(|x|))\frac{\delta_{i,j}-x_ix_j}{|x|^3}
\end{align*}
for $j=1,2,3,4$ and where $\delta_{i,j}$ denotes the Kronecker delta.
Note that
\begin{align*}
\sum_{i,j=1}^{4}\frac{x_ix_j}{|x|}\left(\frac{\delta_{i,j}}{|x|}-\frac{x_ix_j}{|x|^3}\right)=0.
\end{align*}
and therefore
\begin{align*}
&\quad\sum_{i=1}^4\left\|\frac{x_i}{|x|}\left[\sin(h(|x|))-\sin(h_*(|x|))\right]\right\|_{\dot{H}^1_x(\Bz)}^2
\\
&=
\sum_{i,j=1}^4\left\|\frac{x_ix_j}{|x|}\left[h'(|x|) \cos(h(|x|))-h_*'(|x|)\cos(h_*(|x|))\right]\right\|_{L^2_x(\Bz)}^2
\\
&\quad+\sum_{i,j=1}^4\left\|\frac{\delta_{i,j}-x_ix_j}{|x|^3}\left[\sin(h(|x|))-\sin(h_*(|x|)))\right]\right\|_{L^2(\Bz)}^2
\\
&\geq \|\sin(|.|f)-\sin(|.|f_*)\|_{\dot{H}^1(\Bz)}^2.
\end{align*}
Finally, we calculate
\begin{align*}
\partial_k\partial_j \sin h(|x|)&=h'' (|x|)\cos(h(|x|))\frac{x_k x_j}{|x|^2}-h' (|x|)^2\sin(h(|x|))\frac{x_k x_j}{|x|^2}
\\
&\quad +h' (|x|)\cos(h(|x|))\frac{\delta_{j,k}-x_k x_j}{|x|^3}
\end{align*}
and
\begin{align}\label{Fprimeprime}
\partial_k\partial_j\frac{x_i}{|x|}\sin(h(|x|))&=
h'' (|x|)\cos(h(|x|))\frac{x_k x_j x_i}{|x|^3}-h' (|x|)^2\sin(h(|x|))\frac{x_k x_j x_i}{|x|^3}\nonumber
\\
&\quad+h' (|x|)\cos(h(|x|))\left(\frac{\delta_{i,k}x_j+\delta_{j,k}x_i+\delta_{i,j}x_k}{|x|^2}-3\frac{x_ix_jx_k}{|x|^4}\right)
\\
&\quad+\sin(h(|x|))\left(-\frac{\delta_{j,k} x_i+\delta_{i,k} x_j+x_k\delta_{i,j}}{|x|^3}+3\frac{x_ix_jx_k}{|x|^5}\right) \nonumber
\end{align}
for $j,k=1,2,3,4$.
As before 
\begin{align*}
\sum_{i,j=1}^{4}\frac{x_ix_j}{|x|}\left(\frac{\delta_{i,j}}{|x|}-\frac{x_ix_j}{|x|^3}\right)=0
\end{align*}
 and so
 \begin{align}\label{estimate:H21}
&\quad\| \sin(h)-\sin( h_*)\|_{\dot{H}^2(\Bz)}^2\nonumber
\\
&=\sum_{j,k=1}^{4} \Big\| \frac{x_k x_j}{|x|^2}\big[(h''(|x|)\cos(h(|x|))-h'(|x|)^2\sin(h(|x|))\nonumber
\\
&\quad-h_*''(|x|)\cos(h_*(|x|))+h'(|x|)^2\sin(h(|x|))\big]\Big\|_{L^2_x(\Bz)}^2 \nonumber
\\
&\quad+\left\|\frac{\delta_{j,k}-x_k x_j}{|x|^3}\left[h'(|x|) \cos(h(|x|))-h_*'(|x|)\cos(h(|x|))\right]\right\|_{L^2_x(\Bz)}^2.
 \end{align}
 Moreover, since 
 \begin{align*}
 &\quad\sum_{j,k=1}^{4}\left\|\frac{\delta_{j,k}-x_k x_j}{|x|^3}\left(h'(|x|)\cos(h(|x|))-h_*'(|x|)\cos(h(|x|))\right)\right\|_{L^2_x(\Bz)}^2
 \\
 &\lesssim \left\|\frac{1}{|x|}\left[h'(|x|)\cos(h(|x|))-h_*'(|x|)\cos(h(|x|))\right]\right\|_{L^2_x(\Bz)}^2,
 \end{align*}
 an application of Hardy's inequality yields
 \begin{align}\label{estimate:H22}
 &\quad\left\|\frac{1}{|x|}\left[h'(|x|)\cos(h(|x|))-h_*'(|x|)\cos(h(|x|))\right]\right\|_{L^2_x(\Bz)}^2\nonumber
 \\
 &\lesssim
 \left\|h'(|x|)\cos(h(|x|))-h_*'(|x|)\cos(h(|x|))\right\|_{H^1_x(\Bz)}^2.
 \end{align}
A direct calculation shows that 
\begin{align*}
\sum_{i,j,k=1}^4 \frac{x_ix_jx_k}{|x|^3}\left(\frac{\delta_{i,k}x_j+\delta_{j,k}x_i+\delta_{i,j}x_k}{|x|^2}-3\frac{x_ix_jx_l}{|x|^4}
\right)=0,
\end{align*}
from which we can infer that 
\begin{align}\label{estimate:H33}
&\quad\left\| \frac{x_i}{|x|}\left[\sin( h(|x|))-\sin( h_*(|x|))\right]\right\|_{\dot{H}^2_x(\Bz)}^2\nonumber
\\
&\geq 
\sum_{j,k,i=1}^4\bigg\|\frac{x_k x_j x_i}{|x|^3} \big[h''(|x|)\cos(h(|x|))-h_*''(|x|) \cos(h_*(|x|))\nonumber
\\
&\quad-h'(|x|)^2\sin(h(|x|))+h'(|x|)^2_*\sin(h_*(|x|))\big]\bigg\|_{L^2_x(\Bz)}^2.
\end{align}
Subsequently, combining \eqref{estimate:H21}, \eqref{estimate:H22} and \eqref{estimate:H33}
yields
\begin{align*}
\| \sin(h)-\sin( h_*)\|_{\dot{H}^2(\Bz)}^2 &\lesssim \left\| \frac{x_i}{|x|}\left[\sin( h(|x|))-\sin( h_*(|x|))\right]\right\|_{\dot{H}^2_x(\Bz)}^2
\\
&\quad+ \left\|h'(|x|)\cos(h(|x|))-h_*'(|x|)\cos(h(|x|))\right\|_{H^1_x(\Bz)}^2
\\
&\lesssim \left\| \frac{x_i}{|x|}\left[\sin( h(|x|))-\sin( h_*(|x|))\right]\right\|_{H^2_x(\Bz)}^2.
\end{align*}
\end{proof}
Similarly we also have
\begin{lem} \label{symmetryreduction2}
Let $\delta_0>0$ be sufficiently small and for $f \in \mathcal{F}_{\delta_0}$ and $g \in C^\infty( \overline{\B_{1+\delta_0}^4})$ define 
\[G_{f,g}(x)=
    \begin{pmatrix}
      \cos(|x|f(x))g(x)x \\
      -\sin(|x|f(x))|x|g(x)
    \end{pmatrix}.
  \]
Then the estimate
\begin{align*}
\|g-g_*\|_{H^1(\B_{1+\delta}^6)}\lesssim\left(\|G_{f,g}-G_*\|_{H^1(\Bz)}+\|f-f_*\|_{H^1(\B_{1+\delta}^6)}\right)
\end{align*}
holds for all $ g \in C^\infty( \overline{\B_{1+\delta_0}^4})$, $f \in \mathcal{F}_{\delta_0}$, and $\delta \in [0, \delta_0)$.
\end{lem} \begin{proof}
Let $G_{f,g}=G$. Then, we have the explicit representation 
\begin{align*}
|x|g(x)=\sum_{i=1}^4\frac{x_i}{|x|}\cos(|x|f(x))G_i(x)-\sin(|x|f(x))G_5(x)
\end{align*}
and therefore
\begin{align*}
\||.|(g-g_*)\|_{L^2(\Bz)}^2&\leq \sum_{j=1}^4\|\cos(|x|f)G_{j}(x)-\cos(|x|f_*)G_{*_j}(x)\|_{L^2(\Bz)}^2
\\
&\quad+\|\sin(|x|f)G_5(x)-\sin(|x|f_*)G_{*_5}(x)\|_{L^2(\Bz)}^2
\\
&\lesssim \|f-f_*\|_{L^2(\Bz)}^2+\|G-G_*\|_{L^2(\Bz)}^2.
\end{align*}
Furthermore, using Hardy's inequality one can directly verify the estimate
\begin{align*}
\||.|(g-g_*)\|_{\dot{H}^1(\Bz)}\lesssim\left(\|G-G_*\|_{H^1(\Bz)}+\|f-f_*\|_{H^1(\B_{1+\delta}^6)}\right).
\end{align*}
\end{proof}
Combining these last two Lemmas now yields 
\begin{prop}\label{symmetryreduction}
Let $\delta_0>0$ be sufficiently small and
for $f \in \mathcal{F}_{\delta_0}$ and $g \in C^\infty( \overline{\B_{1+\delta_0}^4})$ define $F_f$ and $G_{f,g}$ as in Lemma \ref{symmetryreduction1} and \ref{symmetryreduction2}, respectively. Then the estimate
\begin{align*}
\|(f,g)-u^1_*[0]\|_{H^2\times H^1(\B^6_{1+\delta})} \lesssim \|(F_f,G_{f,g})-U^1_*[0]\|_{H^2\times H^1(\B^4_{1+\delta})}
\end{align*}
holds for all $f \in \mathcal{F}_{\delta_0}$, $g \in C^\infty( \overline{\B_{1+\delta_0}^4})$ and $\delta \in [0,\delta_0)$.
\end{prop}

Loosely speaking, we have just shown that initial data of Eq.~\eqref{wavemaps} reduces to initial data of Eq. \ref{startingeq2} in a sensible manner and all that is left to do is to go back in the other direction, i.e. to prove the following result.
\begin{lem}\label{goback}
For $u \in C^\infty(\R\times \R^6)$ define
\begin{align*}
U_u(t,.)= 
    \begin{pmatrix}
      \sin(|.|(u(t,.))\frac{.}{|.|} \\
      \cos(|.|u(t,.))
    \end{pmatrix}.
\end{align*}
 Then the estimates
\begin{align*}
\||.|^{-\frac{5}{6}}\left(U_u(t,.)-U_*^T(t,.)\right)\|_{L^{12}(\B^4_{T-t})}
&\lesssim \|u(t,.)-u_*^T(t,.)\|_{L^{12}(\B^6_{T-t})}
\\
\||.|^{-\frac{1}{2}}\partial_j\left(U_u(t,.)-U_*^T(t,.)\right)\|_{L^{4}(\B^4_{T-t})}
&\lesssim \|u(t,.)-u_*^{T}(t,.)\|_{\dot W^{1,4}(\B^6_{T-t})}
\end{align*}
hold for $j=1,2,3,4$, $T \in \left[\frac{1}{2},\frac{3}{2}\right], t\in[0,T)$, and $u\in C^\infty(\R\times \R^6)$.
\end{lem}

Before we can show this, we need one final auxiliary lemma.
\begin{lem}\label{finalhelplem}
The estimate
\begin{align*}
\||.|^{-\frac{1}{2}}(|.|u)'\|_{L^{4}(\B^4_R)}\lesssim \left(\|u'\|_{L^{4}(\B^6_R)}+\|u\|_{L^{12}(\B^6_R)}\right)
\end{align*}
holds true for all $u\in C^\infty(\overline{\B^6_R})$ and all $R>0$.
\end{lem}
\begin{proof}
Let $R>0$ fixed and note that 
\begin{align*}
\||.|^{-\frac{1}{2}}(|.|u)'\|_{L^{4}(\B^4_R)}^4\leq\int_0^R |u'(r)|^4 r^5  +|u(r)|^4 r dr
\end{align*}
and thus we are left with showing
\begin{align*}
\int_0^R |u(r)|^4 r dr \lesssim \int_0^R |u'(r)|^4 r^5 dr+\left(\int_0^R |u(r)|^{12}r^{5} dr\right)^{\frac{1}{3}}.
\end{align*}
An integration by parts combined with Hölder's inequality yields
\begin{align*}
\int_0^R |u(r)|^4 r dr &\lesssim  R^2|u(R)|^4 + \int_0^R |u(r)|^3 |u'(r)| r^2 dr 
\\
&\leq R^2|u(R)|^4 + \left(\int_0^R |u(r)|^4 r dr\right)^{\frac{3}{4}}\left(\int_0^R |u'(r)|^4 r^5 dr\right)^{\frac{1}{4}}.
\end{align*}
From Young's inequality, we can now infer that
\begin{align*}
\int_0^R |u(r)|^4 r dr &\lesssim  R^2|u(R)|^4 + \varepsilon^{\frac{4}{3}} \int_0^R |u(r)|^4 r dr+\frac{1}{\varepsilon^4}\int_0^R |u'(r)|^4 r^5 dr,
\end{align*}
for any $\varepsilon>0$. Ergo, choosing $\varepsilon$ small enough implies that 
\begin{align*}
\int_0^R |u(r)|^4 r dr &\lesssim  R^2|u(R)|^4 +\int_0^R |u'(r)|^4 r^5 dr.
\end{align*}
Finally, we compute
\begin{align*}
|u(R)|&= R^{-5}\left|\int_0^R \partial_rp(u(r)r^5) dr \right|\lesssim R^{-5}\left|\int_0^R u'(r)r^5 dr\right|+R^{-5}\left|\int_0^R u(r)r^4 dr\right|
\\
&\lesssim R^{-5+\frac{3}{4}}\left(\int_0^R |u'(r)|^4 r^{20} dr\right)^{\frac{1}{4}}+R^{-5+\frac{11}{12}}\left(\int_0^R |u(r)|^{12}r^{48} dr\right)^{\frac{1}{12}}
\\
&\leq R^{-\frac{1}{2}}\left(\left(\int_0^R |u'(r)|^4 r^{5} dr\right)^{\frac{1}{4}}+\left(\int_0^R |u(r)|^{12}r^{5} dr\right)^{\frac{1}{12}}\right).
\end{align*}
Hence
\begin{align*}
R^2|u(R)|^4\lesssim \int_0^R |u'(r)|^4 r^{5} dr+\left(\int_0^R |u(r)|^{12}r^{5} dr\right)^{\frac{1}{3}}
\end{align*}
and since all the implicit constants can be chosen independently of $R$ the claim follows.
\end{proof}
\begin{proof}[Proof of Lemma \ref{goback}]
The first estimate follows immediately from the Lipschitz continuity of the trigonometric functions. 
For the second one, it suffices to show that
\begin{align}\label{finalestimate}
\||.|^{-\frac{1}{2}}\partial_j\left(U_u(t,.)-U_*^T(t,.)\right)\|_{L^{4}(\B^4_{T-t})}&\lesssim \|u(t,.)-u^{*}(t,.)\|_{L^{12}(\B^4_{T-t})} \nonumber
\\
&\quad+\||.|^{-\frac{1}{2}}\hu(t,.)-\hu^{T}_*(t,.)\|_{\dot W^{1,4}(\B^4_{T-t})}
\end{align}
where $\widehat{u}=|.|u$ and $\widehat{u}^T_*=|.|u_*^T$,  thanks to Lemma \ref{finalhelplem}. For $i=1,2,3,4$ we calculate
\begin{align*}
\partial_j\left[U_u(t,x)-U_*^T(t,x)\right]_i&=\cos(\widehat{u}(t,|x|))\partial_1\hu(t,|x|)\frac{x_i x_j}{|x|^2}+\sin(\widehat{u}(t,|x|))\frac{\delta_i^j|x|^2-x_ix_j}{|x|^3}
\\
&\quad-\cos(\hu_*^T(t,|x|))\partial_1\hu^{T}_*(t,|x|)\frac{x_i x_j}{|x|^2}-\sin(\hu_*^T(t,|x|))\frac{\delta_i^j|x|^2-x_ix_j}{|x|^3}
\end{align*}
and so
\begin{align*}
&\quad\left\||.|^{-\frac{1}{2}}\partial_j\left[U_u(t,.)-U_*^T(t,.)\right]_i\right\|_{L^4(\B^4_{T-t})}^4
\\
&\lesssim  \int_0^{T-t} \left|\cos(\widehat{u}(t,r))\partial_r\hu(t,r)-\cos(\hu_*^T(t,r))\partial_r\hu^{T}_*(t,r)\right|^4r dr
\\
&\quad+\left|\sin(\widehat{u}(t,r))-\sin(\hu_*^T(t,r))\right|^4 r^{-3} dr.
\end{align*}
From Hölder's inequality, we now immediately infer that
\begin{align*}
\int_0^{T-t}\left|\sin(\widehat{u}(t,r))-\sin(\hu_*^T(t,r))\right|^4 r^{-3} dr
&\lesssim\left(\int_0^{T-t}\left|\sin(\widehat{u}(t,r))-\sin(\hu_*^T(t,r))\right|^{12} r^{-7} dr\right)^{\frac{1}{3}}
\\
&\leq \|u(t,.)-u_*^T(t,.)\|_{L^{12}(\B^6_{T-t})}^4.
\end{align*}
To proceed, we remark that
\begin{align*}
\int_0^{T-t} |\partial_r \hu^{T}_*(t,r)|^6 r^5 dr=(T-t)^6\int_0^1
  \left |(T-t)^{-1}\partial_\rho \widehat u^T(t,(T-t)\rho)\right |^6\rho^5d\rho\lesssim 1
\end{align*}
for all $t\in [0,T)$.
With this in mind, we estimate 
\begin{align*}
&\quad\int_0^{T-t} \left|\cos(\widehat{u}(t,r))\partial_r\hu(t,r)-\cos(\hu_*^T(t,r))\partial_r\hu_*^{T}(t,r)\right|^4r dr 
\\ &\leq\int_0^{T-t} |\cos(\hu_*(t,r))|^4 \left|\partial_r \hu(t,r)-\partial_r\hu_*^{T}(t,r)\right|^4r dr
\\
&\quad+\int_0^{T-t} |\partial_r \hu^{T}_*(t,r)|^4 \left|\cos(\hu(t,r))-\cos(\hu_*^T(t,r))\right|^4r dr
\\
&\leq\||.|^{-\frac{1}{2}}\hu(t,.)-\hu^{T}_*(t,.)\|_{\dot W^{1,4}(\B^4_{T-t})}^4
+\int_0^{T-t} |\partial_r \hu^{T}_*(t,r)|^4 \left|\cos(\hu(t,r))-\cos(\hu_*^T(t,r))\right|^4r dr.
\end{align*}
One last application of Hölder's inequality then yields
\begin{align*}
&\quad\int_0^{T-t} |\partial_r \hu^{T}_*(t,r)|^4 \left|\cos(\hu(t,r))-\cos(\hu_*^T(t,r))\right|^4r dr 
\\
&\lesssim 
\left(\int_0^{T-t} |\partial_r \hu^{T}_*(t,r)|^6 r^5\right)^{\frac{2}{3}} \left(\int_0^{T-t}|\cos(\hu(t,r))-\cos(\hu ^ *(t,r))|^{12} r^{-7} dr\right)^{\frac{1}{3}} 
\\
&\lesssim \||.|^{-\frac{5}{6}}[\hu(t,.)-\hu_*^T(t,.)]\|_{L^{12}(\B^4_{T-t})}^{4}=\|u(t,.)-u_*^T(t,.)\|_{L^{12}(\B^6_{T-t})}^{4}.
\end{align*}
Since $ \||.|^{-\frac{5}{6}}\partial_j\left[U_u(t,.)-U_*^T(t,.)\right]_5\|_{L^{4}(\B^4_{T-t})}$ can be bounded in the same manner we can conclude this proof.
\end{proof}
\begin{proof}[Proof of Theorem \ref{maintheorem}]
Let $(F,G)$ be as in the assumptions of Theorem \ref{maintheorem}. Proposition \ref{symmetryreduction} tells us that 
\begin{align*}
\|(f,g)-u^1_*[0]\|_{H^2 \times H^1(\B^6_{1+\delta})}^2\lesssim \|(F,G) -U^1_*[0]\|_{H^2 \times H^1(\B^4_{1+\delta})}^2
\end{align*}
and hence, if $(F,G)$ are chosen close enough to $U^1_*[0]$, then $(f,g)$ will satisfy the assumptions of Theorem \ref{stability}.
Therefore, we can conclude by invoking Proposition \ref{globalsol} and Lemma \ref{goback}.
\end{proof}
\bibliography{references}
\bibliographystyle{plain}

\end{document}